\documentclass[12pt,letter]{amsart}

\usepackage{amssymb,latexsym, amsmath, amsxtra, bbm}
\usepackage[dvips]{graphics}
\usepackage{xypic}
\usepackage{verbatim}
\usepackage[abs]{overpic}
\usepackage{hyperref}

\allowdisplaybreaks[1]

\theoremstyle{plain}
        \newtheorem{theorem}{Theorem}[section]
        \newtheorem*{theorem*}{Theorem}
        \newtheorem*{conj*}{Conjecture}
        \newtheorem{lemma}[theorem]{Lemma}
        \newtheorem{prop}[theorem]{Proposition}
        \newtheorem{cor}[theorem]{Corollary}

\theoremstyle{definition}
        \newtheorem{definition}[theorem]{Definition}
        \newtheorem{rem}[theorem]{Remark}
         \newtheorem{rems}[theorem]{Remarks}
         \newtheorem{ex}[theorem]{Example}
         \newtheorem*{assumptions}{The Assumptions}

\theoremstyle{remark}
        \newtheorem*{remark}{Remark}

\numberwithin{equation}{section}
\numberwithin{theorem}{section}
\numberwithin{table}{section}
\numberwithin{figure}{section}

\providecommand{\defn}[1]{\emph{#1}}

\newcommand{\diam}  {\operatorname{diam}}

\newcommand{\inte}  {\operatorname{inte}}
\newcommand{\inter}  {\operatorname{int}}

\newcommand{\id} {\operatorname{id}}

\newcommand{\card} {\operatorname{card}}
\newcommand{\supp}{\operatorname{supp}}
\newcommand{\R}{\mathbb{R}}
\newcommand{\B}{\mathbb{B}}    
\newcommand{\C}{\mathbb{C}}      

\newcommand{\N}{\mathbb{N}}      
\newcommand{\Z}{\mathbb{Z}}      

\providecommand{\abs}[1]{\lvert#1\rvert}
\providecommand{\Abs}[1]{\left|#1\right|}

\providecommand{\Norm}[1]{\left\|#1\right\|}
\renewcommand{\:}{\colon}

\def\({\left(}
\def\){\right)}
\def\[{\left[}
\def\]{\right]}

\newcommand{\crit}{\operatorname{crit}}
\newcommand{\post}{\operatorname{post}}

\newcommand{\CC}{\mathcal{C}}

 \newcommand{\DD}{\mathbf{D}}

\newcommand{\PP}{\mathbf{P}}

\newcommand{\X} {\mathbf{X}}

\newcommand{\E} {\mathbf{E}}

\newcommand{\V} {\mathbf{V}}

\newcommand{\CCC}{C}

\newcommand{\PPP}{\mathcal{P}}

\newcommand{\MMM}{\mathcal{M}}

\newcommand{\BB}{\mathcal{B}}

\newcommand{\Holder}[1] {\CCC^{0,#1}}

\newcommand{\Hseminorm}[2] {\Abs{#2}_{#1}}

\newcommand{\Hnorm}[2] {\Norm{#2}_{\Holder{#1}}}

\newcommand{\RR}{\mathcal{L}}

\begin{document}
\title[Equilibrium States for Expanding Thurston Maps]{Equilibrium States for Expanding Thurston Maps}
\author{Zhiqiang~Li}
\address{Department of Mathematics, UCLA, Los Angeles CA 90095-1555}
\thanks{The author was partially supported by NSF grant DMS-1162471.}
\email{lizq@math.ucla.edu}




\begin{abstract}
In this paper, we use the thermodynamical formalism to show that there exists a unique equilibrium state $\mu_\phi$ for each expanding Thurston map $f\: S^2\rightarrow S^2$ together with a real-valued H\"{o}lder continuous potential $\phi$. Here the sphere $S^2$ is equipped with a natural metric induced by $f$, called a visual metric. We also prove that identical equilibrium states correspond to potentials which are co-homologous upto a constant, and that the measure-preserving transformation $f$ of the probability space $(S^2,\mu_\phi)$ is exact, and in particular, mixing and ergodic. Moreover, we establish versions of equidistribution of preimages under iterates of $f$, and a version of equidistribution of a random backward orbit, with respect to the equilibrium state. As a consequence, all the above results hold for a postcritically-finite rational map with no periodic critical points on the Riemann sphere equipped with the chordal metric.
\end{abstract}

\maketitle

\tableofcontents

\section{Introduction}
Ergodic theory has been an important tool in the study of dynamical systems. The investigation of the existence and uniqueness of invariant measures and their properties has been a central part of ergodic theory. The realization of the connection between the orbit structure and the existence of a finite invariant measure can be traced back to H.~Poincar\'e.

However, a dynamical system may possess a large class of invariant measures, some of which may be more interesting than others. It is therefore crucial to examine the relevant invariant measures.

The \emph{thermodynamical formalism}  is one such mechanism to produce invariant measures with some nice properties under assumptions on the regularity of their \emph{Jacobian functions}. More precisely, for a continuous transformation on a compact metric space, we can consider the \emph{topological pressure} as a weighted version of the \emph{topological entropy}, with the weight induced by a real-valued continuous function, called \emph{potential}. The Variational Principle identifies the topological pressure with the supremum of its measure-theoretic counterpart, the \emph{measure-theoretic pressure}, over all invariant Borel probability measures \cite{Bo75, Wa76}. Under additional regularity assumptions on the transformation and the potential, one gets existence and uniqueness of an invariant Borel probability measure maximizing the measure-theoretic pressure, called the \emph{equilibrium state} for the given transformation and the potential. Often the Jacobian function for the transformation with respect to the equilibrium state is prescribed by a function induced by the potential. The study of the existence and uniqueness of the equilibrium states and their various properties such as ergodic properties, equidistribution, fractal dimensions, etc., has been the main motivation for much research in the area.

This theory, as a successful approach to choosing relevant invariant measures, was inspired by statistical mechanics, and created by D.~Ruelle, Ya.~Sinai, and others in the early seventies \cite{Do68, Si72, Bo75, Wa82}. Since then, the thermodynamical formalism has been applied in many classical contexts (see for example, \cite{Bo75, Ru89, Pr90, KH95, Zi96, MauU03, BS03, Ol03, Yu03, PU10, MayU10}). However, beyond several classical dynamical systems, even the existence of equilibrium states is largely unknown, and for those dynamical systems that do possess equilibrium states, often the uniqueness is unknown or at least requires additional conditions. The investigation of different dynamical systems from this perspective has been an active area of current research.

In this paper, we apply the theory of thermodynamical formalism to study the equilibrium states for a class of non-uniformly expanding dynamical systems that are not among the classical dynamical systems studied in the works mentioned above, namely, the class of expanding Thurston maps on the sphere $S^2$. Thurston maps are branched covering maps on the sphere $S^2$ that generalize rational maps with finitely many postcritical points on the Riemann sphere. More precisely, a (non-homeomorphic) branched covering map $f\: S^2 \rightarrow S^2$ is a Thurston map if it has finitely many critical points each of which is preperiodic. These maps arose in W.~P.~Thurston's characterization of postcritically-finite rational maps (see [DH93]). For a more detailed introduction to Thurston maps, see Section~\ref{sctThurstonMap}.

In order to obtain the existence and uniqueness of \emph{the measure of maximal entropy} (i.e., the equilibrium state for the constant potential 0) for a Thurston map, some condition of expansion had to be imposed. P.~Ha\"{i}ssinsky and K.~Pilgrim introduced such a notion for any finite branched coverings between two topological spaces that are Hausdorff, locally compact, and locally connected (see \cite[Section~2.1 and Section~2.2]{HP09}). M.~Bonk and D.~Meyer formulated \cite{BM10} an equivalent definition of expansion in the context of Thurston maps. We will discuss the precise definition in Section~\ref{sctThurstonMap}. We call Thurston maps with such an expansion property \emph{expanding Thurston maps}. We refer to \cite[Proposition~8.2]{BM10} for a list of equivalent definitions.

As a consequence of their general results in \cite{HP09}, P.~Ha\"issinsky and K.~Pilgrim proved that for each expanding Thurston map, there exists a measure of maximal entropy and that the measure of maximal entropy is unique for expanding Thurston maps without periodic critical points. M.~Bonk and D.~Meyer, on the other hand, proved the existence and uniqueness of the measure of maximal entropy for all expanding Thurston maps in \cite{BM10} using an explicit combinatorial construction.

Actually the notion of expansion on a Thurston map $f\: S^2\rightarrow S^2$ is sufficient for us to establish the existence and uniqueness of the equilibrium state, denoted by $\mu_\phi$, for a H\"{o}lder continuous potential $\phi\: S^2\rightarrow \R$. Here the sphere $S^2$ is equipped with a natural metric called a \emph{visual metric} (see Lemma~\ref{lmCellBoundsBM} and the preceding discussion). This generalizes the existence and uniqueness of the measure of maximal entropy of an expanding Thurston map in \cite{HP09} and \cite{BM10}. We also prove that the measure-preserving transformation $f$ of the probability space $(S^2,\mu_\phi)$ is \emph{exact} (see Definition~\ref{defExact}), and in particular, mixing and ergodic (Theorem~\ref{thmFisExact} and Corollary~\ref{corMixing}). This generalizes the corresponding results in \cite{BM10} and \cite{HP09} for the measure of maximal entropy to our context.

In \cite{HP09} and \cite{Li13}, various versions of equidistribution of preimage, periodic, and preperiodic points of an expanding Thurston map with respect to the measure of maximal entropy were established. In Section~\ref{sctEquidistribution}, we prove some versions of equidistribution with respect to the equilibrium state in our context. These results generalize the corresponding equidistribution results in \cite{HP09} and \cite{Li13}.

In this paper, we use the framework set by M.~Bonk and D.~Meyer in \cite{BM10} to study expanding Thurston maps, but our approach to the investigation of equilibrium states is different from the treatment of measures of maximal entropy in \cite{BM10} and \cite{HP09}. We use the thermodynamical formalism to establish the existence and uniqueness of the equilibrium states and their various properties.

In order to state our results more precisely, we quickly review some key concepts.

For an expanding Thurston map $f\: S^2\rightarrow S^2$ and a continuous function $\psi\: S^2\rightarrow \R$, each $f$-invariant Borel probability measure $\mu$ on $S^2$ corresponds to a quantity
$$
P_\mu(f,\psi) = h_\mu (f) + \int \! \psi \,\mathrm{d} \mu
$$ 
called the \emph{measure-theoretic pressure} of $f$ for $\mu$ and $\psi$, where $h_\mu(f)$ is the measure-theoretic entropy of $f$ for $\mu$. The well-known Variational Principle (see for example, \cite[Theorem~3.4.1]{PU10}) asserts that
\begin{equation}  \label{eqIntroVP}
P(f,\psi) = \sup P_\mu(f,\psi),
\end{equation}
where the supremum is taken over all $f$-invariant Borel probability measures $\mu$, and $P(f,\psi)$ is the \emph{topological pressure} of $f$ with respect to $\psi$ defined in (\ref{defTopPressure}). A measure $\mu$ that attains the supremum in (\ref{eqIntroVP}) is called an \emph{equilibrium state} for $f$ and $\psi$.

We assume for now that $\psi$ is H\"{o}lder continuous (with respect to a given visual metric for $f$ on $S^2$). One characterization of the topological pressure in our context is given by the following formula (Proposition~\ref{propTopPressureDefPreImg}):
\begin{equation}   \label{eqIntroTopPressurePreImgDef}
P(f,\psi)= \lim\limits_{n\to +\infty} \frac{1}{n} \log \sum\limits_{y\in f^{-n}(x)} \deg_{f^n}(y) \exp (S_n\psi(y)),
\end{equation}
for each $x\in S^2$, independent of $x$, where $\deg_{f^n}(y)$ is the local degree of $f^n$ at $y$ and $S_n\psi(y)=\sum\limits_{i=0}^{n-1} \psi(f^i(y))$. 

An important tool we use to find the equilibrium state and to establish its uniqueness, is the \emph{Ruelle operator} $\RR_{\psi}$ on the Banach space $\CCC(S^2)$ of real-valued continuous functions on $S^2$, given by 
$$
\RR_\psi(u)(x) = \sum\limits_{y\in f^{-1}(x)} \deg_f(y)u(y)\exp(\psi(y)),
$$
for $u\in \CCC(S^2)$ and $x\in S^2$. 

The Ruelle operator plays a central role in the thermodynamical formalism, and has been studied carefully for various dynamical systems (see for example, \cite{Bo75, Ru89, Pr90, Zi96, MauU03, PU10, MayU10}). Some of the ideas that we apply in this paper for its investigation are well-known and repeatedly used in the literature, see for example \cite{PU10, Zi96}. 

A main difficulty of our analysis comes from the lack of uniform expansion property that arises from the existence of critical points (i.e., branch points of a branched covering map). As an example, identities of the form (\ref{eqIntroTopPressurePreImgDef}) that are usually easy to derive for classical dynamical systems (see for example, \cite[Proposition~4.4.3]{PU10}) become difficult to verify directly in our context.

We remark on the subtlety of our notion of expansion by pointing out that each expanding Thurston map without periodic critical points is \emph{asymptotically $h$-expansive}, but not \emph{$h$-expansive}; on the other hand, expanding Thurston maps with at least one periodic critical point are not even asymptotically $h$-expansive \cite{Li14}. Asymptotic $h$-expansiveness and $h$-expansiveness are two notions of weak expansion introduced by M.~Misiurewicz \cite{Mi73} and R.~Bowen \cite{Bo72}, respectively. Note that forward-expansiveness implies $h$-expansiveness, which in turn implies asymptotic $h$-expansiveness \cite{Mi76}. Both conditions guarantee that the measure-theoretic entropy as a function on the space of invariant Borel probability measures (equipped with the weak$^*$ topology) is upper semi-continuous \cite{Mi76}. We do not use the last fact in this paper, but we remark here that the upper semi-continuity of the measure-theoretic entropy implies the existence of at least one equilibrium state for a general real-valued continuous potential. So we can get a stronger existence result for equilibrium states for expanding Thurston maps without periodic critical points than that in the Main Theorem below (see \cite[Theorem~1.3]{Li14}).

\smallskip

The following statement summarizes the main results of this paper.

\begin{theorem}[Main Theorem]    \label{thmMain}
Let $f\:S^2 \rightarrow S^2$ be an expanding Thurston map and $d$ be a visual metric on $S^2$ for $f$. Let $\phi$ be a real-valued H\"{o}lder continuous function on $S^2$ with respect to the metric $d$. 

Then there exists a unique equilibrium state $\mu_\phi$ for the map $f$ and the potential $\phi$. If $\psi$ is another real-valued H\"{o}lder continuous function on $S^2$ with respect to the metric $d$, then $\mu_\phi=\mu_\psi$ if and only if there exists a constant $K\in\R$ such that $\phi - \psi$ and $K \mathbbm{1}_{S^2}$ are co-homologous in the space of real-valued continuous functions on $S^2$, i.e., $\phi-\psi - K \mathbbm{1}_{S^2} = u\circ f - u$ for some real-valued continuous function $u$ on $S^2$.

Moreover, $\mu_\phi$ is a non-atomic $f$-invariant Borel probability measure on $S^2$ and the measure-preserving transformation $f$ of the probability space $(S^2,\mu_\phi)$ is forward quasi-invariant, nonsingular, exact, and in particular, mixing and ergodic.

In addition, the preimages points of $f$ are equidistributed with respect to $\mu_\phi$, i.e., for each sequence $\{x_n\}_{n\in\N}$ of points in $S^2$, as $n\longrightarrow +\infty$,
\begin{equation}    \label{eqMainThm1}
\frac{1}{Z_n(\phi)} \sum\limits_{y\in f^{-n}(x_n)} \deg_{f^n} (y) \exp\(S_n\phi(y)\) \frac{1}{n} \sum\limits_{i=0}^{n-1} \delta_{f^i(y)} \stackrel{w^*}{\longrightarrow} \mu_\phi,
\end{equation}
\begin{equation}   \label{eqMainThm2}
\frac{1}{Z_n\big(\widetilde\phi \hspace{0.5mm} \big)} \sum\limits_{y\in f^{-n}(x_n)} \deg_{f^n} (y) \exp\big(S_n\widetilde\phi(y)\big) \delta_y \stackrel{w^*}{\longrightarrow} \mu_\phi,
\end{equation} 
where $Z_n(\psi) = \sum\limits_{y\in f^{-n}(x_n)} \deg_{f^n} (y) \exp\(S_n\psi(y)\)$, for each $n\in\N$ and each $\psi\in\CCC(S^2)$.
\end{theorem}

Here the symbol $w^*$ indicates the convergence is in the weak$^*$ topology, $\deg_{f^n}(x)$ denotes the local degree of the map $f^n$ at $x$, $S_n\psi(y)= \sum\limits_{i=0}^{n-1} \psi(f^i(y))$, and $\widetilde\phi$ is a potential related to $\phi$ defined in (\ref{eqDefPhiTilde}).

The Main Theorem above combines Theorem~\ref{thmUniqueES}, Theorem~\ref{thmFisExact}, Corollary~\ref{corNonAtomic}, Corollary~\ref{corMixing}, Theorem~\ref{thmCohomologous}, and Proposition~\ref{propWeakConvPreImgWithWeight}.

As a quick consequence of the proof of the uniqueness of the equilibrium state, we show in Proposition~\ref{propWeakConvR*^nProbToMu} that under the assumptions in the Main Theorem, the images of each Borel probability measure $\mu$ under iterates of the adjoint of the Ruelle operator $\RR_{\widetilde\phi}$ converge, in the weak$^*$ topology to the unique equilibrium state $\mu_\phi$, i.e.,
\begin{equation}   \label{eqIntroWeakConvR*^nProbToMu}
\big(\RR^*_{\widetilde\phi} \big)^n (\mu)  \stackrel{w^*}{\longrightarrow} \mu_\phi \text{ as } n\longrightarrow +\infty.
\end{equation}

A rational Thurston map is expanding if and only if it has no periodic critical points (see \cite[Proposition~19.1]{BM10}). So when we restrict to rational Thurston maps, we get the following corollary as an immediate consequence of Theorem~\ref{thmMain} and Remark~\ref{rmChordalVisualQSEquiv}.

\begin{cor}   \label{corMainRational}
Let $f$ be a postcritically-finite rational map on the Riemann sphere $\widehat\C$ with no periodic critical points and with degree at least $2$. Let $\phi$ be a real-valued H\"{o}lder continuous function on $\widehat\C$ equipped with the chordal metric. 

Then there exists a unique equilibrium state $\mu_\phi$ for the map $f$ and the potential $\phi$. If $\psi$ is another real-valued H\"{o}lder continuous function on $\widehat\C$, then $\mu_\phi=\mu_\psi$ if and only if there exists a constant $K\in\R$ such that $\phi - \psi$ and $K \mathbbm{1}_{\widehat\C}$ are co-homologous in the space of real-valued continuous functions on $\widehat\C$, i.e., $\phi-\psi - K \mathbbm{1}_{\widehat\C} = u\circ f - u$ for some real-valued continuous function $u$ on $\widehat\C$.

Moreover, $\mu_\phi$ is an non-atomic $f$-invariant Borel probability measure on $S^2$ and the measure-preserving transformation $f$ of the probability space $(S^2,\mu_\phi)$ is forward quasi-invariant, nonsingular, exact, and in particular, mixing and ergodic.

In addition, both (\ref{eqMainThm1}) and (\ref{eqMainThm2}) hold as $n\longrightarrow +\infty$.
\end{cor}

The postcritically-finite rational maps (with degree at least $2$) is another name for the rational Thurston maps, used by many authors in holomorphic dynamics.

The existence and uniqueness of the equilibrium state for a general rational map $R$ on the Riemann sphere and a real-valued H\"{o}lder continuous potential $\phi$ can be established under the additional assumption that $\sup \{ \phi(z) \,|\, z\in J(R) \}  < P(R,\phi)\}$, where $J(R)$ is the Julia set of $R$ and $P(R,\phi)$ is the topological pressure of $R$ with respect to $\phi$ (see \cite{DU91,Pr90,DPU96}). This assumption can sometimes be dropped: one can either restrict to certain subclasses of rational maps, such as topological Collet-Eckmann maps, see \cite{CRL11}, or hyperbolic rational maps (more generally, distance-expanding maps), see \cite{PU10}; or one can impose other conditions on the function $\phi$, such as hyperbolicity of $\phi$, see \cite{IRRL12}. It is easy to check that a rational expanding Thurston map is topological Collet-Eckmann.

\smallskip

We will now give a brief description of the structure of this paper.

After fixing some notation in Section~\ref{sctNotation}, we review Thurston maps in Section~\ref{sctThurstonMap}. A few key concepts and results from \cite{BM10}, which are going to be used in this paper, are recorded or generalized. In Lemma~\ref{lmMetricDistortion}, we prove that an expanding Thurston map locally expands the distance, with respect to a visual metric, between two points exponentially as long as they belong to one set in some particular partition of $S^2$ induced by a backward iteration of some Jordan curve on $S^2$. This observation, generalizing a result of M.~Bonk and D.~Meyer \cite[Lemma~16.1]{BM10}, enables us to establish the distortion lemmas (Lemma~\ref{lmSnPhiBound} and Lemma~\ref{lmSigmaExpSnPhiBound}) in Section~\ref{sctExistence}, which serve as cornerstones for the mechanism of thermodynamical formalism.

In Section~\ref{sctAssumptions}, we state the assumptions on some of the objects in this paper, which we are going to repeatedly refer to later as \emph{the Assumptions}.

In Section~\ref{sctExistence}, following the ideas from \cite{PU10} and \cite{Zi96}, we use the thermodynamical formalism to prove the existence of the equilibrium states for expanding Thurston maps and real-valued H\"{o}lder continuous potentials. We first recall briefly some concepts from dynamical systems, such as the measure-theoretic pressure, the topological pressure, the equilibrium state, and others. We then establish two distortion lemmas (Lemma~\ref{lmSnPhiBound} and Lemma~\ref{lmSigmaExpSnPhiBound}), which will be used frequently throughout this paper. Next, we define \emph{Gibbs states} and \emph{radial Gibbs states}. Later in Proposition~\ref{propRadialGibbsIFFGibbs}, we prove that for an expanding Thurston map the notion of a Gibbs state is equivalent to that of a radial Gibbs state if and only if the map does not have periodic critical points. 

We then introduce the Ruelle operator $\RR_\psi$ on the Banach space $\CCC(S^2)$ of real-valued continuous functions on $S^2$ for $\psi\in\CCC(S^2)$, which is the main tool for our investigation. By applying the Schauder-Tikhonov Fixed Point Theorem, we establish in Theorem~\ref{thmMexists} the existence of an eigenmeasure $m_\phi$ of the adjoint $\RR^*_\phi$ of the Ruelle operator $\RR_\phi$, for a real-valued H\"{o}lder continuous potential $\phi$. We also show in Theorem~\ref{thmMexists} that the Jacobian function $J$ for $f$ with respect to $m_\phi$ is
$$
J = c\exp(-\phi),
$$
where $c$ is the eigenvalue corresponding to $m_\phi$, which is proved to be equal to $\exp(P(f,\phi))$ later in Proposition~\ref{propTopPressureDefPreImg}. We establish in Proposition~\ref{propMisGibbsState} that $m_\phi$ is a Gibbs state. The measure $m_\phi$ may not be $f$-invariant. In Theorem~\ref{thmMuExist}, we adjust the potential $\phi$ to get a new potential $\overline\phi$ such that there exists an eigenfunction $u_\phi$ of $\RR_{\overline\phi}$ with eigenvalue $1$. The positive function $u_\phi$ constructed as the uniform limit of the sequence
$$
\Bigg\{\frac{1}{n}\sum\limits_{j=0}^{n-1} \RR_{\overline{\phi}}^j(\mathbbm{1})\Bigg\}_{n\in\N}
$$
is shown to be bounded away from $0$ and $+\infty$, and H\"{o}lder continuous with the same exponent as that of $\phi$. Then we demonstrate that the measure $\mu_\phi = u_\phi m_\phi$ is an $f$-invariant Gibbs state. Finally, by combining Proposition~\ref{propInvGibbsIsEqlbStatePLessThanPressure} and Proposition~\ref{propTopPressureDefPreImg}, we prove in Corollary~\ref{corExistES} that $\mu_\phi$ is an equilibrium state for $f$ and $\phi$.

In Section~\ref{sctUniqueness}, we establish the uniqueness of the equilibrium state for an expanding Thurston map $f$ and a real-valued H\"{o}lder continuous potential $\phi$. We use the idea in \cite{PU10} to apply the G\^ateaux differentiability of the topological pressure function and some techniques from functional analysis. More precisely, a general fact from functional analysis (recorded in Theorem~\ref{thmUniqueTangent}) states that for an arbitrary convex continuous function $Q\: V\rightarrow \R$ on a separable Banach space $V$, there exists a unique continuous linear functional $L\:V\rightarrow\R$ \emph{tangent to $Q$ at $x\in V$} if and only if the function $t\longmapsto Q(x+ty)$ is differentiable at $0$ for all $y$ in a subset $U$ of $V$ that is dense in the weak topology on $V$. One then observes that for each continuous map $g\: X\rightarrow X$ on a compact metric space $X$, the topological pressure function $P(g,\cdot) \: \CCC(X)\rightarrow \R$ is continuous and convex (see \cite[Theorem~3.6.1 and Theorem~3.6.2]{PU10}), and if $\mu$ is an equilibrium state for $g$ and $\psi\in\CCC(X)$, then the continuous linear functional $u \longmapsto \int \! u\,\mathrm{d}\mu$, for $u\in\CCC(X)$, is tangent to $P(g,\cdot)$ at $\psi$ (see \cite[Proposition~3.6.6]{PU10}). So in order to verify the uniqueness of the equilibrium state for an expanding Thurston map $f$ and a real-valued H\"{o}lder continuous potential $\phi$, it suffices to prove that the function $t\longmapsto P(f,\phi + t\gamma)$ is differentiable at $0$, for all $\gamma$ in a suitable subspace of $\CCC(S^2)$. This is established in Theorem~\ref{thmPisDiff}.

Following the procedures in \cite{PU10} to prove Theorem~\ref{thmPisDiff}, we introduce a new potential $\widetilde\phi$ induced by $\phi$, and establish some uniform bounds in Theorem~\ref{thmChbInChb} and Lemma~\ref{lmRtildeSupBound}, which are then used to show uniform convergence results in Theorem~\ref{thmRR^nConv} and Lemma~\ref{lmDerivConv}. In some sense, Theorem~\ref{thmChbInChb} gives a quantitative form of the fact that $\RR_{\widetilde\phi}$ is \emph{almost periodic} (see Corollary~\ref{corRAlmostPeriodic}), and Theorem~\ref{thmRR^nConv} exhibits a uniform version of the contracting behavior of $\RR_{\widetilde\phi}$ on a codimension $1$ subspace of $\CCC(S^2)$. As a by-product, we demonstrate in Corollary~\ref{corMandMuUnique} that for each expanding Thurston map $f$ and each real-valued H\"{o}lder continuous potential $\phi$, the operator $\RR^*_\phi$ has a unique eigenmeasure $m_\phi$. Moreover, the measure $\mu_\phi$ is the unique eigenmeasure $m_{\widetilde\phi}$ of $\RR^*_{\widetilde\phi}$ with the corresponding eigenvalue $1$. Another consequence is Proposition~\ref{propWeakConvR*^nProbToMu} which implies (\ref{eqIntroWeakConvR*^nProbToMu}) mentioned earlier.

In Section~\ref{sctErgodicity}, we prove that the measure-preserving transformation $f$ of the probability space $(S^2,\mu_\phi)$ is exact (Theorem~\ref{thmFisExact}), where the equilibrium state $\mu_\phi$ is non-atomic (Corollary~\ref{corNonAtomic}). It follows in particular that the transformation $f$ is mixing and ergodic (Corollary~\ref{corMixing}). To establish these results, we first show in Proposition~\ref{propEdgeIs0set} that
\begin{equation*}
m_\phi \(\bigcup_{i=0}^{+\infty} f^{-i}(\CC)\) = \mu_\phi \(\bigcup_{i=0}^{+\infty} f^{-i}(\CC)\)  = 0.
\end{equation*}
for each Jordan curve $\CC\subseteq S^2$ containing the postcritical points of $f$ that satisfies $f^l(\CC)\subseteq\CC$ for some $l\in\N$. This proposition is also used in the proof of Theorem~\ref{thmCohomologous}.

Theorem~\ref{thmCohomologous}, the main result of Section~\ref{sctCohomologous}, asserts that if $\phi$ and $\psi$ are two real-valued H\"older continuous functions with the corresponding equilibrium states $\mu_\phi$ and $\mu_\psi$, respectively, then $\mu_\phi=\mu_\psi$ if and only if there exists a constant $K\in\R$ such that $\phi - \psi$ and $K \mathbbm{1}_{S^2}$ are \emph{co-homologous} in the space $\CCC(S^2)$ of real-valued continuous functions, i.e., $\phi-\psi - K \mathbbm{1}_{S^2} = u\circ f - u$ for some $u\in\CCC(S^2)$. For Theorem~\ref{thmCohomologous}, we first formulate a form of the \emph{closing lemma} for expanding Thurston maps (Lemma~\ref{lmClosingLemma}). For such maps, we then include in Lemma~\ref{lmwLimitPt} a direct proof of the existence of a point whose forward orbit is dense in $S^2$. Finally, we give the proof of Theorem~\ref{thmCohomologous} at the end of the section.

In Section~\ref{sctEquidistribution}, we first establish in Proposition~\ref{propWeakConvPreImgWithWeight} versions of equidistribution of preimages with respect to the equilibrium state, using results we obtain in Section~\ref{sctUniqueness}. These results partially generalize Theorem~1.2 in \cite{Li13} where we treated the case for the measure of maximal entropy. At the end of this paper, we include in Theorem~\ref{thmASConvToES} a generalization of Theorem~7.1 in \cite{Li13}, following the idea of J.~Hawkins and M.~Taylor \cite{HT03}. Theorem~\ref{thmASConvToES} states that the equilibrium state $\mu_\phi$ from the Main Theorem above is almost surely the limit of
$$
\frac{1}{n} \sum\limits_{i=0}^{n-1} \delta_{q_i}
$$
as $n\longrightarrow +\infty$ in the weak$^*$ topology, where $q_0$ is an arbitrary fixed point in $S^2$, and for each $i\in\N_0$, the point $q_{i+1}$ is randomly chosen from the set $f^{-1}(q_i)$ with the probability of each $x\in f^{-1}(q_i)$ being $q_{i+1}$ conditional on $q_i$ proportional to the local degree of $f$ at $x$ times $\exp\big( \widetilde\phi(x) \big)$. This theorem is an immediate consequence of a theorem of H.~Furstenberg and Y.~Kifer in \cite{FK83} and the fact that the equilibrium state is the unique Borel probability measure invariant under the adjoint of the Ruelle operator $\RR_{\widetilde\phi}$ (Corollary~\ref{corMandMuUnique}). A similar result for certain hyperbolic rational maps on the Riemann sphere and the measures of maximal entropy was proved by M.~Barnsley \cite{Ba88}. J.~Hawkins and M.~Taylor generalized it to any rational map on the Riemann sphere of degree $d\geq 2$ \cite{HT03}.

\bigskip
\noindent
\textbf{Acknowledgments.} The author wants to express his gratitude to M.~Bonk for his introduction to this subject of expanding Thurston maps and his patient teaching and guidance as the advisor of the author. The author also wants to thank P.~Ha\"issinsky for his valuable remarks.

\section{Notation} \label{sctNotation}
Let $\C$ be the complex plane and $\widehat{\C}$ be the Riemann sphere. We use the convention that $\N=\{1,2,3,\dots\}$ and $\N_0 = \{0\} \cup \N$. As usual, the symbol $\log$ denotes the logarithm to the base $e$, and $\log_b$ the logarithm to the base $b$ for $b>0$.

The cardinality of a set $A$ is denoted by $\card{A}$. For $x\in\R$, we define $\lfloor x\rfloor$ as the greatest integer $\leq x$, and $\lceil x \rceil$ the smallest integer $\geq x$.

Let $g\: X\rightarrow Y$ be a function between two sets $X$ and $Y$. We denote the restriction of $g$ to a subset $Z$ of $X$ by $g|_Z$. 

Let $(X,d)$ be a metric space. For subsets $A,B\subseteq X$, we set $d(A,B)=\inf \{d(x,y)\,|\, x\in A,\,y\in B\}$, and $d(A,x)=d(x,A)=d(A,\{x\})$ for $x\in X$. For each subset $Y\subseteq X$, we denote the diameter of $Y$ by $\diam_d(Y)=\sup\{d(x,y)\,|\,x,y\in Y\}$, the interior of $Y$ by $\inter Y$, and the characteristic function of $Y$ by $\mathbbm{1}_Y$, which maps each $x\in Y$ to $1\in\R$. We use the convention that $\mathbbm{1}=\mathbbm{1}_X$ when the space $X$ is clear from the context. The identity map $\id_X\: X\rightarrow X$ sends each $x\in X$ to $x$ itself. For each $r>0$, we define $N^r_d(A)$ to be the open $r$-neighborhood $\{y\in X \,|\, d(y,A)<r\}$ of $A$, and $\overline{N^r_d}(A)$ the closed $r$-neighborhood $\{y\in X \,|\, d(y,A)\leq r\}$ of $A$. For $x\in X$, we denote the open (resp.\ closed) ball of radius $r$ centered at $x$ by $B_d(x, r)$ (resp.\ $\overline{B_d}(x,r)$). 

We set $\CCC(X)$ (resp.\ $B(X)$) to be the space of continuous (resp.\ bounded Borel) functions from $X$ to $\R$, by $\MMM(X)$ the set of finite signed Borel measures, and $\PPP(X)$ the set of Borel probability measures on $X$. For $\mu\in\MMM(X)$, we use $\Norm{\mu}$ to denote the total variation norm of $\mu$, $\supp \mu$ the support of $\mu$, and
$$
\langle \mu,u \rangle = \int \! u \,\mathrm{d}\mu
$$
for each $u\in\CCC(S^2)$. If we do not specify otherwise, we equip $\CCC(X)$  with the uniform norm $\Norm{\cdot}_\infty$. For a point $x\in X$, we define $\delta_x$ as the Dirac measure supported on $\{x\}$. For $g\in\CCC(X)$ we set $\MMM(X,g)$ to be the set of $g$-invariant Borel probability measures on $X$.

The space of real-valued H\"{o}lder continuous functions with an exponent $\alpha\in (0,1]$ on a compact metric space $(X,d)$ is denoted as $\Holder{\alpha}(X,d)$. For each $\phi\in\Holder{\alpha}(X,d)$,
\begin{equation}   \label{eqDef|.|alpha}
\Hseminorm{\alpha}{\phi}= \sup \bigg\{\frac{\abs{\phi(x)- \phi(y)}}{ d(x,y)^\alpha} \,\bigg|\, x,y\in X, \,x\neq y \bigg\},
\end{equation}
and the H\"{o}lder norm is defined as
\begin{equation}  \label{eqDefHolderNorm}
\Hnorm{\alpha}{\phi} = \Hseminorm{\alpha}{\phi}  + \Norm{\phi}_{\infty}.
\end{equation}
For given $f\: X \rightarrow X$ and $\varphi \in \CCC(X)$, we define
\begin{equation}    \label{eqDefSnPt}
S_n \varphi (x)  = \sum\limits_{j=0}^{n-1} \varphi(f^j(x)) 
\end{equation}
for $x\in X$ and $n\in\N_0$. Note that when $n=0$, by definition, we always have $S_0 \varphi = 0$.

\section{Thurston maps} \label{sctThurstonMap}
In this section, we quickly go over some key concepts and results on Thurston maps, and expanding Thurston maps in particular. For a more thorough treatment of the subject, we refer to \cite{BM10}. We end this section by generalizing a lemma in \cite{BM10}.

Most of the definitions and results here were discussed in \cite[Section~3]{Li13}, but for the convenience of the reader, we record them here nonetheless. We will use the same formulation as in \cite[Section~3]{Li13} whenever possible.

Let $S^2$ denote an oriented topological $2$-sphere. A continuous map $f\:S^2\rightarrow S^2$ is called a \defn{branched covering map} on $S^2$ if for each point $x\in S^2$, there exists a positive integer $d\in \N$, open neighborhoods $U$ of $x$ and $V$ of $y=f(x)$, open neighborhoods $U'$ and $V'$ of $0$ in $\widehat{\C}$, and orientation-preserving homeomorphisms $\varphi\:U\rightarrow U'$ and $\eta\:V\rightarrow V'$ such that $\varphi(x)=0$, $\eta(y)=0$, and
$$
(\eta\circ f\circ\varphi^{-1})(z)=z^d
$$
for each $z\in U'$. The positive integer $d$ above is called the \defn{local degree} of $f$ at $x$ and is denoted by $\deg_f (x)$. The \defn{degree} of $f$ is
\begin{equation}   \label{eqDeg=SumLocalDegree}
\deg f=\sum\limits_{x\in f^{-1}(y)} \deg_f (x)
\end{equation}
for $y\in S^2$ and is independent of $y$. If $f\:S^2\rightarrow S^2$ and $g\:S^2\rightarrow S^2$ are two branched covering maps on $S^2$, then so is $f\circ g$, and
\begin{equation} \label{eqLocalDegreeProduct}
 \deg_{f\circ g}(x) = \deg_g(x)\deg_f(g(x)), \qquad \text{for each } x\in S^2,
\end{equation}   
and moreover, 
\begin{equation}  \label{eqDegreeProduct}
\deg(f\circ g) =  (\deg f)( \deg g).
\end{equation}

A point $x\in S^2$ is a \defn{critical point} of $f$ if $\deg_f(x) \geq 2$. The set of critical points of $f$ is denoted by $\crit f$. A point $y\in S^2$ is a \defn{postcritical point} of $f$ if $y = f^n(x)$ for some $x\in\crit f$ and $n\in\N$. The set of postcritical points of $f$ is denoted by $\post f$. Note that $\post f=\post f^n$ for all $n\in\N$.

\begin{definition} [Thurston maps] \label{defThurstonMap}
A Thurston map is a branched covering map $f\:S^2\rightarrow S^2$ on $S^2$ with $\deg f\geq 2$ and $\card(\post f)<+\infty$.
\end{definition}

We now recall the notation for cell decompositions of $S^2$ used in \cite{BM10} and \cite{Li13}. A \defn{cell of dimension $n$} in $S^2$, $n \in \{1,2\}$, is a subset $c\subseteq S^2$ that is homeomorphic to the closed unit ball $\overline{\B^n}$ in $\R^n$. We define the \defn{boundary of $c$}, denoted by $\partial c$, to be the set of points corresponding to $\partial\B^n$ under such a homeomorphism between $c$ and $\overline{\B^n}$. The \defn{interior of $c$} is defined to be $\inte (c) = c \setminus \partial c$. For each point $x\in S^2$, the set $\{x\}$ is considered a \defn{cell of dimension $0$} in $S^2$. For a cell $c$ of dimension $0$, we adopt the convention that $\partial c=\emptyset$ and $\inte (c) =c$. 

We record the following three definitions from \cite{BM10}.

\begin{definition}[Cell decompositions]\label{defcelldecomp}
Let $\DD$ be a collection of cells in $S^2$.  We say that $\DD$ is a \defn{cell decomposition of $S^2$} if the following conditions are satisfied:

\begin{itemize}

\smallskip
\item[(i)]
the union of all cells in $\DD$ is equal to $S^2$,

\smallskip
\item[(ii)] if $c\in \DD$, then $\partial c$ is a union of cells in $\DD$,

\smallskip
\item[(iii)] for $c_1,c_2 \in \DD$ with $c_1 \ne c_1$, we have $\inte (c_1) \cap \inte (c_2)= \emptyset$,  

\smallskip
\item[(iv)] every point in $S^2$ has a neighborhood that meets only finitely many cells in $\DD$.

\end{itemize}
\end{definition}

\begin{definition}[Refinements]\label{defrefine}
Let $\DD'$ and $\DD$ be two cell decompositions of $S^2$. We
say that $\DD'$ is a \defn{refinement} of $\DD$ if the following conditions are satisfied:
\begin{itemize}

\smallskip
\item[(i)] every cell $c\in \DD$ is the union of all cells $c'\in \DD'$ with $c'\subseteq c$.

\smallskip
\item[(ii)] for every cell $c'\in \DD'$ there exits a cell $c\in \DD$ with $c'\subseteq c$,

\end{itemize}
\end{definition}

\begin{definition}[Cellular maps and cellular Markov partitions]\label{defcellular}
Let $\DD'$ and $\DD$ be two cell decompositions of  $S^2$. We say that a continuous map $f \: S^2 \rightarrow S^2$ is \defn{cellular} for  $(\DD', \DD)$ if for every cell $c\in \DD'$, the restriction $f|_c$ of $f$ to $c$ is a homeomorphism of $c$ onto a cell in $\DD$. We say that $(\DD',\DD)$ is a \defn{cellular Markov partition} for $f$ if $f$ is cellular for $(\DD',\DD)$ and $\DD'$ is a refinement of $\DD$.
\end{definition}

Let $f\:S^2 \rightarrow S^2$ be a Thurston map, and $\CC\subseteq S^2$ be a Jordan curve containing $\post f$. Then the pair $f$ and $\CC$ induces natural cell decompositions $\DD^n(f,\CC)$ of $S^2$, for $n\in\N_0$, in the following way:

By the Jordan curve theorem, the set $S^2\setminus\CC$ has two connected components. We call the closure of one of them the \defn{white $0$-tile} for $(f,\CC)$, denoted by $X^0_w$, and the closure of the other one the \defn{black $0$-tile} for $(f,\CC)$, denoted by $X^0_b$. The set of \defn{$0$-tiles} is $\X^0(f,\CC)=\{X_b^0,X_w^0\}$. The set of \defn{$0$-vertices} is $\V^0(f,\CC)=\post f$. We set $\overline\V^0(f,\CC) = \{ \{x\} \,|\, x\in \V^0(f,\CC) \}$. The set of \defn{$0$-edges} $\E^0(f,\CC)$ is the set of the closures of the connected components of $\CC \setminus  \post f$. Then we get a cell decomposition 
$$
\DD^0(f,\CC)=\X^0(f,\CC) \cup \E^0(f,\CC) \cup \overline\V^0(f,\CC)
$$
of $S^2$ consisting of \emph{cells of level $0$}, or \defn{$0$-cells}.

We can recursively define unique cell decompositions $\DD^n(f,\CC)$, $n\in\N$, consisting of \defn{$n$-cells} such that $f$ is cellular for $(\DD^{n+1}(f,\CC),\DD^n(f,\CC))$. We refer to \cite[Lemma~5.4]{BM10} for more details. We denote by $\X^n(f,\CC)$ the set of $n$-cells of dimension 2, called \defn{$n$-tiles}; by $\E^n(f,\CC)$ the set of $n$-cells of dimension 1, called \defn{$n$-edges}; by $\overline\V^n(f,\CC)$ the set of $n$-cells of dimension 0; and by $\V^n(f,\CC)$ the set $\big\{x\,\big|\, \{x\}\in \overline\V^n(f,\CC)\big\}$, called the set of \defn{$n$-vertices}. The \defn{$k$-skeleton}, for $k\in\{0,1,2\}$, of $\DD^n(f,\CC)$ is the union of all $k$-cells in this cell decomposition. 

We record Proposition~6.1 of \cite{BM10} here in order to summarize properties of the cell decompositions $\DD^n(f,\CC)$ defined above.

\begin{prop}[M.~Bonk \& D.~Meyer, 2010] \label{propCellDecomp}
Let $k,n\in \N_0$, let   $f\: S^2\rightarrow S^2$ be a Thurston map,  $\CC\subseteq S^2$ be a Jordan curve with $\post f \subseteq \CC$, and   $m=\card(\post f)$. 
 
\smallskip
\begin{itemize}

\smallskip
\item[(i)] The map  $f^k$ is cellular for $(\DD^{n+k}(f,\CC), \DD^n(f,\CC))$. In particular, if  $c$ is any $(n+k)$-cell, then $f^k(c)$ is an $n$-cell, and $f^k|_c$ is a homeomorphism of $c$ onto $f^k(c)$.

\smallskip
\item[(ii)]  Let  $c$ be  an $n$-cell.  Then $f^{-k}(c)$ is equal to the union of all 
$(n+k)$-cells $c'$ with $f^k(c')=c$.

\smallskip
\item[(iii)] The $1$-skeleton of $\DD^n(f,\CC)$ is  equal to  $f^{-n}(\CC)$. The $0$-skeleton of $\DD^n(f,\CC)$ is the set $\V^n(f,\CC)=f^{-n}(\post f )$, and we have $\V^n(f,\CC) \subseteq \V^{n+k}(f,\CC)$. 

\smallskip
\item[(iv)] $\card(\X^n(f,\CC))=2(\deg f)^n$,  $\card(\E^n(f,\CC))=m(\deg f)^n$,  and $\card (\V^n(f,\CC)) \leq m (\deg f)^n$.

\smallskip
\item[(v)] The $n$-edges are precisely the closures of the connected components of $f^{-n}(\CC)\setminus f^{-n}(\post f )$. The $n$-tiles are precisely the closures of the connected components of $S^2\setminus f^{-n}(\CC)$.

\smallskip
\item[(vi)] Every $n$-tile  is an $m$-gon, i.e., the number of $n$-edges and the number of $n$-vertices contained in its boundary are equal to $m$.  

\end{itemize}
\end{prop}

We also note that for each $n$-edge $e\in\E^n(f,\CC)$, $n\in\N_0$, there exist exactly two $n$-tiles $X,X'\in\X^n(f,\CC)$ such that $X\cap X' = e$.

For $n\in \N_0$, we define \defn{the set of black $n$-tiles} as
$$
\X_b^n(f,\CC)=\{X\in\X^n(f,\CC) \, |\,  f^n(X)=X_b^0\},
$$
and the \defn{set of white $n$-tiles} as
$$
\X_w^n(f,\CC)=\{X\in\X^n(f,\CC) \, |\, f^n(X)=X_w^0\}.
$$
It follows immediately from Proposition~\ref{propCellDecomp} that
\begin{equation}   \label{eqCardBlackNTiles}
\card \( \X_b^n(f,\CC) \) = \card \(\X_w^n(f,\CC)\) = (\deg f)^n 
\end{equation}
for each $n\in\N_0$.

From now on, if the map $f$ and the Jordan curve $\CC$ are clear from the context, we will sometimes omit $(f,\CC)$ in the notation above.

If we fix the cell decomposition $\DD^n(f,\CC)$, $n\in\N_0$, we can define for each $v\in \V^n$ the \defn{$n$-flower of $v$} as
\begin{equation}   \label{defFlower}
W^n(v) = \bigcup  \{\inte (c) \,|\, c\in \DD^n,\, v\in c \},
\end{equation}
which is the interior of the union of all $n$-cells containing $v$. We denote for each $x\in S^2$
\begin{align}  \label{defU^n}
U^n(x) = \bigcup \{Y^n\in \X^n \,|\,  & \text{there exists } X^n\in\X^n \\
                                      & \text{with } x\in X^n, \, X^n\cap Y^n \neq \emptyset  \},  \notag
\end{align}
and for each integer $m\leq -1$, set $U^m(x)= S^2$. We define the \defn{$n$-partition} $O_n$ of $S^2$ induced by $(f,\CC)$ as
\begin{equation}   \label{defn-partition}
O_n= \{\inte (X^n)\,|\,X^n\in\X^n\} \cup \{\inte (e^n) \,|\,e^n\in\E^n\} \cup \overline{\V}^n.
\end{equation}

We can finally give a definition of expanding Thurston maps.

\begin{definition} [Expansion] \label{defExpanding}
A Thurston map $f\:S^2\rightarrow S^2$ is called \defn{expanding} if there exist a metric $d$ on $S^2$ that induces the standard topology on $S^2$ and a Jordan curve $\CC\subseteq S^2$ containing $\post f$ such that 
\begin{equation*}
\lim\limits_{n\to+\infty}\max \{\diam_d(X) \,|\, X\in \X^n(f,\CC)\}=0.
\end{equation*}
\end{definition}

It is proved in \cite[Corollary~6.4]{BM10} that for each expanding Thurs\-ton map $f$, we have $\card(\post f) \geq 3$.

\begin{rems}  \label{rmExpanding}
It is clear that if $f$ is an expanding Thurston map, so is $f^n$ for each $n\in\N$. We observe that being expanding is a topological property of a Thurston map and independent of the choice of the metric $d$ that generates the standard topology on $S^2$. By Lemma~8.1 in \cite{BM10}, it is also independent of the choice of the Jordan curve $\CC$ containing $\post f$. More precisely, if $f$ is an expanding Thurston map, then
\begin{equation*}
\lim\limits_{n\to+\infty}\max \!\big\{ \! \diam_{\widetilde{d}}(X) \,\big|\, X\in \X^n\big(f,\widetilde\CC \hspace{0.5mm}\big)\hspace{-0.3mm} \big\}\hspace{-0.3mm}=0,
\end{equation*}
for each metric $\widetilde{d}$ that generates the standard topology on $S^2$ and each Jordan curve $\widetilde\CC\subseteq S^2$ that contains $\post f$.
\end{rems}

P. Ha\"{\i}ssinsky and K. Pilgrim developed a notion of expansion in a more general context for finite branched coverings between topological spaces (see \cite[Section~2.1 and Section~2.2]{HP09}). This applies to Thurston maps and their notion of expansion is equivalent to our notion defined above in the context of Thurston maps (see \cite[Proposition~8.2]{BM10}). Such concepts of expansion are natural analogs, in the contexts of finite branched coverings and Thurston maps, to some of the more classical versions, such as expansive homeomorphisms and forward-expansive continuous maps between compact metric spaces (see for example, \cite[Definition~3.2.11]{KH95}), and distance-expanding maps between compact metric spaces (see for example, \cite[Chapter~4]{PU10}). Our notion of expansion is not equivalent to any such classical notion in the context of Thurston maps. In fact, as mentioned in the introduction, there are subtle connections between our notion of expansion and some classical notions of weak expansion. More precisely, one can prove that an expanding Thurston map is asymptotically $h$-expansive if and only if it has no periodic points. Moreover, such a map is never $h$-expansive. See \cite{Li14} for details.

For an expanding Thurston map $f$, we can fix a particular metric $d$ on $S^2$ called a \emph{visual metric for $f$} . For the existence and properties of such metrics, see \cite[Chapter~8]{BM10}. One major advantage of a visual metric $d$ is that in $(S^2,d)$ we have good quantitative control over the sizes of the cells in the cell decompositions discussed above. We summarize several results of this type (\cite[Lemma~8.10, Lemma~8.12, Lemma~8.13]{BM10}) in the lemma below.

\begin{lemma}[M.~Bonk \& D.~Meyer, 2010]   \label{lmCellBoundsBM}
Let $f\:S^2 \rightarrow S^2$ be an expanding Thurston map, and $\CC \subseteq S^2$ be a Jordan curve containing $\post f$. Let $d$ be a visual metric on $S^2$ for $f$ with an expansion factor $\Lambda>1$. Then there exist constants $C\geq 1$, $C'\geq 1$, $K\geq 1$, and $n_0\in\N_0$ with the following properties:
\begin{enumerate}
\smallskip
\item[(i)] $d(\sigma,\tau) \geq C^{-1} \Lambda^{-n}$ whenever $\sigma$ and $\tau$ are disjoint $n$-cells for $n\in \N_0$.

\smallskip
\item[(ii)] $C^{-1} \Lambda^{-n} \leq \diam_d(\tau) \leq C\Lambda^{-n}$ for all $n$-edges and all $n$-tiles $\tau$ for $n\in\N_0$.

\smallskip
\item[(iii)] $B_d(x,K^{-1} \Lambda^{-n} ) \subseteq U^n(x) \subseteq B_d(x, K\Lambda^{-n})$ for $x\in S^2$ and $n\in\N_0$.

\smallskip
\item[(iv)] $U^{n+n_0} (x)\subseteq B_d(x,r) \subseteq U^{n-n_0}(x)$ where $n= \lceil -\log r / \log \Lambda \rceil$ for $r>0$ and $x\in S^2$.

\smallskip
\item[(v)] For every $n$-tile $X^n\in\X^n(f,\CC)$, $n\in\N_0$, there exists a point $p\in X^n$ such that $B_d(p,C^{-1}\Lambda^{-n}) \subseteq X^n \subseteq B_d(p,C\Lambda^{-n})$.
\end{enumerate}

Conversely, if $\widetilde{d}$ is a metric on $S^2$ satisfying conditions \textnormal{(i)} and \textnormal{(ii)} for some constant $\widetilde{C}\geq 1$, then $\widetilde{d}$ is a visual metric with an expansion factor $\widetilde{\Lambda}>1$.
\end{lemma}

Recall $U^n(x)$ is defined in (\ref{defU^n}).

In addition, we will need the fact that a visual metric $d$ induces the standard topology on $S^2$ (\cite[Proposition~8.9]{BM10}) and the fact that the metric space $(S^2,d)$ is linearly locally connected (\cite[Proposition~16.3]{BM10}). A metric space $(X,d)$ is \defn{linearly locally connected} if there exists a constant $L\geq 1$ such that the following conditions are satisfied:
\begin{enumerate}
\smallskip

\item  For all $z\in X$, $r > 0$, and $x,y\in B_d(z,r)$ with $x\neq y$, there exists a continuum $E\subseteq X$ with $x,y\subseteq E$ and $E\subseteq B_d(z,rL)$.

\smallskip

\item For all $z\in X$, $r > 0$, and $x,y\in X \setminus B_d(z,r)$ with $x\neq y$, there exists a continuum $E\subseteq X$ with $x,y\subseteq E$ and $E\subseteq X \setminus B_d(z,r/L)$.
\end{enumerate}
We call such a constant $L \geq 1$ a \defn{linear local connectivity constant of $d$}.

\begin{rem}   \label{rmChordalVisualQSEquiv}
If $f\: \widehat\C \rightarrow \widehat\C$ is a rational expanding Thurston map, then a visual metric is quasisymmetrically equivalent to the chordal metric on the Riemann sphere $\widehat\C$ (see \cite[Corollary~19.4]{BM10}). Here the chordal metric $\sigma$ on $\widehat\C$ is given by
$
\sigma(z,w) =\frac{2\abs{z-w}}{\sqrt{1+\abs{z}^2} \sqrt{1+\abs{w}^2}}
$
for $z,w\in\C$, and $\sigma(\infty,z)=\sigma(z,\infty)= \frac{2}{\sqrt{1+\abs{z}^2}}$ for $z\in \C$. We also note that a quasisymmetric embedding of a bounded connected metric space is H\"older continuous (see \cite[Section~11.1 and Corollary~11.5]{He01}). Accordingly, the classes of H\"older continuous functions on $\widehat\C$ equipped with the chordal metric and on $S^2=\widehat\C$ equipped with any visual metric for $f$ are the same (upto a change of the H\"older exponent).
\end{rem}

A Jordan curve $\CC\subseteq S^2$ is \defn{$f$-invariant} if $f(\CC)\subseteq \CC$. We are interested in $f$-invariant Jordan curves that contain $\post f$, since for such a Jordan curve $\CC$, we get a cellular Markov partition $(\DD^1(f,\CC),\DD^0(f,\CC))$ for $f$. According to Example~15.5 in \cite{BM10}, such $f$-invariant Jordan curves containing $\post{f}$ need not exist. However, M.~Bonk and D.~Meyer \cite[Theorem~1.2]{BM10} proved that there exists an $f^n$-invariant Jordan curve $\CC$ containing $\post{f}$ for each sufficiently large $n$ depending on $f$.

\begin{theorem}[M.~Bonk \& D.~Meyer, 2010]   \label{thmCexistsBM}
Let $f\:S^2\rightarrow S^2$ be an expanding Thurston map. Then for each $n\in\N$ sufficiently large, there exists a Jordan curve $\CC\subseteq S^2$ containing $\post f$ such that $f^n(\CC)\subseteq\CC$.
\end{theorem}

The following lemma was proved in \cite[Lemma~3.14]{Li13}. 

\begin{lemma}   \label{lmPreImageDense}
Let $f\:S^2\rightarrow S^2$ be an expanding Thurston map. Then for each $p\in S^2$, the set $\bigcup\limits_{n=1}^{+\infty}  f^{-n}(p)$ is dense in $S^2$, and
\begin{equation}
\lim\limits_{n\to +\infty}  \card(f^{-n}(p))  = +\infty.
\end{equation}
\end{lemma}

Expanding Thurston maps are Lipschitz with respect to a visual metric.

\begin{lemma}    \label{lmLipschitz}
Let $f\:S^2 \rightarrow S^2$ be an expanding Thurston map, and $d$ be a visual metric on $S^2$ for $f$ with an expansion factor $\Lambda>1$. Then $f$ is Lipschitz with respect to $d$.
\end{lemma}

\begin{proof}
Let $x,y\in S^2$ and we assume that 
\begin{equation}   \label{eqPflmLipschitz1}
0<d(x,y)< K^{-1}\Lambda^{-2},
\end{equation}
where $K\ge 1$ is a constant from Lemma~\ref{lmCellBoundsBM} depending only on $f$, $d$, $\CC$, and $\Lambda$.

Set $m=\max \big\{k\in\N_0 \,\big|\, y\in U^k(x) \big\}$, where $U^k(x)$ is defined in (\ref{defU^n}). By Lemma~\ref{lmCellBoundsBM}(iii), the number $m$ is finite. Then $y\notin U^{m+1}(x)$. Thus by Lemma~\ref{lmCellBoundsBM}(iii),
$$
\frac{1}{K} \Lambda^{-m-1} \leq d(x,y) \leq K \Lambda^{-m}.
$$
By (\ref{eqPflmLipschitz1}) we get $m\geq 1$. Since $f(y)\in f\(U^m(x)\) \subseteq U^{m-1}(f(x))$ by Proposition~\ref{propCellDecomp}, we get from Lemma~\ref{lmCellBoundsBM}(iii) that 
$$
d(f(x),f(y)) \leq K\Lambda^{-m+1}.
$$
Therefore, 
$$
\frac{d(f(x),f(y))}{d(x,y)}   \leq \frac{K\Lambda^{-m+1}}{\frac{1}{K}\Lambda^{-m-1}} =K^2\Lambda^2,
$$
and $f$ is Lipschitz with respect to $d$.
\end{proof}

We now establish a generalization of \cite[Lemma~16.1]{BM10}. It is an essential ingredient for the distortion lemmas (Lemma~\ref{lmSnPhiBound} and Lemma~\ref{lmSigmaExpSnPhiBound}) that we will repeatedly use later.

\begin{lemma}  \label{lmMetricDistortion}
Let $f\:S^2 \rightarrow S^2$ be an expanding Thurston map, and $\CC \subseteq S^2$ be a Jordan curve that satisfies $\post f \subseteq \CC$ and $f^{n_\CC}(\CC)\subseteq\CC$ for some $n_\CC\in\N$. Let $d$ be a visual metric on $S^2$ for $f$ with an expansion factor $\Lambda>1$. Then there exists a constant $C_0 > 1$, depending only on $f$, $d$, and $\Lambda$, with the following property:

If $k,n\in\N_0$, $X^{n+k}\in\X^{n+k}(f,\CC)$, and $x,y\in X^{n+k}$, then 
\begin{equation}   \label{eqMetricDistortion}
\frac{1}{C_0} d(x,y) \leq \frac{d(f^n(x),f^n(y))}{\Lambda^n}  \leq C_0 d(x,y).
\end{equation}
\end{lemma}

\begin{proof}
In this proof, we set a constant $K= 2\max\{1,l_f\}$, where $l_f$ is the Lipschitz constant of $f$ with respect to $d$. Let $N=n_\CC$.

By Remark~\ref{rmExpanding}, the map $f^{N}$ is an expanding Thurston map. It is easy to see from Lemma~\ref{lmCellBoundsBM} that the metric $d$ is a visual metric for the expanding Thurston map $f^{N}$ with an expansion factor $\Lambda^{N}$. So by Lemma~16.1 in \cite{BM10}, there exists a constant $D\geq 1$ depending only on $f^N,d,\CC$, and $\Lambda^N$ such that for each $k,l\in\N_0$, each $X\in\X^{(l+k)N} (f,\CC)$, and each pair of points $x,y \in X$, we have
\begin{equation}     \label{eqExponDecayBM}
\frac{1}{D} d(x,y) \leq \frac{d(f^{lN}(x),f^{lN}(y))}{\Lambda^{lN}}  \leq D d(x,y).
\end{equation}

Fix $m,l\in \N_0$, $s,t\in\{0,1,\dots,N-1\}$, $X\in\X^{(mN +s)+(lN +t)} (f,\CC)$, and $x,y\in X$.

We prove the second inequality in (\ref{eqMetricDistortion}) with $n=mN+s$ and $k=lN+t$ by considering the following cases depending on whether $l=0$ or $l\geq 1$.

If $l=0$, then by Lemma~\ref{lmLipschitz} and the fact that $K> l_f$,
$$
d\(f^{lN+t}(x),f^{lN+t}(y)\) \leq K^t d(x,y) \leq K^{2N}d(x,y) \Lambda^{lN+t}.
$$

If $l\geq 1$, then by Lemma~\ref{lmLipschitz}, (\ref{eqExponDecayBM}), and the fact that $K>l_f$,
\begin{align*}
      & d\(f^{lN+t}(x),f^{lN+t}(y)\)  \\
   =  & d\(f^{(l-1)N+(N-s)}\(f^{t+s}(x)\) , f^{(l-1)N+(N-s)}\(f^{t+s}(y)\)     \)  \\
 \leq & K^{N-s} d\(f^{(l-1)N }\(f^{t+s}(x)\) , f^{(l-1)N }\(f^{t+s}(y)\)     \) \\
 \leq & K^{N-s} D d\( f^{t+s}(x), f^{t+s}(y)  \)  \Lambda^{(l-1)N}  \\
 \leq & K^{N-s} D \(K^{t+s} d\( x,y \) \) \Lambda^{lN + t}  \\
 \leq & K^{2N}  D d(x,y) \Lambda^{lN + t}.
\end{align*}

We consider the first inequality in (\ref{eqMetricDistortion}) with $n=mN+s$ and $k=lN+t$ now. By Proposition~\ref{propCellDecomp}(i), we can choose  $Y\in \X^{(m+l+2)N}(f,\CC)$ and two points $x',y'\in Y$ such that $f^{2N -s-t} (Y)=X$, $f^{2N -s-t} (x')=x$, and $f^{2N -s-t} (y')=y$. Note that $2N -s-t \geq 2$. Then by Lemma~\ref{lmLipschitz}, (\ref{eqExponDecayBM}), and the fact that $K>l_f$,
\begin{align*}
      & d\(f^{lN+t}(x),f^{lN+t}(y)\)  \\
   =  & d\(f^{lN+t}\(f^{2N -s-t}(x')\),f^{lN+t}\(f^{2N -s-t}(y')\)\)  \\
   =  & d\(f^{lN+2N -s}(x'),f^{lN+2N -s}(y')\)  \\
 \geq & K^{-s} d\(f^{lN+2N}(x'),f^{lN+2N}(y')\)  \\
 \geq & K^{-s} D^{-1} d(x',y') \Lambda^{lN+2N}   \\
 \geq & K^{-s} D^{-1} K^{-(2N -s-t)} d(x,y)   \Lambda^{lN+t}   \\
 \geq & K^{-2N} D^{-1} d(x,y) \Lambda^{lN+t} .
\end{align*}
Therefore, 
\begin{equation*}
\frac{1}{\widetilde{C}_0} d(x,y) \leq \frac{d(f^{lN +t}(x),f^{lN +t}(y))}{\Lambda^{lN +t}}  \leq \widetilde{C}_0 d(x,y),
\end{equation*}
where $\widetilde{C}_0 = K^{2N} D$ is a constant depending only on $f$, $d$, $\CC$, $\Lambda$, and $N$. For $f$, $d$, and $\Lambda$ fixed, we can now define $C_0$ to be its infimum of $\widetilde{C}_0$ over all Jordan curves $\CC\subseteq S^2$ containing $\post f$ and $N\in\N$ that satisfy $f^{N}(\CC)\subseteq \CC$. Then $C_0\geq K> 1$ satisfies (\ref{eqMetricDistortion}) and only depends on $f$, $d$, and $\Lambda$.
\end{proof}

\section{The Assumptions}      \label{sctAssumptions}
We state below the hypothesis under which we will develop our theory in most parts of this paper. We will repeatedly refer to such assumptions in the later sections.

\begin{assumptions}
\quad

\begin{enumerate}

\smallskip

\item $f\:S^2 \rightarrow S^2$ is an expanding Thurston map.

\smallskip

\item $\CC\subseteq S^2$ is a Jordan curve containing $\post f$ with the property that there exists $n_\CC\in\N$ such that $f^{n_\CC} (\CC)\subseteq \CC$ and $f^m(\CC)\nsubseteq \CC$ for each $m\in\{1,2,\dots,n_\CC-1\}$.

\smallskip

\item $d$ is a visual metric on $S^2$ for $f$ with an expansion factor $\Lambda>1$ and a linear local connectivity constant $L\geq 1$.

\smallskip

\item $\phi\in \Holder{\alpha}(S^2,d)$ is a real-valued H\"{o}lder continuous function with an exponent $\alpha\in(0,1]$.

\end{enumerate}

\end{assumptions}

Observe that by Theorem~\ref{thmCexistsBM}, for each $f$ in (1), there exists at least one Jordan curve $\CC$ that satisfies (2). Since for a fixed $f$, the number $n_\CC$ is uniquely determined by $\CC$ in (2), in the remaining part of the paper we will say that a quantity depends on $\CC$ even if it also depends on $n_\CC$.

Note that even though the values of $\Lambda$ and $L$ are not uniquely determined by the metric $d$, in the remainder of this paper, for each visual metric $d$ on $S^2$ for $f$, we will fix a choice of expansion factor $\Lambda$ and a choice of linear local connectivity constant $L$. We will say a quantity depends on the visual metric $d$ without mentioning the dependence on $\Lambda$ or $L$, even though if we had not fixed a choice of $\Lambda$ and a choice of $L$, it would have depended on $\Lambda$ or $L$ as well.

In the discussion below, depending on the conditions we will need, we will sometimes say ``Let $f$, $\CC$, $d$, $\phi$, $\alpha$ satisfy the Assumptions.'', and sometimes say ``Let $f$ and $d$ satisfy the Assumptions.'', etc.

\section{Existence}  \label{sctExistence}

We start this section with a brief review of some necessary concepts from the ergodic theory and dynamical systems. We then establish two distortion lemmas (Lemma~\ref{lmSnPhiBound} and Lemma~\ref{lmSigmaExpSnPhiBound}) before giving the definitions of a Gibbs state and a radial Gibbs state. Next, we define the Ruelle operator $\RR_\phi \: \CCC(S^2) \rightarrow \CCC(S^2)$ and show that any eigenmeasure $m_\phi$ for its adjoint $\RR_\phi^*$ is a Gibbs state. We will eventually see in Corollary~\ref{corMandMuUnique} that there is exactly one eigenmeasure $m_\phi$ for $\RR_\phi^*$. In Theorem~\ref{thmMuExist}, we construct an $f$-invariant Gibbs state $\mu_\phi$ which is absolutely continuous with respect to $m_\phi$. Finally, we prove that $\mu_\phi$ is an equilibrium state in Proposition~\ref{propTopPressureDefPreImg} and Corollary~\ref{corExistES}. We end this section by proving in Proposition~\ref{propRadialGibbsIFFGibbs} that the concept of a Gibbs state and that of a radial Gibbs state coincide if and only if the map has no periodic critical point.

We first review some concepts from dynamical systems. We refer the readers to \cite[Chapter~3]{PU10}, \cite[Chapter~9]{Wa82} or \cite[Chapter~20]{KH95} for a more detailed study of these concepts.

Let $(X,d)$ be a compact metric space and $g\:X\rightarrow X$ a continuous map. For $n\in\N$ and $x,y\in X$,
$$
d^n_g(x,y)=\operatorname{max}\big\{ \hspace{-0.2mm} d\!\(g^k(x),g^k(y)\) \hspace{-0.3mm} \big| k\in\{0,1,\dots,n-1\} \!\big\}
$$
defines a new metric on $X$. A set $F\subseteq X$ is \defn{$(n,\epsilon)$-separated}, for some $n\in\N$ and $\epsilon>0$, if for each pair of distinct points $x,y\in F$, we have $d^n_g(x,y)\geq \epsilon$. For $\epsilon > 0$ and $n\in\N$, let $F_n(\epsilon)$ be a maximal (in the sense of inclusion) $(n,\epsilon)$-separated set in $X$.

For each $\psi \in\CCC(X)$, the following limits exist and are equal, and we denote the limits by $P(g,\psi)$ (see for example, \cite[Theorem~3.3.2]{PU10}):
\begin{align}  \label{defTopPressure}
P(g,\psi)  & = \lim \limits_{\epsilon\to 0} \limsup\limits_{n\to+\infty} \frac{1}{n} \log  \sum\limits_{x\in F_n(\epsilon)} \exp(S_n \psi(x)) \notag \\
           & = \lim \limits_{\epsilon\to 0} \liminf\limits_{n\to+\infty} \frac{1}{n} \log  \sum\limits_{x\in F_n(\epsilon)} \exp(S_n \psi(x)), 
\end{align}
where $S_n \psi (x) = \sum\limits_{j=0}^{n-1} \psi(g^j(x))$ is defined in (\ref{eqDefSnPt}). We call $P(g,\psi)$ the \defn{topological pressure} of $g$ with respect to the \emph{potential} $\psi$. The quantity $h_{\operatorname{top}}(g) = P(g,0)$ is called the \defn{topological entropy} of $g$. Note that $P(g,\psi)$ is independent of $d$ as long as the topology on $X$ defined by $d$ remains the same (see \cite[Section~3.2]{PU10}).

A \defn{measurable partition} $\xi$ of $X$ is a countable collection $\xi=\{A_j\,|\,j\in J\}$ of mutually disjoint Borel sets with $\bigcup \xi = X$, where $J$ is a countable index set. For $x\in X$, we denote by $\xi(x)$ the unique element of $\xi$ that contains $x$. Let $\mu\in \MMM(X,g)$ be a $g$-invariant Borel probability measure on $X$. The \defn{information function} $I$ maps a measurable partition $\xi$ of $X$ to a $\mu$-a.e.\ defined real-valued function on $X$ in the following way:
\begin{equation}   \label{eqDefI}
I(\xi)(x) = -\log \mu(\xi(x)), \qquad \text{for } x\in X.
\end{equation} 

Let $\xi=\{A_j\,|\,j\in J\}$ and $\eta=\{B_k\,|\,k\in K\}$ be measurable partitions of $X$, where $J$ and $K$ are the corresponding index sets. We say $\xi$ is a \defn{refinement} of $\eta$ if for each $A_j\in\xi$, there exists $B_k\in\eta$ such that $A_j\subseteq B_k$. The \defn{common refinement} $\xi \vee \eta= \{A_j\cap B_k \,|\, j\in J,\, k\in K\}$ of $\xi$ and $\eta$ is also a measurable partition. Set $g^{-1}(\xi)=\{g^{-1}(A_j) \,|\,j\in J\}$, and denote for $n\in\N$,
$$
\xi^n_g= \bigvee\limits_{j=0}^{n-1} g^{-j}(\xi) = \xi\vee g^{-1}(\xi)\vee\cdots\vee g^{-(n-1)}(\xi),
$$
and let $\xi^\infty_g$ be the smallest $\sigma$-algebra containing $\bigcup\limits_{n=1}^{+\infty}\xi^n_g$. The \defn{entropy} of $\xi$ is
$$
H_{\mu}(\xi)=-\sum\limits_{j\in J} \mu(A_j)\log\(\mu (A_j)\),
$$
where $0\log 0$ is defined to be 0. One can show (see \cite[Chapter 4]{Wa82}) that if $H_{\mu}(\xi)<+\infty$, then the following limit exists:
$$
h_{\mu}(g,\xi)=\lim\limits_{n\to+\infty} \frac{1}{n} H_{\mu}(\xi^n_g) \in[0,+\infty).
$$

The \defn{measure-theoretic entropy} of $g$ for $\mu$ is given by
\begin{align}   \label{eqDefMeasThEntropy}
h_{\mu}(g)=\sup\{h_{\mu}(g,\xi)\,|\, \xi & \text{ is a measurable partition of } X\\
                    &\qquad \qquad \quad \text{ with } H_{\mu}(\xi)<+\infty\}.    \notag
\end{align}
For each $\psi\in\CCC(X)$, the \defn{measure-theoretic pressure} $P_\mu(g,\psi)$ of $g$ for the measure $\mu$ and the potential $\psi$ is
\begin{equation}  \label{defMeasTheoPressure}
P_\mu(g,\psi)= h_\mu (g) + \int \! \psi \,\mathrm{d}\mu.
\end{equation}

By the Variational Principle (see for example, \cite[Theorem~3.4.1]{PU10}), we have that for each $\psi\in\CCC(X)$,
\begin{equation}  \label{eqVPPressure}
P(g,\psi)=\sup\{P_\mu(g,\psi)\,|\,\mu\in \MMM(X,g)\}.
\end{equation}
In particular, when $\psi$ is the constant function $0$,
\begin{equation}  \label{eqVPEntropy}
h_{\operatorname{top}}(g)=\sup\{h_{\mu}(g)\,|\,\mu\in \MMM(X,g)\}.
\end{equation}
A measure $\mu$ that attains the supremum in (\ref{eqVPPressure}) is called an \defn{equilibrium state} for the transformation $g$ and the potential $\psi$. A measure $\mu$ that attains the supremum in (\ref{eqVPEntropy}) is called a \defn{measure of maximal entropy} of $g$.

\smallskip

Now we go back to the dynamical system $(S^2,f)$ where $f$ is an expanding Thurston map.

By the work of P.~Ha\"{\i}ssinsky and K.~Pilgrim \cite{HP09}, and M.~Bonk and D.~Meyer \cite{BM10}, we know that there exists a unique measure of maximal entropy $\mu_f$ for $f$, and that
\begin{equation*}
h_{\operatorname{top}} (f)=\log(\deg f).
\end{equation*}
In this section, we generalize the existence part of this result to equilibrium states for real-valued H\"older continuous potentials. We prove the uniqueness in the next section.

We first establish the following two distortion lemmas that serve as the cornerstones for all the analysis in the thermodynamical formalism. 

\begin{lemma}   \label{lmSnPhiBound}
Let $f$, $\CC$, $d$, $L$, $\Lambda$, $\phi$, $\alpha$ satisfy the Assumptions. Then there exists a constant $C_1=C_1(f,d,\phi,\alpha)$ depending only on $f$, $d$, $\phi$, and $\alpha$ such that
\begin{equation}  \label{eqSnPhiBound}
\Abs{S_n\phi(x)-S_n\phi(y)}  \leq C_1 d(f^n(x),f^n(y))^\alpha,
\end{equation}
for $n,m\in\N_0$ with $n\leq m $, $X^m\in\X^m(f,\CC)$, and $x,y\in X^m$. Quantitatively, we choose 
\begin{equation}   \label{eqC1Expression}
C_1= \frac{\Hseminorm{\alpha}{\phi} C_0^\alpha}{1-\Lambda^{-\alpha}},
\end{equation}
where $C_0 > 1$ is a constant depending only on $f$ and $d$ from Lemma~\ref{lmMetricDistortion}.
\end{lemma}

Note that due to the convention described in Section~\ref{sctAssumptions}, we do not say that $C_1$ depends on $\Lambda$.

\begin{proof}
For $n=0$, inequality (\ref{eqSnPhiBound}) trivially follows from the definition of $S_n$.

By Lemma~\ref{lmMetricDistortion}, we have that for each $m\in \N_0$, each $m$-tile $X^m \in\X^m(f,\CC)$, each $x,y\in X^m$, and for $0\leq j \leq n \leq m$,
$$
d(f^j(x),f^j(y)) \leq  C_0 \Lambda^{-(n-j)} d(f^n(x),f^n(y)).
$$
So $\Abs{\phi(f^j(x))-\phi(f^j(y))}  \leq \Hseminorm{\alpha}{\phi} C_0^\alpha \Lambda^{-\alpha(n-j)} d(f^n(x),f^n(y))^\alpha$. Thus for each $n\in\N$ with $n \leq m $, we have
\begin{align*}
 \Abs{S_n\phi(x)-S_n\phi(y)} \leq & \sum\limits_{j=0}^{n-1}   \Abs{\phi(f^j(x))-\phi(f^j(y))}  \\
   \leq & \Hseminorm{\alpha}{\phi}  C_0^\alpha   d(f^n(x),f^n(y))^\alpha \sum\limits_{j=0}^{n-1} \Lambda^{-\alpha(n-j)}  \\
   \leq & \Hseminorm{\alpha}{\phi}  C_0^\alpha   d(f^n(x),f^n(y))^\alpha \sum\limits_{k=0}^{+\infty} \Lambda^{-\alpha k}  \\
   \leq & \frac{\Hseminorm{\alpha}{\phi} C_0^\alpha}{1-\Lambda^{-\alpha}}   d(f^n(x),f^n(y))^\alpha  \\
    =   & C_1   d(f^n(x),f^n(y))^\alpha. 
\end{align*}
\end{proof}

\begin{lemma}   \label{lmSigmaExpSnPhiBound}
Let $f$, $\CC$, $d$, $L$, $\Lambda$, $\phi$, $\alpha$ satisfy the Assumptions. Then there exists $C_2 = C_2(f,d,\phi,\alpha) \geq 1$ depending only on $f$, $d$, $\phi$, and $\alpha$ such that for each $x,y\in S^2$, and each $n\in\N_0$, we have
\begin{equation}   \label{eqSigmaExpSnPhiBound}
\frac{\sum\limits_{x'\in f^{-n}(x)}  \deg_{f^n}(x')  \exp (S_n\phi(x'))}{\sum\limits_{y'\in f^{-n}(y)}  \deg_{f^n}(y')  \exp (S_n\phi(y'))} \leq \exp\(4C_1 Ld(x,y)^\alpha\) \leq C_2,
\end{equation}
where $C_1$ is the constant from Lemma~\ref{lmSnPhiBound}. Quantitatively, we choose 
\begin{equation}  \label{eqC2Bound}
C_2 = \exp\(4C_1 L \(\diam_d(S^2)\)^\alpha \)  = \exp\(4 \frac{\Hseminorm{\alpha}{\phi} C_0}{1-\Lambda^{-1}} L \(\diam_d (S^2)\)^\alpha \),
\end{equation}
where $C_0 > 1$ is a constant depending only on $f$ and $d$ from Lemma~\ref{lmMetricDistortion}.
\end{lemma}

\begin{proof}
We denote $\Sigma(x,n) = \sum\limits_{x'\in f^{-n}(x)}  \deg_{f^n}(x')  \exp (S_n\phi(x'))$ for $x\in S^2$ and $n\in \N_0$.

We start with proving the first inequality in (\ref{eqSigmaExpSnPhiBound}). 

Let $X^0$ be either the black 0-tile $X^0_b$ or the white 0-tile $X^0_w$ in $\X^0(f,\CC)$. For $n\in \N_0$ and $X^n\in\X^n(f,\CC)$ with $f^n(X^n)=X^0$, by Proposition~\ref{propCellDecomp}(i), $f^n|_{X^n}$ is a homeomorphism of $X^n$ onto $X^0$. So for $x,y\in X^0$, there exist unique points $x',y'\in X^n$ with $x'\in f^{-n}(x)$ and $y'\in f^{-n}(y)$. Then by Lemma~\ref{lmSnPhiBound}, we have 
\begin{align*}
\exp \( S_n\phi(x') - S_n\phi(y')  \) & \leq \exp\(C_1 d(f^n(x'),f^n(y'))^\alpha\) \\
                                      & = \exp\(C_1 d(x,y)^\alpha\).
\end{align*}
Thus 
$
\exp \( S_n\phi(x')\) \leq \exp\(C_1 d(x,y)^\alpha\)  \exp \( S_n\phi(y')\).
$

By summing the last inequality over all pairs of $x',y'$ that are contained in the same $n$-tile $X^n$ with $f^n(X^n)=X^0$, and noting that each $x'$ (resp.\ $y'$) is contained in exactly $\deg_{f^n}(x')$ (resp.\ $\deg_{f^n}(y')$) distinct $n$-tiles $X^n$ with $f^n(X^n)=X^0$, we can conclude that
$$
\frac{\Sigma(x,n)}{\Sigma(y,n)}  \leq \exp\(C_1 d(x,y)^\alpha\).
$$

Recall that $f$, $\CC$, $d$, $L$, $\Lambda$, $\phi$, $\alpha$ satisfy the Assumptions. We then consider arbitrary $x\in X^0_w$ and $y\in X^0_b$. Since the metric space $(S^2,d)$ is linearly locally connected with a linear local connectivity constant $L\geq 1$, there exists a continuum $E\subseteq S^2$ with $x,y\in E$ and $E\subseteq B_d(x,Ld(x,y))$. We can then fix a point $z\in\CC \cap E$. Thus, we have
\begin{align*}
\frac{\Sigma(x,n)}{\Sigma(y,n)} \leq & \frac{\Sigma(x,n)}{\Sigma(z,n)} \frac{\Sigma(z,n)}{\Sigma(y,n)} \leq \exp\(C_1 \(d(x,z)^\alpha+d(z,y)^\alpha\)\)\\
\leq & \exp\(2 C_1 (\diam_d(E))^\alpha\) \leq \exp\(4 C_1 L d(x,y)^\alpha \).
\end{align*}

Finally, (\ref{eqC2Bound}) follows from (\ref{eqC1Expression}) in Lemma~\ref{lmSnPhiBound}.
\end{proof}

Let $f$, $\CC$, $d$, $L$, $\Lambda$, $\phi$, $\alpha$ satisfy the Assumptions. We now define the Gibbs states with respect to $f$, $\CC$, and $\phi$.

\begin{definition}   \label{defGibbsState}
A Borel probability measure $\mu \in \PPP(S^2)$ is a \defn{Gibbs state} with respect to $f$, $\CC$, and $\phi$ if there exist constants $P_\mu\in \R$ and $C_\mu \geq 1$ such that for each $n\in\N_0$, each $n$-tile $X^n\in\X^n(f,\CC)$, and each $x\in X^n$, we have
\begin{equation}  \label{eqGibbsState}
\frac{1}{C_\mu}  \leq  \frac{\mu(X^n)}{\exp(S_n\phi(x)-nP_\mu)}  \leq  C_\mu.
\end{equation}
\end{definition}

Compare the above definition with the following one, which is used in some classical dynamical systems.

\begin{definition}   \label{defRadialGibbsState}
A Borel probability measure $\mu \in \PPP(S^2)$ is a \defn{radial Gibbs state} with respect to $f$, $d$, and $\phi$ if there exist constants $\widetilde{P}_\mu\in\R$ and $\widetilde{C}_\mu \geq 1$ such that for each $n\in\N_0$, and each $x\in S^2$, we have
\begin{equation}  \label{eqRadialGibbsState}
\frac{1}{\widetilde{C}_\mu}  \leq  \frac{\mu\big( B_d(x,\Lambda^{-n})\big)}{\exp\big(S_n\phi(x)-n\widetilde{P}_\mu\big)}  \leq  \widetilde{C}_\mu.
\end{equation}
\end{definition}

One observes that for each Gibbs state $\mu$ with respect to $f$, $\CC$, and $\phi$, the constant $P_\mu$ is unique. Similarly, the constant $\widetilde{P}_\mu$ is unique for each radial Gibbs state with respect to $f$, $d$, and $\phi$.

\begin{ex}
Let $f\: S^2 \rightarrow S^2$ be an expanding Thurston map. There exists a unique measure of maximal entropy $\mu_0$ of $f$ (see \cite[Section~3.4 and Section~3.5]{HP09} and \cite[Theorem~20.9]{BM10}), which is an equilibrium state for a potential $\phi \equiv 0$. We can show that $\mu_0$ is a Gibbs state for $f$, $\CC$, $\phi\equiv 0$, whenever $\CC$ is a Jordan curve on $S^2$ containing $\post f$.

Indeed, we know that there exist constants $w,b\in (0,1)$ depending only on $f$ such that for each $n\in\N_0$, each white $n$-tile $X^n_w\in\X^n_w(f,\CC)$, and each black $n$-tile $X^n_b \in \X^n_b(f,\CC)$, we have $\mu_0 (X^n_w) = w (\deg f)^{-n}$ and $\mu_0 (X^n_b) = b (\deg f)^{-n}$ (\cite[Proposition~20.7 and Theorem~20.9]{BM10}). Thus $\mu_0$ is a Gibbs state for $f$, $\CC$, $\phi\equiv 0$, with $P_{\mu_0} = \deg f = h_{\operatorname{top}} (f)$ (see \cite[Corollary~20.8]{BM10}).
\end{ex}

As we see from the example above, Definition~\ref{defGibbsState} is a more appropriate definition for expanding Thurston map. Moreover, we will prove in Proposition~\ref{propRadialGibbsIFFGibbs} that the concept of a Gibbs state and that of a radial Gibbs state coincide if and only if $f$ has no periodic critical point.

\begin{prop}  \label{propInvGibbsIsEqlbStatePLessThanPressure}
Let $f$, $\CC$, $n_\CC$, $d$, $\phi$, $\alpha$ satisfy the Assumptions. Then for each $f$-invariant Gibbs state $\mu\in\MMM(S^2,f)$ with respect to $f$, $\CC$, and $\phi$, we have
\begin{equation}
P_\mu \leq h_\mu(f) + \int \! \phi \,\mathrm{d}\mu \leq P(f,\phi).
\end{equation}
\end{prop}

\begin{proof}
Note that the second inequality follows from the Variational Principle (\ref{eqVPPressure}) (see for example, \cite[Theorem~3.4.1]{PU10} for details).

Let $N=n_\CC$.

Recall measurable partitions $O_n,n\in\N$, of $S^2$ defined in (\ref{defn-partition}). Since $f^N(\CC)\subseteq \CC$, it is clear that $O_{iN}$ is a refinement of $O_{jN}$ for $i\geq j \geq 1$. Observe that by Proposition~\ref{propCellDecomp}(i) and induction, we can conclude that for each $k\in\N$,
\begin{equation}   \label{eqVeeO}
O_N \vee f^{-N} (O_N) \vee \cdots \vee f^{-kN}(O_N) = O_{(k+1)N}.
\end{equation}
So for $m,k\in\N$, the measurable partition $\bigvee\limits_{j=0}^{kN+m-1} f^{-j}(O_N)$ is a refinement of $O_{(k+1)N}$.

By Shannon-McMillan-Breiman Theorem (see for example, \cite[Theorem 2.5.4]{PU10}), $h_\mu (f,O_N) = \int \! f_\mathcal{I} \,\mathrm{d}\mu$, where
$$
f_\mathcal{I}=\lim\limits_{n\to+\infty} \frac{1}{n+1} I\bigg(\bigvee\limits_{j=0}^n f^{-j}(O_N)\bigg)  \quad \mu\text{-a.e.\ and in }L^1(\mu),
$$
and the information function $I$ is defined in (\ref{eqDefI}).

Note that for $n\in\N$, $c\in O_n$, and $X^n\in\X^n(f,\CC)$, either $c\cap X^n =\emptyset$ or $c\subseteq X^n$.

For $n\in\N_0$ and $x\in S^2$, we denote by $X^n(x)$ any one of the $n$-tiles containing $x$. Recall that $O_n(x)$ denotes the unique set in the measurable partition $O_n$ that contains $x$. Note that $O_n(x) \subseteq X^n(x)$. By (\ref{eqVeeO}) and (\ref{eqGibbsState}) we get
\begin{align*}
\int \! f_\mathcal{I}\,\mathrm{d}\mu  &= \lim\limits_{k\to+\infty}  \int \! \frac{1}{kN+1} I\bigg(\bigvee\limits_{j=0}^{kN} f^{-j}(O_N)\bigg)(x)\,\mathrm{d}\mu(x)  \\
      &\geq \liminf\limits_{k\to+\infty}  \int \! \frac{1}{kN+1} I(O_{(k+1)N})(x) \,\mathrm{d}\mu(x) \\
      &\geq \liminf\limits_{k\to+\infty}  \int \! \frac{1}{kN+1} \(-\log\mu\(X^{(k+1)N}(x)\)\)\,\mathrm{d}\mu(x)  \\
      &\geq \liminf\limits_{k\to+\infty}  \int \! \frac{(k+1)N P_\mu-S_{(k+1)N}\phi(x)-\log C_\mu}{(k+1)N}\,\mathrm{d}\mu(x)  \\
      & = P_\mu - \liminf\limits_{k\to+\infty} \frac{1}{(k+1)N} \int \!S_{(k+1)N}\phi(x) \,\mathrm{d}\mu(x)  \\
      & = P_\mu- \int\!\phi\,\mathrm{d}\mu,
\end{align*}
where the last equality comes from (\ref{eqDefSnPt}) and the identity $\int\! \psi \circ f \,\mathrm{d}\mu = \int\! \psi \,\mathrm{d}\mu$ for each $\psi\in\CCC(S^2)$ which is equivalent to the fact that $\mu$ is $f$-invariant. Since $O_N$ is a finite measurable partition, the condition that $H_\mu(O_N)<+\infty$ in (\ref{eqDefMeasThEntropy}) is fulfilled. By (\ref{eqDefMeasThEntropy}), we get that 
$$
h_\mu(f) \geq h_\mu(f,O_N)\geq P_\mu- \int\!\phi\,\mathrm{d}\mu.
$$
Therefore, $P_\mu \leq h_\mu(f) + \int\!\phi\,\mathrm{d}\mu$.
\end{proof}

Let $f\: S^2 \rightarrow S^2$ be an expanding Thurston map and $\psi\in \CCC(S^2)$ a continuous function. We define the \defn{Ruelle operator} $\RR_\psi$ on $\CCC(S^2)$ as the following
\begin{equation}   \label{eqDefRuelleOp}
\RR_\psi(u)(x)= \sum\limits_{y\in f^{-1}(x)}  \deg_f(y) u(y) \exp(\psi(y)),
\end{equation}
for each $u\in \CCC(S^2)$. To show that $\RR_\psi$ is well-defined, we need to prove that $\RR_\psi(u)(x)$ is continuous in $x\in S^2$ for each $u\in \CCC(S^2)$. Indeed, by fixing an arbitrary Jordan curve $\CC\subseteq S^2$ containing $\post f$, we know that for each $x$ in the white $0$-tile $X^0_w$, 
$$
\RR_\psi(u)(x)=\sum\limits_{X\in \X^1_w}u(y_X) \exp(\psi(y_X)), 
$$
where $y_X$ is the unique point contained in the white $1$-tile $X$ with the property that $f(y_X)=x$ (Proposition~\ref{propCellDecomp}(i)). If we move $x$ around continuously within $X^0_w$, then $y_X$ moves around continuously within $X$ for each $X\in\X^1_w$. Thus $\RR_\psi(u)(x)$ restricted to $X^0_w$ is continuous in $x$. Similarly, $\RR_\psi(u)(x)$ restricted to the black $1$-tile $X^0_b$ is also continuous in $x$. Hence $\RR_\psi(u)(x)$ is continuous in $x\in S^2$.

Note that by a similar argument as above, we see that the Ruelle operator $\RR_\psi \: \CCC(S^2) \rightarrow \CCC(S^2)$ has a natural extension to the space of real-valued bounded Borel functions $B(S^2)$ (equipped with the uniform norm) given by (\ref{eqDefRuelleOp}) for each $u\in B(S^2)$.

It is clear that $\RR_\psi$ is a positive, continuous operator on $\CCC(S^2)$ (resp.\ $B(S^2)$) with the operator norm $\sup\{\RR_\psi(\mathbbm{1})(x) \,|\, x\in S^2\}$. Moreover, we note that by induction and (\ref{eqLocalDegreeProduct}) we have
\begin{equation}   \label{eqR^nExpr}
\RR_\psi^n(u)(x)= \sum\limits_{y\in f^{-n}(x)} \deg_{f^n}(y)u(y)\exp (S_n\psi(y)),
\end{equation}
and
\begin{align}  \label{eqRuvf}
\RR_\psi(u(v\circ f))(x) & = \sum\limits_{y\in f^{-1}(x)} \deg_{f}(y)u(y)(v\circ f)(y) \exp (\psi(y)) \notag  \\ 
                           & = v (x) \RR_\psi(u)(x),
\end{align}
for $u,v\in B(S^2)$, $x\in S^2$, and $n\in\N$. Recall that the adjoint operator $\RR_\psi^*\: \CCC^*(S^2)\rightarrow \CCC^*(S^2)$ of $\RR_\psi$ acts on the dual space $\CCC^*(S^2)$ of the Banach space $\CCC(S^2)$. We identify $\CCC^*(S^2)$ with the space $\MMM(S^2)$ of finite signed Borel measures on $S^2$ by the Riesz representation theorem. From now on, we write $\langle\mu,u\rangle=\int \! u\,\mathrm{d}\mu$ whenever $u\in B(S^2)$ and $\mu\in \MMM(S^2)$.

\begin{lemma} \label{lmRmuLocal}
Let $f\: S^2 \rightarrow S^2$ be an expanding Thurston map, $\psi\in\CCC(S^2)$, and $\mu\in\CCC^*(S^2)$. Then
\begin{enumerate}

\smallskip
\item[(i)] $\langle \RR^*_\psi (\mu), u \rangle = \langle \mu , \RR_\psi(u) \rangle$ for $u\in B(S^2)$.

\smallskip
\item[(ii)] For each Borel set $A\subseteq S^2$ on which $f$ is injective, we have that $f(A)$ is a Borel set, and
\begin{equation} \label{eqRmuLocal}
\RR_\psi^*(\mu)(A) = \int_{f(A)} \!  (\deg_f (\cdot) \exp(\psi))\circ (f|_A)^{-1} \,\mathrm{d}\mu.
\end{equation}
\end{enumerate}
\end{lemma}

Recall that a collection $\mathfrak{P}$ of subsets of a set $\Omega$ is a \defn{$\pi$-system} if it is closed under intersection, i.e., if $A,B\in\mathfrak{P}$ then $A\cap B\in\mathfrak{P}$. A collection $\mathfrak{L}$ of subsets of $\Omega$ is a \defn{$\lambda$-system} if the following are satisfied: (1) $\Omega\in\mathfrak{L}$. (2) If $B,C\in\mathfrak{L}$ and $B\subseteq C$, then $C\setminus B \in \mathfrak{L}$. (3) If $A_n\in\mathfrak{L}$, $n\in\N$, with $A_n\subseteq A_{n+1}$, then $\bigcup\limits_{n\in\N} A_n \in\mathfrak{L}$.

\begin{proof}
For (i), it suffices to show that for each Borel set $A\subset S^2$,
\begin{equation}   \label{eqR*BorelSet}
\langle \RR^*_\psi (\mu), \mathbbm{1}_A \rangle = \langle \mu , \RR_\psi(\mathbbm{1}_A) \rangle.
\end{equation}

Let $\mathfrak{L}$ be the collection of Borel sets $A\subseteq S^2$ for which (\ref{eqR*BorelSet}) holds. Denote the collection of open subsets of $S^2$ by $\mathfrak{G}$. Then $\mathfrak{G}$ is a $\pi$-system.

We first observe from (\ref{eqDefRuelleOp}) that if $\{u_n\}_{n\in\N}$ is a non-decreasing sequence of real-valued functions on $S^2$, then so is $\{ \RR_\psi(u_n) \}_{n\in\N}$.
 
By the definition of $\RR^*_\psi$, we have
\begin{equation}    \label{eqR*CtsFns}
\langle \RR^*_\psi (\mu), u \rangle = \langle \mu , \RR_\psi(u) \rangle
\end{equation}
for $u\in\CCC(S^2)$. Fix an open set $U\subseteq S^2$, then there exists a non-decreasing sequence $\{g_n\}_{n\in\N}$ of real-valued continuous functions on $S^2$ supported in $U$ such that $g_n$ converges to $\mathbbm{1}_U$ pointwise as $n\longrightarrow +\infty$. Then $\{ \RR_\psi(g_n) \}_{n\in\N}$ is also a non-decreasing sequence of continuous functions, whose pointwise limit is $\RR_\psi(\mathbbm{1}_U)$. By the Lebesgue Monotone Convergence Theorem and (\ref{eqR*CtsFns}), we can conclude that (\ref{eqR*BorelSet}) holds for $A= U$. Thus $\mathfrak{G}\subseteq\mathfrak{L}$.

We now prove that $\mathfrak{L}$ is a $\lambda$-system. Indeed, since (\ref{eqR*CtsFns}) holds for $u=\mathbbm{1}_{S^2}$, we get $S^2\in\mathfrak{L}$. Given $B,C\in \mathfrak{L}$ with $B\subseteq C$, then $\mathbbm{1}_C - \mathbbm{1}_B = \mathbbm{1}_{C\setminus B}$ and $\RR_\psi(\mathbbm{1}_C) - \RR_\psi(\mathbbm{1}_B) = \RR_\psi(\mathbbm{1}_C-\mathbbm{1}_B)  = \RR_\psi(\mathbbm{1}_{C\setminus B})$ by (\ref{eqDefRuelleOp}). Thus $C\setminus B \in \mathfrak{L}$. Finally, given $A_n \in\mathfrak{L}$, $n\in\N$, with $A_n\subseteq A_{n+1}$, and let $A=\bigcup\limits_{n\in\N} A_n$. Then $\{ \mathbbm{1}_{A_n}\}_{n\in\N}$ and $\{ \RR_\psi( \mathbbm{1}_{A_n} )\}_{n\in\N}$ are non-decreasing sequences of real-valued Borel functions on $S^2$ that converge to $\mathbbm{1}_A$ and $\RR_\psi (\mathbbm{1}_A)$, respectively, as $n\longrightarrow+\infty$. Then by the Lebesgue Monotone Convergence Theorem, we get $A\in \mathfrak{L}$. Hence $\mathfrak{L}$ is a $\lambda$-system.

Recall that Dynkin's $\pi$-$\lambda$ theorem (see for example, \cite[Theorem~3.2]{Bi95}) states that if $\mathfrak{P}$ is a $\pi$-system and $\mathfrak{L}$ is a $\lambda$-system that contains $\mathfrak{P}$, then the $\sigma$-algebra $\sigma(\mathfrak{P})$ generated by $\mathfrak{P}$ is a subset of $\mathfrak{L}$. Thus by Dynkin's $\pi$-$\lambda$ theorem, the Borel $\sigma$-algebra $\sigma(\mathfrak{G})$ is a subset of $\mathfrak{L}$, i.e., equality (\ref{eqR*BorelSet}) holds for each Borel set $A\subseteq S^2$.

\smallskip

For (ii), we fix a Borel set $A\subseteq S^2$ on which $f$ is injective. By (\ref{eqDefRuelleOp}), we get that $\RR_\psi (\mathbbm{1}_A)(x)\neq 0$ if and only if $x\in f(A)$. Thus $f(A)$ is Borel. Then (\ref{eqRmuLocal}) follows immediately from (i) and (\ref{eqDefRuelleOp}) for $u\in B(S^2)$.
\end{proof}

\begin{definition}   \label{defJacobian}
Let $f\: S^2 \rightarrow S^2$ be an expanding Thurston map and $\mu \in \PPP(S^2)$ a Borel probability measure on $S^2$. A Borel function $J\:S^2 \rightarrow [0,+\infty)$ is a \defn{Jacobian (function)} for $f$ with respect to $\mu$ if for every Borel $A \subseteq S^2$ on which $f$ is injective, the following equation holds:
\begin{equation}   \label{eqDefJacobian}
\mu(f(A)) = \int_A \! J \, \mathrm{d}\mu.
\end{equation}
\end{definition}

\begin{cor}  \label{corJacobian}
Let $f\: S^2 \rightarrow S^2$ be an expanding Thurston map. For each $\psi\in \CCC(S^2)$ and each Borel probability measure $\mu\in\PPP(S^2)$, if $\RR_\psi^*(\mu)=c\mu$ for some constant $c > 0$, then the Jacobian $J$ for $f$ with respect to $\mu$ is given by
\begin{equation}   \label{eqJforRu=cu}
J(x) = \frac{c}{\deg_f(x)\exp(\psi(x))}  \qquad \text{for } x\in S^2.
\end{equation}
\end{cor}

\begin{proof}
We fix some $\CC$, $d$, $L$, $\Lambda$ that satisfy the Assumptions. 

By Lemma~\ref{lmRmuLocal}, for every Borel $A\subseteq S^2$ on which $f$ is injective, we have that $f(A)$ is Borel, and
\begin{equation}  \label{eqPfcorJacobian1}
\mu(A) = \frac{\RR_\psi^*(\mu)(A)}{c} =  \int_{f(A)} \!  \frac{1}{ J \circ (f|_A)^{-1} } \,\mathrm{d}\mu,
\end{equation}
for the function $J$ given in (\ref{eqJforRu=cu}).

Since $f$ is injective on each $1$-tile $X^1\in\X^1(f,\CC)$, and both $X^1$ and $f(X^1)$ are closed subsets of $S^2$ by Proposition~\ref{propCellDecomp}, in order to verify (\ref{eqDefJacobian}), it suffices to assume that $A\subseteq X$ for some $1$-tile $X\in\X^1(f,\CC)$. Denote the restriction of $\mu$ on $X$ by $\mu_X$, i.e., $\mu_X$ assigns $\mu(B)$ to each Borel subset $B$ of $X$.

Let $\widetilde\mu$ be a function defined on the set of Borel subsets of $X$ in such a way that $\widetilde\mu(B)= \mu(f(B))$ for each Borel $B\subseteq X$. It is clear that $\widetilde\mu$ is a Borel measure on $X$. In this notation, we can write (\ref{eqPfcorJacobian1}) as
\begin{equation}   \label{eqPfcorJacobian2}
\mu_X(A) = \int_A \! \frac{1}{J|_X} \,\mathrm{d}\widetilde\mu,
\end{equation}
for each Borel $A\subseteq X$.

By (\ref{eqPfcorJacobian2}), we know that $\mu_X$ is absolutely continuous with respect to $\widetilde\mu$. On the other hand, since $J$ is positive and uniformly bounded away from $+\infty$ on $X$, we can conclude that $\widetilde\mu$ is absolutely continuous with respect to $\mu_X$. Therefore, by the Radon-Nikodym theorem, for each Borel $A\subseteq X$, we get
$
\mu(f(A)) = \widetilde\mu(A) = \int_A \! J|_X \,\mathrm{d} \mu_X = \int_A \! J \,\mathrm{d}\mu.
$
\end{proof}

\begin{lemma}  \label{lmTileInIntTile}
Let $f\:S^2 \rightarrow S^2$ be an expanding Thurston map, and $\CC\subseteq S^2$ be a Jordan curve containing $\post f$. Then there exists a constant $M\in\N$ with the following property:

For each $m\in\N$ with $m \geq M$, each $n\in\N_0$, and each $n$-tile $X^n\in\X^n(f,\CC)$, there exist a white $(n+m)$-tile $X_w^{n+m}\in\X_w^{n+m}(f,\CC)$ and a black $(n+m)$-tile $X_b^{n+m}\in\X_b^{n+m}(f,\CC)$ such that $X_w^{n+m}\cup X_b^{n+m}\subseteq \inte (X^n)$.
\end{lemma}

\begin{proof}
We fix some $d$, $L$, $\Lambda$ that satisfy the Assumptions. 

By Lemma~\ref{lmCellBoundsBM}(v), there exists a constant $C \geq 1$ depending only on $f$, $\CC$, and $d$ such that for each $k\in\N_0$, each $k$-tile $Z^k\in\X^k(f,\CC)$, there exists a point $q\in Z^k$ such that 
\begin{equation*}
B_d(q,C^{-1}\Lambda^{-k}) \subseteq  Z^k  \subseteq B_d(q, C\Lambda^{-k}).
\end{equation*}
We set $M= \lceil \log_\Lambda (4C^2) \rceil + 1$. We fix an arbitrary $n\in\N$ and an $n$-tile $X^n\in\X^n(f,\CC)$. Choose a point $p\in X^n$ with $B_d(p,C^{-1}\Lambda^{-n})\subseteq X^n \subseteq B_d(p,C\Lambda^{-n})$. Then for each $m\in\N$ with $m \geq M$, we have $4C\Lambda^{-(n+m)} < C^{-1}\Lambda^{-n}$, and we can choose $X^{n+m}, Y^{n+m} \in\X^{n+m}(f,\CC)$ in such a way that $X^{n+m}$ is the $(n+m)$-tile containing $p$ and $Y^{n+m}\cap X^{n+m} = e^{n+m} \in \E^{n+m}(f,\CC)$ for each $m>M$. Thus $\diam_d(X^{n+m})\leq 2C\Lambda^{-(n+m)},\diam_d(Y^{n+m}) \leq 2C\Lambda^{-(n+m)}$, and 
$$
X^{n+m}\cup Y^{n+m} \subseteq \overline{B_d}\(p, 4 C\Lambda^{-(n+m)}\)  \subseteq B_d \(p, C^{-1}\Lambda^{-n}\) \subseteq \inte (X^n).
$$
Moreover, exactly one of $X^{n+m}$ and $Y^{n+m}$ is a white $(n+m)$-tile and the other one is a black $(n+m)$-tile.
\end{proof}

\begin{theorem}   \label{thmMexists}
Let $f\:S^2 \rightarrow S^2$ be an expanding Thurston map, and $d$ be a visual metric on $S^2$ for $f$. Let $\phi\in \Holder{\alpha}(S^2,d)$ be a real-valued H\"{o}lder continuous function with an exponent $\alpha\in(0,1]$. Then there exists a Borel probability measure $m_\phi \in\PPP(S^2)$ such that
\begin{equation}   \label{eqRm=cm}
\RR_\phi^*(m_\phi)=c m_\phi,
\end{equation}
where $c=\langle \RR_\phi^*(m_\phi),\mathbbm{1} \rangle$. Moreover, any $m_\phi \in\PPP(S^2)$ that satisfies (\ref{eqRm=cm}) for some $c>0$ has the following properties:
\begin{enumerate}
\smallskip

\item[(i)] The Jacobian for $f$ with respect to $m_\phi$ is 
$$
J(x)=c \exp(-\phi(x)).
$$

\smallskip

\item[(ii)] $m_\phi \bigg(\bigcup\limits_{j=0}^{+\infty} f^{-j} (\post f)\bigg)  = 0$.

\smallskip

\item[(iii)] The map $f$ with respect to $m_\phi$ is forward quasi-invariant (i.e., for each Borel set $A\subseteq S^2$, if $m_\phi(A)=0$, then $m_\phi(f(A))=0$), and nonsingular (i.e., for each Borel set $A\subseteq S^2$, $m_\phi(A)=0$ if and only if $m_\phi(f^{-1}(A))=0$).
\end{enumerate}
\end{theorem}

We will see later in Corollary~\ref{corMandMuUnique} that $m_\phi\in\PPP(S^2)$ satisfying (\ref{eqRm=cm}) is unique. We will also prove in Corollary~\ref{corNonAtomic} that $m_\phi$ is non-atomic.

\begin{proof}
We pick a Jordan curve $\CC\subseteq S^2$ that satisfies the Assumptions (see Theorem~\ref{thmCexistsBM} for the existence of such $\CC$).

Define $\tau\:\PPP(S^2)\rightarrow \PPP(S^2)$ by $\tau(\mu)=\frac{\RR_\phi^*(\mu)}{\langle\RR_\phi^*(\mu),\mathbbm{1}\rangle}$. Then $\tau$ is a continuous transformation on the non-empty, convex, compact (in the weak$^*$ topology, by Alaoglu's theorem) space $\PPP(S^2)$ of Borel probability measures on $S^2$. By the Schauder-Tikhonov Fixed Point Theorem (see for example, \cite[Theorem~3.1.7]{PU10}), there exists a measure $m_\phi\in\PPP(S^2)$ such that $\tau(m_\phi)=m_\phi$. Thus $\RR_\phi^*(m_\phi)=c m_\phi$ with $c=\langle \RR_\phi^*(m_\phi),\mathbbm{1} \rangle$. 

By Corollary~\ref{corJacobian}, the formula for the Jacobian for $f$ with respect to $m_\phi$ is
\begin{equation}  \label{eqJacobianWithDeg}
J(x)=c(\deg_f(x)\exp(\phi(x)))^{-1}, \qquad \text{for } x\in S^2.
\end{equation}

Since $\bigcup\limits_{j=0}^{+\infty} f^{-j} (\post f)$ is a countable set, the property (ii) follows if we can prove that $m_\phi(\{y\})=0$ for each $y\in \bigcup\limits_{j=0}^{+\infty} f^{-j} (\post f)$. Since for each $x\in S^2$,
\begin{equation} \label{eqm(f(x))}
m_\phi(\{f(x)\}) = \frac{c}{\deg_f(x)\exp(\phi(x))}   m_\phi(\{x\}),
\end{equation}
it suffices to prove that $m_\phi(\{x\})=0$ for each periodic $x\in \post f$.

Suppose that there exists $x\in \post f$ such that $f^l(x)=x$ for some $l\in\N$ and $m_\phi(\{x\})\neq 0$. Then by (\ref{eqm(f(x))}), (\ref{eqLocalDegreeProduct}), and induction,
\begin{equation}
m_\phi (\{x\}) = \frac{c^l}{\deg_{f^l}(x) \exp\(S_l\phi(x)\) } m_\phi(\{x\}),
\end{equation}
where $S_l \phi$ is defined in (\ref{eqDefSnPt}). Thus $c^l = \deg_{f^l}(x) \exp\( S_l\phi(x)\) $.

Similarly, for each $k\in\N$ and each $y\in f^{-kl}(x)$, we have
\begin{equation}
m_\phi (\{x\}) = \frac{c^{kl}}{\deg_{f^{kl}}(y) \exp\( S_{kl}\phi(y) \) } m_\phi(\{y\}).
\end{equation}
Thus
\begin{equation}
m_\phi(\{y\}) = \frac{ \deg_{f^{kl}}(y) \exp(S_{kl}\phi(y))}{\(\deg_{f^{l}}(x)\)^k \exp(S_{kl}\phi(x))}  m_\phi(\{x\}).
\end{equation}

\begin{figure}
    \centering
    \begin{overpic}
    [width=10cm, 
    tics=20]{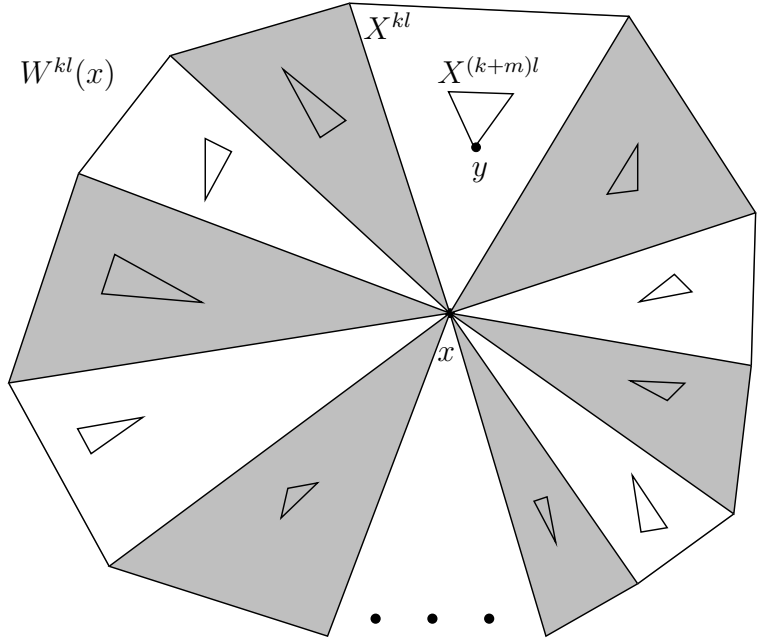}
    \put(5,210){$W^{kl}(x)$}
    \put(163,105){$x$}
    \put(176,175){$y$}
    \put(135,227){$X^{kl}$}
    \put(163,210){$X^{(k+m)l}$}
    \end{overpic}
    \caption{A $(kl)$-flower $W^{kl}(x)$, with $\card(\post f) = 3$.}
    \label{figFlower}
\end{figure}
    
Note that for each $k\in\N$, we have $x\in \V^{kl}(f,\CC)$. The closure of the $(kl)$-flower $W^{kl}(x)$ of $x$ contains exactly $2\(\deg_{f^l}(x)\)^k$ distinct $(kl)$-tiles whose intersection is $\{x\}$ (see \cite[Lemma~7.2(i)]{BM10}). By Lemma~\ref{lmTileInIntTile}, there exists $m\in\N$ that only depends on $f$, $\CC$, and $d$ such that for each $k\in\N$, each $(kl)$-tile $X^{kl}\in \X^{kl}(f,\CC)$ contained in $\overline{W^{kl}(x)}$, there exists a $((k+m)l)$-tile $X^{(k+m)l}\in\X^{(k+m)l}(f,\CC)$ such that $X^{(k+m)l} \subseteq \inte \(X^{kl}\)$. So there exists a unique $y\in X^{(k+m)l}\subseteq \inte \(X^{kl}\)$ such that $f^{(k+m)l}(y) = x$ by Proposition~\ref{propCellDecomp}(i), see Figure~\ref{figFlower}. For each $k\in\N$, we denote by $T_k$ the set consisting of one such $y$ from each $(kl)$-tile $X^{kl}\subseteq \overline{W^{kl}(x)}$. Note that 
\begin{equation}  \label{eqCardTk}
T_k = 2\(\deg_{f^l}(x)\)^k.
\end{equation}
Then $\{T_k\}_{k\in\N}$ is a sequence of subsets of $\bigcup\limits_{j=0}^{+\infty} f^{-j} (\post f)$. Since $f$ is expanding, we can choose an increasing sequence $\{k_i\}_{i\in\N}$ of integers recursively in such a way that $W^{lk_{i+1}}(x) \cap \Big(\bigcup\limits_{j=1}^i T_{k_j}\Big) = \emptyset$ for each $i\in\N$. Then $\{T_{k_i}\}_{i\in\N}$ is a sequence of mutually disjoint sets. Thus by Lemma~\ref{lmSnPhiBound}, there exists a constant $D$ that only depends on $f$, $d$, $\phi$, and $\alpha$ such that
\begin{align*}
      & m_\phi\bigg( \bigcup\limits_{j=0}^{+\infty} f^{-j} (\post f) \bigg)  
\geq \sum\limits_{i=1}^{+\infty}  \sum\limits_{y\in T_{k_i}}  m_\phi(\{y\})   \\
=     & \sum\limits_{i=1}^{+\infty}  \sum\limits_{y\in T_{k_i}}  \frac{ \deg_{f^{(k_i+m)l}}(y) \exp(S_{(k_i+m)l}\phi(y))}{\(\deg_{f^{l}}(x)\)^{k_i+m} \exp(S_{(k_i+m)l}\phi(x))}  m_\phi(\{x\})  \\
\geq  &  m_\phi(\{x\}) \sum\limits_{i=1}^{+\infty}  \sum\limits_{y\in T_{k_i}}  \frac{ \exp(S_{k_i l}\phi(y)-S_{k_i l}\phi(x)) \exp(-2ml\Norm{\phi}_{\infty})}{\(\deg_{f^{l}}(x)\)^{k_i+m} }\\
\geq  & m_\phi(\{x\}) \sum\limits_{i=1}^{+\infty}  \sum\limits_{y\in T_{k_i}}  \frac{ \exp(D - 2ml\Norm{\phi}_{\infty})}{\(\deg_{f^{l}}(x)\)^m \(\deg_{f^{l}}(x)\)^{k_i}}.
\end{align*}
Combining the above with (\ref{eqCardTk}), we get
\begin{equation*}
m_\phi\bigg( \bigcup\limits_{j=0}^{+\infty} f^{-j} (\post f) \bigg) =      \frac{m_\phi(\{x\}) \exp(D - 2ml\Norm{\phi}_{\infty})}{\(\deg_{f^{l}}(x)\)^m } \sum\limits_{i=1}^{+\infty}  2 =     +\infty.
\end{equation*}
This contradicts the fact that $m_\phi$ is a finite Borel measure.

Next, in order to prove the formula for the Jacobian for $f$ with respect to $m_\phi$ in property (i), we observe that by Lemma~\ref{lmRmuLocal} and (\ref{eqJacobianWithDeg}), for every Borel set $A\subseteq S^2$ on which $f$ is injective, we have that $f(A)$ is a Borel set and
\begin{align*}
m_\phi(f(A)) & = m_\phi(f(A) \setminus \post f ) = m_\phi(f(A\setminus(\post f \cup \crit f)))  \\
          & = \int_{A\setminus (\post f \,\cup \,\crit f)} \! c \exp(-\phi) \,\mathrm{d}m_\phi = \int_A \! c \exp(-\phi) \,\mathrm{d}m_\phi.
\end{align*} 

Finally, we prove the last property. Fix a Borel set $A\subseteq S^2$ with $m_\phi(A)=0$. For each $1$-tile $X^1\in\X^1(f,\CC)$, the map $f$ is injective both on $A\cap X^1$ and on $f^{-1}(A) \cap X^1$ by Proposition~\ref{propCellDecomp}(i). So it follows from the formula for the Jacobian that  $m_\phi\(f\(A\cap X^1\)\) = 0$ and $m_\phi\(f^{-1}(A) \cap X^1\) = 0$. Thus $m_\phi(f(A))=0$ and $\phi(f^{-1}(A))=0$. It is clear now that $f$ is forward quasi-invariant and nonsingular with respect to $m_\phi$.
\end{proof}

\begin{prop}   \label{propMfABounds}
Let $f$, $d$, $\phi$, $\alpha$ satisfy the Assumptions. Let $m_\phi$ be a Borel probability measure defined in Theorem~\ref{thmMexists} with $\RR_\phi^* (m_\phi) = c m_\phi$ where $c = \langle \RR_\phi^*(m_\phi),\mathbbm{1} \rangle$. Then for every Borel set $A\subseteq S^2$, we have
$$
\frac{1}{\deg f} \int_A \! J \, \mathrm{d}m_\phi  \leq m_\phi(f(A)) \leq  \int_A \! J \, \mathrm{d}m_\phi.
$$
where $J=c\exp(-\phi)$.
\end{prop}

\begin{proof}
We pick a Jordan curve $\CC\subseteq S^2$ that satisfies the Assumptions.

The second inequality follows from Definition~\ref{defJacobian} and Theorem~\ref{thmMexists}.

Let $B=f(A) \cap X^0_w$ and $C=f(A)\cap \inte (X^0_b)$, where $X^0_w,X^0_b\in \X^0(f,\CC)$ are the white $0$-tile and the black $0$-tile, respectively. Then $B\cap C=\emptyset$ and $B\cup C = f(A)$. For each white 1-tile $X^1_w\in\X^1_w(f,\CC)$ and each black 1-tile $X^1_b\in\X^1_b(f,\CC)$, we have
$$
\int_{f^{-1}(B) \,\cap\, X^1_w} \!  J \,\mathrm{d} m_\phi = m_\phi(B),
$$
$$
\int_{f^{-1}(C)\,\cap\, \inte (X^1_b) } \!  J \,\mathrm{d} m_\phi  = m_\phi(C),
$$
by Definition~\ref{defJacobian} and Theorem~\ref{thmMexists}. Then the first inequality follows from the fact that $\card \(\X^1_w(f,\CC)\) = \card \(\X^1_b(f,\CC)\) = \deg f$ (see (\ref{eqCardBlackNTiles})).
\end{proof}

\begin{prop}    \label{propMisGibbsState}
Let $f$, $\CC$, $d$, $\phi$, $\alpha$ satisfy the Assumptions. Let $m_\phi$ be a Borel probability measure defined in Theorem~\ref{thmMexists} which satisfies $\RR_\phi^* (m_\phi) = c m_\phi$ where $c = \langle \RR_\phi^*(m_\phi),\mathbbm{1} \rangle$. Then $m_\phi$ is a Gibbs state with respect to $f$, $\CC$, and $\phi$, with 
\begin{equation}  \label{eqLogc=P}
P_{m_\phi} = \log c = \lim\limits_{n\to+\infty}  \frac{1}{n} \log \RR_\phi^n(\mathbbm{1})(y),
\end{equation}
for each $y\in S^2$.
\end{prop}

In particular, since the existence of $m_\phi$ in Theorem~\ref{thmMexists} is independent of $\CC$, this proposition asserts that $m_\phi$ is a Gibbs state with respect to $f$, $\CC$, and $\phi$, for each $\CC$ that satisfies the Assumptions. In general, it is not clear that a Gibbs state with respect to $f$, $\CC_1$, and $\phi$ is also a Gibbs state with respect to $f$, $\CC_2$, and $\phi$, even though the answer is positive in the case when $f$ has no periodic critical points as shown in Corollary~\ref{corGibbsIndepCurves}.

\begin{proof}
We first need to prove that $\mu = m_\phi$ satisfies (\ref{eqGibbsState}).

We observe that
\begin{equation}    \label{eqf^nJacobianM}
m_\phi(f^i(B)) = \int_{B}\!  \exp(i \log c -S_i\phi(x)) \,\mathrm{d} m_\phi(x)
\end{equation}
for $n\in\N$, $i\in\{0,1,\dots,n\}$, and each Borel set $B\subseteq S^2$ on which $f^n$ is injective. Indeed, by the formula for the Jacobian in Theorem~\ref{thmMexists}, for a given Borel set $A\subseteq S^2$ on which $f$ is injective, we have
$$
\int_{f(A)}\!  g(x) \,\mathrm{d} m_\phi(x) = \int_{A}\! (g\circ f)(x) \exp( \log c - \phi(x)) \,\mathrm{d} m_\phi(x)
$$
for each simple function $g$ on $S^2$, thus also for each integrable function $g$. We establish (\ref{eqf^nJacobianM}) for each $n\in\N$ and each Borel set $B\subseteq S^2$ on which $f^n$ is injective by induction on $i$. For $i=0$, equation (\ref{eqf^nJacobianM}) holds trivially. Assume that (\ref{eqf^nJacobianM}) is established for some $i\in\{0,1,\dots,n-1\}$, then since $f^i$ is injective on $f(B)$, we get
\begin{align*}
m_\phi(f^{i+1}(B)) & = \int_{f(B)}\!  \exp(i \log c -S_i\phi(x)) \,\mathrm{d} m_\phi(x) \\
                   & =  \int_{B}\!  \exp((i+1) \log c -S_{i+1}\phi(x)) \,\mathrm{d} m_\phi(x).
\end{align*}
The induction is now complete. In particular, by Proposition~\ref{propCellDecomp}(i),
$$
m_\phi(f^n(X^n)) = \int_{X^n}\!  \exp(n \log c -S_n\phi(x)) \,\mathrm{d} m_\phi(x),
$$
for $n\in \N$ and $X^n\in\X^n(f,\CC)$.
 
Thus by Lemma~\ref{lmSnPhiBound}, there exists a constant $C\geq 1$ such that for each $n\in \N_0$, each $X^n\in\X^n(f,\CC)$, and each $x\in X^n$,
$$
m_\phi (f^n(X^n)) \geq C^{-1} \exp(n \log c -S_n\phi(x)) m_\phi(X^n)
$$
and
$$
m_\phi(f^n(X^n)) \leq C\exp(n \log c -S_n\phi(x)) m_\phi(X^n).
$$
Note that $f^n(X^n)$ is either the black 0-tile $X^0_b\in\X^0(f,\CC)$ or the white 0-tile $X^0_w\in\X^0(f,\CC)$. Both $X^0_b$ and $X^0_w$ are of positive $m_\phi$-measure, for otherwise, suppose that $m_\phi(X^0)=0$ for some $X^0\in\X^0(f,\CC)$, then by Proposition~\ref{propMfABounds}, $m_\phi(f^j(X^0))=0$, for each $j\in\N$. Then by Lemma~\ref{lmTileInIntTile}, $m_\phi(S^2)=0$, a contradiction. Hence (\ref{eqGibbsState}) follows, and $m_\phi$ is a Gibbs state with respect to $f$, $\CC$, and $\phi$, with $P_{m_\phi} = \log c$.

To finish the proof, we note that by (\ref{eqR^nExpr}) and Lemma~\ref{lmSigmaExpSnPhiBound}, for each $x,y\in S^2$ and each $n\in\N_0$, we have
\begin{equation}  \label{eqRR1Quot}
\frac{1}{C_2} \leq \frac{\RR_\phi^n(\mathbbm{1})(x)}{\RR_\phi^n(\mathbbm{1})(y)} \leq C_2,
\end{equation}
where $C_2$ is a constant depending only $f$, $d$, $\phi$, and $\alpha$ from Lemma~\ref{lmSigmaExpSnPhiBound}. 
Since $\langle m_\phi , \RR_\phi^n(\mathbbm{1})\rangle = \langle(\RR_\phi^*)^n(m_\phi), \mathbbm{1} \rangle = \langle c^n m_\phi,\mathbbm{1} \rangle = c^n$, by (\ref{eqR^nExpr}) and (\ref{eqRR1Quot}), we have that for each arbitrarily chosen $y\in S^2$,
\begin{align} \label{eqLogc=SumPreImg}
\log c & = \lim\limits_{n\to+\infty} \frac{1}{n} \log \int \!\RR_\phi^n(\mathbbm{1})(x)\,\mathrm{d}m_\phi(x)   \notag  \\
       & = \lim\limits_{n\to+\infty} \frac{1}{n} \log \int \!\RR_\phi^n(\mathbbm{1})(y)\,\mathrm{d}m_\phi(x) \\  
       & = \lim\limits_{n\to+\infty} \frac{1}{n} \log \RR_\phi^n(\mathbbm{1})(y). \notag
\end{align}
\end{proof}

\begin{cor} \label{corLimitPxExists}
Let $f$, $d$, $\phi$, $\alpha$ satisfy the Assumptions. Then the limit $\lim\limits_{n\to+\infty}\frac{1}{n}\log\RR_\phi^n(\mathbbm{1})(x)$ exists for each $x\in S^2$ and is independent of $x\in S^2$. 
\end{cor}

We denote the limit as $D_\phi \in \R$.

\begin{proof}
By Theorem~\ref{thmMexists} , there exists a measure $m_\phi$ such as the one in Proposition~\ref{propMisGibbsState}. The limit then clearly only depends on $f$, $d$, $\phi$, and $\alpha$, and in particular, does not depend on $\CC$ or the choice of $m_\phi$.
\end{proof}

Let $f$, $\CC$, $d$, $\phi$, $\alpha$ satisfy the Assumptions. We define the function 
\begin{equation}   \label{eqDefPhi-}
\overline{\phi}= \phi - D_\phi \in \Holder{\alpha}(S^2,d).
\end{equation}
Then
\begin{equation}  \label{eqDefR-}
\RR_{\overline{\phi}}= e^{-D_\phi}\RR_\phi.
\end{equation}
If $m_\phi$ is a Gibbs state from Theorem~\ref{thmMexists}, then by Proposition~\ref{propMisGibbsState} and Corollary~\ref{corLimitPxExists} we have
\begin{equation}
\RR_{\phi}^*(m_\phi)=e^{D_\phi} m_\phi = e^{P_{m_\phi}} m_\phi,
\end{equation}
and
\begin{equation}   \label{eqR*m}
\RR_{\overline{\phi}}^*(m_\phi)=m_\phi,
\end{equation}
since for each $u\in\CCC(S^2)$, 
\begin{align*}
\langle \RR_{\overline{\phi}}^*(m_\phi), u \rangle= & \langle m_\phi, \RR_{\overline{\phi}}(u)\rangle= e^{-D_\phi } \langle m_\phi, \RR_\phi(u)\rangle \\
             =  & e^{-D_\phi} \langle \RR_\phi^*(m_\phi), u\rangle = \langle m_\phi,u\rangle.
\end{align*}

We summarize in the following lemma the properties of $\RR_{\overline{\phi}}$ that we will need.
\begin{lemma}   \label{lmR1properties}
Let $f$, $\CC$, $d$, $L$, $\Lambda$, $\phi$, $\alpha$ satisfy the Assumptions. Then there exists a constant $C_3$ depending only on $f$, $d$, $\phi$, and $\alpha$ such that for each $x,y\in S^2$ and each $n\in \N_0$ the following equations are satisfied
\begin{equation}  \label{eqR1Quot}
\frac{\RR_{\overline{\phi}}^n(\mathbbm{1})(x)}{\RR_{\overline{\phi}}^n(\mathbbm{1})(y)} \leq \exp\(4C_1 Ld(x,y)^\alpha\) \leq C_2,
\end{equation}
\begin{equation}   \label{eqR1Bound}
\frac{1}{C_2} \leq \RR_{\overline{\phi}}^n(\mathbbm{1})(x)  \leq C_2,
\end{equation}
\begin{align}    \label{eqR1Diff}
       &  \Abs{\RR_{\overline{\phi}}^n(\mathbbm{1})(x) - \RR_{\overline{\phi}}^n(\mathbbm{1})(y) }  \\
\leq   & C_2 \( \exp\(4C_1 Ld(x,y)^\alpha\) - 1 \)  \leq C_3 d(x,y)^\alpha,   \notag
\end{align}
where $C_1, C_2$ are constants in Lemma~\ref{lmSnPhiBound} and Lemma~\ref{lmSigmaExpSnPhiBound} depending only on $f$, $d$, $\phi$, and $\alpha$.
\end{lemma}

\begin{proof}
Inequality (\ref{eqR1Quot}) follows from (\ref{eqDefR-}), (\ref{eqR^nExpr}), and Lemma~\ref{lmSigmaExpSnPhiBound}.

To prove (\ref{eqR1Bound}), we choose a Gibbs state $m_\phi$ with respect to $f$, $\CC$, and $\phi$ from Theorem~\ref{thmMexists}. Then by (\ref{eqR*m}) and (\ref{eqR1Quot}), we have
$$
\RR_{\overline{\phi}}^n(\mathbbm{1})(x)\leq C_2 \big\langle m_\phi, \RR_{\overline{\phi}}^n(\mathbbm{1}) \big\rangle  = C_2 \big\langle (\RR_{\overline{\phi}}^*)^n (m_\phi), \mathbbm{1} \big\rangle = C_2 \langle m_\phi,\mathbbm{1}\rangle = C_2.
$$
The first inequality in (\ref{eqR1Bound}) can be proved similarly.

Applying (\ref{eqR1Quot}) and (\ref{eqR1Bound}), we get
\begin{align*}
\RR_{\overline{\phi}}^n(\mathbbm{1})(x) - \RR_{\overline{\phi}}^n(\mathbbm{1})(y)  & = \( \frac{\RR_{\overline{\phi}}^n(\mathbbm{1})(x)}{\RR_{\overline{\phi}}^n(\mathbbm{1})(y)} - 1 \) \RR_{\overline{\phi}}^n(\mathbbm{1})(x)   \\
&  \leq C_2 \( \exp\(4C_1 Ld(x,y)^\alpha\) - 1 \) \\
&  \leq C_3 d(x,y)^\alpha,
\end{align*}
for some constant $C_3$ depending only on $L$, $C_1$, $C_2$, and $\diam_d(S^2)$.
\end{proof}

We can now prove the existence of an $f$-invariant Gibbs state.

\begin{theorem}  \label{thmMuExist}
Let $f\:S^2 \rightarrow S^2$ be an expanding Thurston map and $\CC \subseteq S^2$ be a Jordan curve containing $\post f$ with the property that $f^l(\CC)\subseteq \CC$ for some $l\in\N$. Let $d$ be a visual metric on $S^2$ for $f$ with an expansion factor $\Lambda>1$. Let $\phi\in \Holder{\alpha}(S^2,d)$ be a real-valued H\"{o}lder continuous function with an exponent $\alpha\in(0,1]$. Then the sequence $\Big\{\frac{1}{n}\sum\limits_{j=0}^{n-1} \RR_{\overline{\phi}}^j(\mathbbm{1})\Big\}_{n\in\N}$ converges uniformly to a function $u_\phi \in \Holder{\alpha}(S^2,d)$, which satisfies
\begin{equation}    \label{eqRu=u}
\RR_{\overline{\phi}}(u_\phi)  = u_\phi,
\end{equation}
and
\begin{equation}    \label{eqU_phiBounds}
\frac{1}{C_2} \leq u_\phi(x) \leq C_2, \qquad  \text{for each } x\in S^2,
\end{equation}
where $C_2\geq 1$ is a constant from Lemma~\ref{lmSigmaExpSnPhiBound}. Moreover, if we let $m_\phi$ be a Gibbs state from Theorem~\ref{thmMexists}, then 
\begin{equation}    \label{eqSudm=1}
\int \! u_\phi \, \mathrm{d}m_\phi =1,
\end{equation}
and $\mu_\phi=u_\phi m_\phi$ is a Gibbs state with respect to $f$, $\CC$, and $\phi$, with 
\begin{equation}   \label{eqP=L}
P_{\mu_\phi}=P_{m_\phi}=D_\phi=\lim\limits_{n\to+\infty}  \frac{1}{n} \log \RR_\phi^n(\mathbbm{1})(y),
\end{equation}
for each $y\in S^2$, and  
\begin{equation}   \label{eqf*mu=mu}
f_*(\mu_\phi)=\mu_\phi.
\end{equation}
\end{theorem}

\begin{proof}
In order to prove this theorem, we first establish (\ref{eqRu=u}), (\ref{eqU_phiBounds}), and (\ref{eqSudm=1}) for a subsequential limit of the sequence $\Big\{\frac{1}{n}\sum\limits_{j=0}^{n-1} \RR_{\overline{\phi}}^j(\mathbbm{1})\Big\}_{n\in\N}$, then prove the above sequence has a unique subsequential limit, and finally justify (\ref{eqP=L}) and (\ref{eqf*mu=mu}).

Define, for each $n\in\N$, $u_n =  \frac{1}{n}\sum\limits_{j=0}^{n-1} \RR_{\overline{\phi}}^j(\mathbbm{1})$. Then $\{u_n\}_{n\in\N}$ is a uniformly bounded sequence of equicontinuous functions on $S^2$ by (\ref{eqR1Bound}) and (\ref{eqR1Diff}). By the Arzel\`a-Ascoli Theorem, there exists a continuous function $u_\phi \in \CCC(S^2)$ and an increasing sequence $\{n_i\}_{i\in \N}$ in $\N$ such that $u_{n_i} \rightarrow u_\phi$ uniformly on $S^2$ as $i\longrightarrow +\infty$.

To prove (\ref{eqRu=u}), we note that by the definition of $u_n$ and (\ref{eqR1Bound}), we have that for each $i\in\N$,
$$
\Norm{\RR_{\overline{\phi}} (u_{n_i}) - u_{n_i}}_{\infty} = \frac{1}{n_i}\Norm{\RR_{\overline{\phi}}^{n_i} (\mathbbm{1}) - \mathbbm{1}}_{\infty} \leq \frac{1+C_2}{n_i}.
$$
By letting $i\longrightarrow +\infty$, we can conclude that $\Norm{\RR_{\overline{\phi}} (u_\phi) - u_\phi}_{\infty}=0$. Thus (\ref{eqRu=u}) holds.

By (\ref{eqR1Bound}), we have that $C_2^{-1} \leq u_n(x) \leq C_2$, for each $n\in \N$ and each $x\in S^2$. Thus (\ref{eqU_phiBounds}) follows.

By (\ref{eqR*m}) and definition of $u_n$, we have $\int\! u_n \,\mathrm{d}m_\phi = \int \!\mathbbm{1} \,\mathrm{d}m_\phi  = 1$ for each $n\in\N$. Then by the Lebesgue Dominated Convergence Theorem, we can conclude that 
$$
\int \!u_\phi\,\mathrm{d}m_\phi= \lim\limits_{i\to+\infty} \int\! u_{n_i} \,\mathrm{d}m_\phi =1,
$$
proving (\ref{eqSudm=1}).

Next, we prove that $u_\phi$  is the unique subsequential limit of the sequence $\{u_n\}_{n\in\N}$ with respect to the uniform norm. Suppose that $v_\phi$ is another subsequential limit of $u_n, n\in\N$, with respect to the uniform norm. Then $v_\phi$ is also a continuous function on $S^2$ satisfying (\ref{eqRu=u}), (\ref{eqU_phiBounds}), and (\ref{eqSudm=1}) by the argument above. Let 
$$
t=\sup \{s\in\R  \,|\,  u_\phi(x) - s v_\phi(x) > 0  \text{ for all } x\in S^2  \}.
$$
By (\ref{eqU_phiBounds}), $t\in(0,+\infty)$. Then there is a point $y\in S^2$ such that $u_\phi(y)-t v_\phi(y) = 0$. By (\ref{eqR^nExpr}) and the equation
$$
\RR_{\overline{\phi}} (u_\phi - t v_\phi) = u_\phi - t v_\phi,
$$
which comes from (\ref{eqRu=u}), we get that $u_\phi(z)-t v_\phi(z)=0$ for all $z\in f^{-1}(y)$. Inductively, we can conclude that $u_\phi(z)-t v_\phi(z)=0$ for all $z\in \bigcup\limits_{i\in\N} f^{-i}(y)$. By Lemma~\ref{lmPreImageDense}, the set $\bigcup\limits_{i\in\N} f^{-i}(y)$ is dense in $S^2$. Hence $u_\phi = tv_\phi$ on $S^2$. Since both $u_\phi$ and $v_\phi$ satisfy (\ref{eqSudm=1}), we get $t=1$. Thus $u_\phi=v_\phi$. We have proved that $u_n$ converges to $u_\phi$ uniformly as $n\longrightarrow +\infty$.

We now prove that $u_\phi \in \Holder{\alpha}(S^2,d)$. Indeed, for each $\epsilon>0$, there exists $n\in\N$ such that $\Norm{u_n-u_\phi}_{\infty} < \epsilon$. Then by (\ref{eqR1Diff}), for each $x,y\in S^2$, we have
\begin{align*}
      & \Abs{u_\phi(x) - u_\phi(y)} \\
 \leq & \Abs{u_\phi(x)-u_n(x)}+\Abs{u_n(x)-u_n(y)}+\Abs{u_n(y)-u_\phi(y)}   \\
 \leq & 2\epsilon + \frac{1}n \sum\limits_{j=0}^{n-1} \Abs{\RR_{\overline{\phi}}^j(\mathbbm{1})(x) - \RR_{\overline{\phi}}^j(\mathbbm{1})(y)}   \\
 \leq & 2\epsilon + C_3 d(x,y)^\alpha,                        
\end{align*}
where $C_3$ is a constant in (\ref{eqR1Diff}) from Lemma~\ref{lmR1properties}. By letting $\epsilon \longrightarrow 0$, we conclude that $u_\phi \in \Holder{\alpha}(S^2,d)$.

Since $m_\phi$ is a Gibbs state by Proposition~\ref{propMisGibbsState}, then by (\ref{eqU_phiBounds}), $\mu_\phi = u_\phi m_\phi$ is also a Gibbs state with $P_{\mu_\phi}=P_{m_\phi}=D_\phi=\lim\limits_{n\to+\infty}  \frac{1}{n} \log \RR_\phi^n(\mathbbm{1})(y)$ for each $y\in S^2$, proving (\ref{eqP=L}).

Finally we need to prove that $\mu_\phi$ is $f$-invariant. It suffices to prove that $\langle \mu_\phi, g\circ f\rangle = \langle \mu_\phi,g\rangle$ for each $g\in\CCC(S^2)$. Indeed, by (\ref{eqR*m}), (\ref{eqRu=u}), and (\ref{eqRuvf}), we get
\begin{align*}
 \langle \mu_\phi, g\circ f \rangle 
 &= \langle m_\phi, u_\phi(g\circ f) \rangle
= \big\langle \RR_{\overline\phi}^* (m_\phi), u_\phi(g\circ f) \big\rangle\\
 &=\big\langle m_\phi,  \RR_{\overline\phi} (u_\phi(g\circ f)) \big\rangle  
= \big\langle m_\phi, g \RR_{\overline\phi} (u_\phi) \big\rangle \\
 &= \langle m_\phi, g u_\phi \rangle
= \langle \mu_\phi, g \rangle.
\end{align*}
\end{proof}

\begin{remark}
By a similar argument to that in the proof of the uniqueness of the subsequential limit of $\Big\{\frac{1}{n}\sum\limits_{j=0}^{n-1} \RR_{\overline{\phi}}^j(\mathbbm{1})\Big\}_{n\in\N}$, one can show that $u_\phi$ is the unique eigenfunction, upto scalar multiplication, of $\RR_{\overline{\phi}}$ corresponding to the eigenvalue $1$.
\end{remark}

We now get the following characterization of the topological pressure $P(f,\phi)$ of an expanding Thurston map $f$ with respect to a H\"older continuous potential $\phi$.

\begin{prop}   \label{propTopPressureDefPreImg}
Let $f$, $d$, $\phi$, $\alpha$ satisfy the Assumptions. Then for each $x\in S^2$, we have
\begin{equation}  \label{eqTopPressureDefPreImg}
P(f,\phi)= \lim\limits_{n\to +\infty} \frac{1}{n} \log \sum\limits_{y\in f^{-n}(x)} \deg_{f^n}(y) \exp (S_n\phi(y)) = D_\phi.
\end{equation}
\end{prop}

Recall that $D_\phi = P_{m_\phi} = P_{\mu_\phi} = \log c = \log \int\! \RR_\phi(\mathbbm{1})\,\mathrm{d}m_\phi$, using the notation from Proposition~\ref{propMisGibbsState} and Theorem~\ref{thmMuExist}.

\begin{proof}
We pick a Jordan curve $\CC\subseteq S^2$ that satisfies the Assumptions (see Theorem~\ref{thmCexistsBM} for the existence of such $\CC$).

By Corollary~\ref{corLimitPxExists} and (\ref{eqR^nExpr}), for each $x\in S^2$, the limit in (\ref{eqTopPressureDefPreImg}) always exists and is equal to $D_\phi$, independent of $x$. Moreover, for an $f$-invariant Gibbs measure $\mu_\phi$ from Theorem~\ref{thmMuExist} with $P_{\mu_\phi} = D_\phi$, we get from Proposition~\ref{propInvGibbsIsEqlbStatePLessThanPressure} that
\begin{equation}
D_\phi = P_{\mu_\phi} \leq P(f,\phi).
\end{equation}

Now it suffices to prove $D_\phi \geq P(f,\phi)$.

Note that by Lemma~\ref{lmCellBoundsBM}(ii), there is a constant $C\geq 1$ depending only on $f$, $d$, and $\CC$ such that for each $n\in\N_0$ and each $n$-tile $X^n\in\X^n(f,\CC)$, we have $C^{-1}\Lambda^{-n}\leq \diam_d(X^n) \leq C\Lambda^{-n}$.

Fix $m\in\N$, let $\epsilon = C\Lambda^{-m}$. For each $n \in\N_0$, let $F_n(m)$ be a maximal $(n,\epsilon)$-separated subset of $S^2$.

We claim that if $y_1,y_2\in F_n(m)$ and $y_1,y_2\in X^{m+n}$ for some $(m+n)$-tile $X^{m+n}$ in $\X^{m+n}(f,\CC)$, then $y_1=y_2$.

Indeed, for each integer $j\in [0,n-1]$, we have that 
\begin{equation}
d(f^j(y_1),f^j(y_2)) \leq \diam_d(f^j(X^{m+n})) \leq C\Lambda^{-(m+n-j)} < \epsilon.
\end{equation}
So suppose that $y_1 \neq y_2$, then $y_1,y_2$ would not be $(n,\epsilon)$-separated, a contradiction.

We fix $x\in \inte (X_w^0)$ and $y\in \inte (X_b^0)$ where $X_w^0$ and $X_b^0$ are the white $0$-tile and black $0$-tile in $\X^0(f,\CC)$, respectively. We can now construct an injective map $i_n\:F_n(m) \rightarrow f^{-(m+n)}(x)\cup f^{-(m+n)}(y)$ for each $n\in \N$ by demanding that $z\in F_n(m)$ and $i_n(z)\in f^{-(m+n)}(x)\cup f^{-(m+n)}(y)$ be in the same $(m+n)$-tile. Since for each $X^{m+n} \in \X^{m+n}(f,\CC)$, $\card\(X^{m+n} \cap (f^{-(m+n)}(x) \cup f^{-(m+n)}(y)) \) = 1$, it follows that $i_n$ is well-defined (but not necessarily uniquely defined) for each $n\in\N$. Thus by Lemma~\ref{lmSnPhiBound} and Lemma~\ref{lmSigmaExpSnPhiBound}, we have that for each $n\in\N$,
\begin{align*}
& \sum\limits_{z\in F_n(m)} \exp (S_n\phi(z))  \\
\leq & C_4 \sum\limits_{z\in f^{-(m+n)}(x)\cup f^{-(m+n)}(y)} \exp (S_n\phi(z)) \\
\leq & C_4 e^{m \Norm{\phi}_{\infty}}\bigg( \sum\limits_{z\in f^{-(m+n)}(x)} \exp (S_{m+n}\phi(z)) + \sum\limits_{z\in f^{-(m+n)}(y)} \exp (S_{m+n}\phi(z)) \bigg)\\
\leq & C_4(1+C_2) \exp(m \Norm{\phi}_{\infty}) \sum\limits_{z\in f^{-(m+n)}(x)} \exp (S_{m+n}\phi(z)),
\end{align*}
where $C_4 = \exp \(C_1 \big(\diam_d(S^2)\big)^\alpha\)$, and $C_1,C_2$ are constants from Lemma~\ref{lmSnPhiBound} and Lemma~\ref{lmSigmaExpSnPhiBound}.  By taking logarithm and next dividing by $n$ on both sides, then taking $n\longrightarrow +\infty$ and finally taking $m\longrightarrow +\infty$ to make $\epsilon \longrightarrow 0$, we get from (\ref{defTopPressure}) that
\begin{align*}
P(f,\phi) & = \lim\limits_{m\to +\infty} \liminf_{n\to +\infty} \frac{1}{n} \log \sum\limits_{w\in F_n(m)} \exp (S_n\phi(w))  \\
          & \leq  \limsup\limits_{m\to +\infty} \liminf_{n\to +\infty} \frac{1}{n} \log \sum\limits_{z\in f^{-(m+n)}(x)} \exp (S_{m+n}\phi(z))  \\
          & = \limsup\limits_{m\to +\infty} \liminf_{n\to +\infty} \frac{1}{m+n} \log \sum\limits_{z\in f^{-(m+n)}(x)} \exp (S_{m+n}\phi(z))  \\
          & = \limsup\limits_{m\to +\infty} \liminf_{n\to +\infty} \frac{1}{n} \log \sum\limits_{z\in f^{-n}(x)} \exp (S_n\phi(z)) \\
          & = D_\phi,
\end{align*}
where the last equality follows from Corollary~\ref{corLimitPxExists}, (\ref{eqR^nExpr}), and the fact that $x\notin \post f$.
\end{proof}

The following corollary gives the existence part of Theorem~\ref{thmMain}.

\begin{cor}   \label{corExistES}
Let $f\:S^2 \rightarrow S^2$ be an expanding Thurston map and $d$ be a visual metric on $S^2$ for $f$. Let $\phi\in \Holder{\alpha}(S^2,d)$ be a real-valued H\"{o}lder continuous function with an exponent $\alpha\in(0,1]$. Then there exists an equilibrium state for $f$ and $\phi$. In fact, any measure $\mu_\phi$ defined in Theorem~\ref{thmMuExist} is an equilibrium state for $f$ and $\phi$.
\end{cor}

\begin{proof}
We pick a Jordan curve $\CC\subseteq S^2$ that satisfies the Assumptions (see Theorem~\ref{thmCexistsBM} for the existence of such $\CC$). Consider an $f$-invariant Gibbs state $\mu_\phi$ with respect to $f$, $\CC$, and $\phi$ from Theorem~\ref{thmMuExist}. Then by Theorem~\ref{thmMuExist} and Proposition~\ref{propTopPressureDefPreImg}, we have $P_{\mu_\phi} = D_\phi = P(f,\phi)$. Then by Proposition~\ref{propInvGibbsIsEqlbStatePLessThanPressure}, we have $P_{\mu_\phi} = h_{\mu_\phi} + \int \! \phi \,\mathrm{d}\mu_\phi = P(f,\phi)$. Therefore, $\mu_\phi$ is an equilibrium state for $f$ and $\phi$.
\end{proof}

We end this section by proving in Proposition~\ref{propRadialGibbsIFFGibbs} that the concept of a Gibbs state and that of a radial Gibbs state coincide if and only if the expanding Thurston map has no periodic critical point. The proof of the forward implication relies on the existence of Gibbs states for $f$, $\CC$, and $\phi$ that satisfy the Assumptions proved in Proposition~\ref{propMisGibbsState}.

\begin{prop}   \label{propRadialGibbsIFFGibbs}
Let $f$, $\CC$, $d$, $\phi$, $\alpha$ satisfy the Assumptions. Then a radial Gibbs state $\mu$ with respect to $f$, $d$, and $\phi$ must be a Gibbs state with respect to $f$, $\CC$, and $\phi$, with $\widetilde P_\mu = P_\mu$. Moreover, the following are equivalent:

\begin{enumerate}
\smallskip

\item $f$ has no periodic critical point.

\smallskip

\item A Borel probability measure $\mu\in\PPP(S^2)$ is a Gibbs state with respect to $f$, $\CC$, and $\phi$ if and only if it is a radial Gibbs state with respect to $f$, $d$, and $\phi$.

\smallskip

\item There exists a radial Gibbs state for $f$, $d$, and $\phi$.
\end{enumerate}
\end{prop}

The implication from (1) to (2) generalizes Proposition 20.10 in \cite{BM10}, which states that for an expanding Thurston map $f$ with no periodic critical point and with the measure of maximal entropy $\mu$ and a visual metric $d$, the metric measure space $(S^2, d, \mu)$ is Ahlfors regular.

\begin{proof}
By Lemma~\ref{lmCellBoundsBM}(v), there exists a constant $C\geq 1$ such that for each $n\in\N_0$, and each $n$-tile $X^n\in\X^n$, there exists a point $p\in X^n$ with
\begin{equation*}
B_d(p,C^{-1}\Lambda^{-n}) \subseteq X^n \subseteq B_d(p,C\Lambda^{-n}).
\end{equation*}
Thus there exists $m_1\in\N$ such that for each $n\in\N_0$, each $X^n\in\X^n$, there exists $p\in X^n$ such that 
\begin{equation}
B_d\(p,\Lambda^{-(n+m_1)}\) \subseteq X^n \subseteq B_d \(p,\Lambda^{-(n-m_1)}\).
\end{equation}
On the other hand, by Lemma~\ref{lmCellBoundsBM}(iv), there exists $m_2\in \N$ such that for each $x\in S^2$ and each $n\in \N_0$, we have
\begin{equation}  \label{eqPfRadialGibbsIFFGibbs1}
U^{n+m_2} (x) \subseteq  B_d(x,\Lambda^{-n})  \subseteq U^{n-m_2} (x),
\end{equation}
where the sets $U^l(x)$ for $l\in\N_0$ and $x\in S^2$ are defined in (\ref{defU^n}).

Note that for each $n\in\N_0$ and each $y\in U^n (x)$,  by choosing $z\in Y^n\cap X^n$ with $X^n,Y^n\in\X^n$ and $x\in X^n$, $y\in Y^n$, and applying Lemma~\ref{lmSnPhiBound}, we get
\begin{align*}
\abs{S_n\phi(x) - S_n\phi(y)} & \leq \abs{S_n\phi(x) - S_n\phi(z)} + \abs{S_n\phi(z) - S_n\phi(y)} \\
                              & \leq 2C_1 \(\diam_d(S^2)\)^\alpha,
\end{align*}
where $C_1$ is a constant from Lemma~\ref{lmSnPhiBound}.

\smallskip

If $\mu$ is a radial Gibbs state with constants $\widetilde{P}_\mu$ and $\widetilde{C}_\mu$, then for each $n\in\N_0$ and each $n$-tile $X^n\in\X^n$, there exists $p\in X^n$ such that
\begin{align*}
\mu(X^n)  & \leq \mu\( B_d\(p,\Lambda^{-(n-m_1)}\) \)   \\
          & \leq \widetilde{C}_\mu  \exp \big(S_{n-m_1} \phi(x)  - (n-m_1)\widetilde{P}_\mu \big) \\
          & \leq \widetilde{C}_\mu  \exp \big(m_1\Norm{\phi}_\infty + m_1 \widetilde{P}_\mu\big) \exp \big(S_{n} \phi(x)  - n\widetilde{P}_\mu \big),
\end{align*}
and
\begin{align*}
\mu(X^n)  & \geq \mu\( B_d\(p,\Lambda^{-(n+m_1)}\) \)   \\
          & \geq \frac{1}{\widetilde{C}_\mu}  \exp \big(S_{n+m_1} \phi(x)  - (n+m_1)\widetilde{P}_\mu \big) \\
          & \geq \frac{1}{\widetilde{C}_\mu \exp \big(m_1\Norm{\phi}_\infty + m_1 \widetilde{P}_\mu\big)}   \exp \big(S_{n} \phi(x)  - n\widetilde{P}_\mu \big).
\end{align*}
Thus $\mu$ is a Gibbs state for $f$, $\CC$, and $\phi$, with $P_\mu = \widetilde P_\mu$.

\smallskip
To prove the equivalence, we start with the implication from (1) to (2).

We have already shown above that any radial Gibbs state for $f$, $d$, and $\phi$ must be a Gibbs state for $f$, $\CC$, and $\phi$.

If we assume that $f$ has no periodic critical point, then there exists a constant $K\in\N$ such that for each $x\in S^2$ and each $n\in\N_0$, the set $U^n(x)$ is a union of at most $K$ distinct $n$-tiles, i.e.,
\begin{align*}
\card \{Y^n\in\X^n  \,|\,  & \text{there exists an $n$-tile } X^n\in \X^n \text{ with } \\
                           & x\in X^n \text{ and } X^n \cap Y^n \neq \emptyset\} \leq K.
\end{align*}
Indeed, if $f$ has no periodic critical point, then there exists a constant $N\in\N$ such that $\deg_{f^n}(x)\leq N$ for all $x\in S^2$ and all $n\in\N$ (\cite[Lemma~17.1]{BM10}). Since each $n$-flower $W^n(p)$ for $p\in\V^n$ is covered by exactly $2\deg_{f^n}(p)$ distinct $n$-tiles (\cite[Lemma~7.2(i)]{BM10}), $U^n(x)$ is covered by a bounded number of $n$-flowers and thus covered by a bounded number, independent of $x\in S^2$ and $n\in\N_0$, of distinct $n$-tiles.

\smallskip

If $\mu$ is a Gibbs state with constants $P_\mu$ and $C_\mu$, then by (\ref{eqPfRadialGibbsIFFGibbs1}) and Lemma~\ref{lmSnPhiBound}, for each $n\in\N_0$ and each $x\in S^2$, we have
\begin{align*}
        \mu\big( B_d(x,\Lambda^{-n}) \big) \geq &  \mu\(U^{n+m_2} (x)\) \geq C_\mu^{-1} \exp( S_{n+m_2} \phi(x) - (n+m_2) P_\mu    )\\
\geq &  \frac{1}{C_\mu \exp( m_2 \Norm{\phi}_\infty + m_2P_\mu)}   \exp (S_n\phi(x) - nP_\mu),
\end{align*}
and moreover, if $n\geq m_2$, then
\begin{align*} 
 & \mu\big( B_d (x,\Lambda^{-n}) \big)  \leq  \mu\(U^{n-m_2} (x)\) \leq \sum\limits_{\substack{X\in\X^{n-m_2} \\  X\subseteq U^{n-m_2}(x)}}  \mu(X)   \\
\leq &  K C_\mu \exp\(2C_1\(\diam_d(S^2)\)^\alpha\) \exp( S_{n-m_2} \phi(x) - (n-m_2) P_\mu    )\\
\leq &  K C_\mu \exp\(2C_1\(\diam_d(S^2)\)^\alpha + m_2 \(\Norm{\phi}_\infty + P_\mu\)\)  \exp (S_n\phi(x) - nP_\mu),
\end{align*}
and if $n < m_2$, then
\begin{equation*}
\mu( B_d (x,\Lambda^{-n}))  \leq 1 \leq \exp\( m_2(\Norm{\phi}_\infty + P_\mu ) \)  \exp(S_n\phi(x) - n P_\mu).
\end{equation*}
Thus $\mu$ is a radial Gibbs state for $f$, $d$, and $\phi$.

\smallskip

Next, we show that (2) implies (3). 

We assume (2) now. Let $\mu= m_\phi$, where $m_\phi$ is from Theorem~\ref{thmMexists}. Then by Proposition~\ref{propMisGibbsState}, $\mu$ is a Gibbs state for $f$, $\CC$, and $\phi$. Thus $\mu$ is also a radial Gibbs state for $f$, $d$, and $\phi$.

\smallskip

Finally, we prove the implication from (3) to (1) by contradiction.

Assume that $f$ has a periodic critical point $x\in S^2$ with a period $l\in\N$, and let $\mu$ be a radial Gibbs state for $f$, $d$, and $\phi$ with constants $\widetilde{P}_\mu$ and $\widetilde{C}_\mu$. So $\mu$ is also a Gibbs state for $f$, $\CC$, and $\phi$ with constants $P_\mu = \widetilde{P}_\mu$ and $C_\mu$, as shown in the first part of the proof.

We note that $x\in\post f \subseteq \V^n$ for each $n\in\N_0$. By (\ref{defFlower}), (\ref{defU^n}), and (\ref{eqPfRadialGibbsIFFGibbs1}), for each $n\in\N$,
\begin{equation*}
\overline{W^{nl+m_2} (x)} \subseteq  U^{nl+m_2} (x)  \subseteq B_d (x,\Lambda^{-nl}).
\end{equation*}
Recall that the number of distinct $(nl+m_2)$-tiles contained in $\overline{W^{nl+m_2}(x)}$ is $2\deg_{f^{nl+m_2}}(x)$. Denote these $(nl+m_2)$-tiles by $X^{nl+m_2}_i \in\X^{nl+m_2}$, $i\in\big\{1,2,\dots,2\deg_{f^{nl+m_2}}(x)\big\}$. Then by Lemma~\ref{lmTileInIntTile}, there exists an $(nl+m_2+M)$-tile $Y_i\in\X^{nl+m_2+M}$ such that $Y_i\subseteq \inte \(X^{nl+m_2}_i\)$. Here $M\in\N$ is a constant from Lemma~\ref{lmTileInIntTile}. We fix $x_i\in Y_i$ for each $i\in\big\{1,2,\dots,2\deg_{f^{nl+m_2}}(x)\big\}$. Note that $Y_i \cap Y_j = \emptyset$ for $1\leq i<j\leq 2\deg_{f^{nl+m_2}}(x)$. 
Thus
\begin{align*}
     & \widetilde{C}_\mu \exp (S_{nl} \phi (x)  - nlP_\mu)
\geq  \mu\big(B_d(x,\Lambda^{-nl}) \big)  \\
\geq & \mu\Big( \overline{W^{nl+m_2} (x)} \Big)  
\geq  \sum\limits_{i=1}^{2\deg_{f^{nl+m_2}}(x)} \mu(Y_i)\\
\geq & 2\deg_{f^{nl+m_2}} (x) \frac{1}{C_\mu}  \exp \( S_{nl+m_2+M} \phi(x_i) - (nl+m_2+M)P_\mu \)  \\
\geq & \frac{2 \(\deg_{f^l}(x) \)^n}{C_\mu \exp(M\Norm{\phi}_\infty + MP_\mu)}  \exp (S_{nl+m_2}\phi(x_i) - (nl+m_2)P_\mu)\\
\geq & \frac{2 \(\deg_{f^l}(x) \)^n \exp (S_{nl+m_2}\phi(x) - (nl+m_2)P_\mu)}{C_\mu \exp\(M\Norm{\phi}_\infty + MP_\mu + C_1 \(\diam_d(S^2)\)^\alpha\)}  \\
\geq & \frac{ \(\deg_{f^l}(x) \)^n \exp (S_{nl}\phi(x) -  nl P_\mu)}{C_\mu \exp\((m_2+M)\Norm{\phi}_\infty + (m_2+M)P_\mu + C_1 \(\diam_d(S^2)\)^\alpha\)},
\end{align*}
where the second-to-last inequality follows from Lemma~\ref{lmSnPhiBound} and the fact that $x_i,x \in X^{nl+m_2}_i$ for $i\in\big\{1,2,\dots,2\deg_{f^{nl+m_2}}(x)\big\}$, and $C_1$ is a constant from Lemma~\ref{lmSnPhiBound}. So 
$$
\(\deg_{f^l}(x) \)^n \leq  \widetilde{C}_\mu C_\mu \exp\((m_2+M)(\Norm{\phi}_\infty + P_\mu) + C_1 \(\diam_d(S^2)\)^\alpha\),
$$
for each $n\in\N$. However, since $x$ is a periodic critical point of $f$, we have $\deg_{f^l}(x)> 1$, a contradiction.
\end{proof}

As an immediate consequence, we get that if the expanding Thurston map does not have periodic critical points, then the property of being a Gibbs state does not depend on the choice of the Jordan curve $\CC\subseteq S^2$ that satisfies the Assumptions.

\begin{cor}    \label{corGibbsIndepCurves}
Let $f$, $d$, $\phi$, $\alpha$ satisfy the Assumptions. We assume that $f$ does not have periodic critical points. Let $\CC_1$ and $\CC_2$ be Jordan curves on $S^2$ that satisfy the Assumptions for $\CC$, and $\mu\in\PPP(S^2)$ be a Borel probability measure. Then $\mu$ is a Gibbs state with respect to $f$, $\CCC_1$, and $\phi$ if and only if $\mu$ is a Gibbs state with respect to $f$, $\CCC_2$, and $\phi$.
\end{cor}

\begin{proof}
By Proposition~\ref{propRadialGibbsIFFGibbs}, since $f$ does not have periodic critical points, $f$ is a Gibbs state with respect to $f$, $\CC_1$, and $\phi$ if and only if $f$ is a radial Gibbs state with respect to $f$, $d$, and $\phi$ if and only if $f$ is a Gibbs state with respect to $f$, $\CC_2$, and $\phi$.
\end{proof}

\section{Uniqueness}  \label{sctUniqueness}
To prove the uniqueness of the equilibrium state of a continuous map $g$ on a compact metric space $X$, one of the techniques is to prove the (G\^ateaux) differentiability of the topological pressure function $P(g,\cdot)\: \CCC(X)\rightarrow \R$. We summarize the general ideas below, but refer the reader to \cite[Section~3.6]{PU10} for a detailed treatment in the case of forward-expansive maps and distance expanding maps.

For a continuous map $g\:X\rightarrow X$ on a compact metric space $X$, the topological pressure function $P(g,\cdot):\CCC(X)\rightarrow \R$ is Lipschitz continuous (\cite[Theorem~3.6.1]{PU10}) and convex (\cite[Theorem~3.6.2]{PU10}). For an arbitrary convex continuous function $Q\:V\rightarrow \R$ on a real topological vector space $V$, we call a continuous linear functional $L\:V\rightarrow\R$ \defn{tangent to $Q$ at $x\in V$} if 
\begin{equation}  \label{eqDefTangent}
Q(x) + L(y) \leq Q(x+y),  \qquad \text{for each } y\in V.
\end{equation}
We denote the set of all continuous linear functionals tangent to $Q$ at $x\in V$ by $V_{x,Q}^*$. It is known (see for example, \cite[Proposition~3.6.6]{PU10}) that if $\mu\in \MMM(X,g)$ is an equilibrium state for $g$ and $\phi\in\CCC(X)$, then the continuous linear functional $u\longmapsto \int\!u\,\mathrm{d}\mu$ for $u\in\CCC(X)$ is tangent to the topological pressure function $P(g,\cdot)$ at $\phi$. Indeed, let $\phi,\gamma\in\CCC(X)$ and $\mu\in\MMM(X,g)$ be an equilibrium state for $g$ and $\phi$. Then $P(g,\phi+\gamma) \geq h_\mu(g) +\int\! \phi +\gamma \,\mathrm{d}\mu$ by the Variational Principle (\ref{eqVPPressure}), and $P(g,\phi)=h_\mu(g) +\int\! \phi \,\mathrm{d}\mu$. It follows that $P(g,\phi) + \int \! \gamma \,\mathrm{d}\mu \leq P(g,\phi+\gamma)$. 

Thus in order to prove the uniqueness of the equilibrium state for an expanding Thurston map $f\: S^2\rightarrow S^2$ and a real-valued H\"older continuous potential $\phi$, it suffices to prove that $\card\big(V_{\phi,P(f,\cdot)}^* \big) = 1 $. Then we can apply the following fact from functional analysis (see \cite[Theorem~3.6.5]{PU10} for a proof):

\begin{theorem}  \label{thmUniqueTangent}
Let $V$ be a separable Banach space and $Q\: V\rightarrow\R$ be a convex continuous function. Then for each $x\in V$, the following statements are equivalent:
\begin{enumerate}
\smallskip

\item $\card \(V_{x,Q}^*\) = 1$.

\smallskip

\item The function $t\longmapsto Q(x+ty)$ is differentiable at $0$ for each $y\in V$.

\smallskip

\item There exists a subset $U\subseteq V$ that is dense in the weak topology on $V$ such that the function $t\longmapsto Q(x+ty)$ is differentiable at $0$ for each $y\in U$.
\end{enumerate}

\end{theorem}

Now the problem of the uniqueness of equilibrium states transforms to the problem of (G\^ateaux) differentiability of the topological pressure function. To investigate the latter, we need a closer study of the fine properties of the Ruelle operator $\RR_\phi$.

A function $h\: [0,+\infty)\rightarrow [0,+\infty)$ is an \defn{abstract modulus of continuity} if it is continuous at $0$, non-decreasing, and $h(0)=0$. Given any metric $d$ on $S^2$ that generates the standard topology, any constant $b \in [0,+\infty]$, and any abstract modulus of continuity $h$, we define the subclass $\CCC_h^b(S^2, d)$ of $\CCC(S^2)$ as
\begin{align*}
C_h^b(S^2,d)=\{u\in\CCC(S^2)\,|\, & \Norm{u}_{\infty}\leq b \text{ and for }x,y\in S^2, \\
                                & \Abs{u(x)-u(y)}\leq h(d(x,y))\}.
\end{align*}
Note that by the Arzel\`a-Ascoli Theorem, each $\CCC_h^b(S^2,d)$ is precompact in $\CCC(S^2)$ equipped with the uniform norm. It is easy to see that each $\CCC_h^b(S^2,d)$ is actually compact. On the other hand, for $u\in\CCC(S^2)$, we can define an abstract modulus of continuity by
\begin{equation}    \label{eqAbsModContForU}
h(t)= \sup \{\abs{u(x)-u(y)} \,|\, x,y \in S^2, d(x,y) \leq t\}
\end{equation}
for $t\in [0,+\infty)$, so that $u\in \CCC_h^{\beta} (S^2,d)$, where $\beta=\Norm{u}_\infty$.

We will need the following lemma in this section.

\begin{lemma}  \label{lmChbChbSubsetChB}
Let $(X,d)$ be a metric space. For each pair of constants $b_1,b_2 \geq 0$ and each pair of abstract moduli of continuity $h_1,h_2$, there exists a constant $b\geq 0$ and an abstract modulus of continuity $h$ such that
\begin{equation}
\big\{u_1 u_2 \,\big|\, u_1 \in \CCC_{h_1}^{b_1} (X,d), u_2 \in \CCC_{h_2}^{b_2} (X,d) \big\}  \subseteq \CCC_h^b (X,d),
\end{equation}
and for each $c>0$,
\begin{equation}
\Big\{ \frac{1}{u} \,\Big|\,  u\in \CCC_{h_1}^{b_1} (X,d), u(x) \geq c \text{ for each } x\in X\Big\}  \subseteq \CCC_{c^{-2} h_1}^{c^{-1}}(X,d).
\end{equation}
\end{lemma}

\begin{proof}
For $u_1 \in \CCC_{h_1}^{b_1} (X,d), u_2 \in \CCC_{h_2}^{b_2} (X,d)$, we have $\Norm{u_1 u_2}_{\infty} \leq b_1 b_2$, and for $x,y\in X$,
\begin{align*}
     & \Abs{u_1(x)u_2(x)- u_1(y)u_2(y)}  \\
\leq & \Abs{u_1(x)}\Abs{u_2(x)-u_2(y)} + \Abs{u_2(y)}\Abs{u_1(x)-u_1(y)}  \\
\leq & b_1 h_2(d(x,y)) + b_2 h_1(d(x,y)).
\end{align*}
For $c>0$ and $u \in \CCC_{h_1}^{b_1} (X,d)$ with $u(x) \geq c$ for each $x\in X$, we have $\Norm{\frac{1}{u}}_{\infty} \leq \frac{1}{c}$, and for $x,y\in X$,
\begin{equation*}
 \Abs{\frac{1}{u(x)} - \frac{1}{u(y)}} 
\leq \Abs{\frac{u(x)-u(y)}{u(x)u(y)}} 
\leq \frac{1}{c^2} h_1(d(x,y)) .
\end{equation*}
\end{proof}

\smallskip

Let $f$, $d$, $\phi$, $\alpha$ satisfy the Assumptions. Recall that by (\ref{eqDefPhi-}) and Proposition~\ref{propTopPressureDefPreImg}, 
\begin{equation*}
\overline\phi=\phi-P(f,\phi).
\end{equation*}
We define the function 
\begin{equation}   \label{eqDefPhiTilde}
 \widetilde{\phi}= \phi -P(f,\phi) + \log u_\phi -\log (u_\phi\circ f),
\end{equation}
where $u_\phi$ is the continuous function given by Theorem~\ref{thmMuExist}. Then for $u\in\CCC(S^2)$ and $x\in S^2$, we have 
\begin{align*}
 & \RR_{\widetilde{\phi}}(u)(x) \\
= & \sum\limits_{y\in f^{-1}(x)}\deg_f(y)u(y)\exp(\phi(y)-P(f,\phi) + \log u_\phi(y) - \log u_\phi(f(y)))   \notag \\
  = & \frac{1}{u_\phi(x)} \sum\limits_{y\in f^{-1}(x)} \deg_f(y) u(y) u_\phi(y) \exp(\phi(y) - P(f,\phi)) \notag  \\
  = & \frac{1}{u_\phi(x)} \RR_{\overline\phi}(u u_\phi)(x),   \notag
\end{align*}
and thus
\begin{equation}  \label{eqRnphiTildeConjugate}
\RR_{\widetilde{\phi}}^n (u)(x) = \frac{1}{u_\phi(x)} \RR_{\overline\phi}^n (u u_\phi)(x), \qquad \text{for } n\in\N.
\end{equation}
Recall $m_\phi$ from Theorem~\ref{thmMexists}. Then we can show that $\mu_\phi = u_\phi m_\phi$ satisfies
\begin{equation}  \label{eqRtildePhi*Mu=Mu}
\RR_{\widetilde\phi}^* (\mu_\phi) = \mu_\phi.
\end{equation}
Indeed, by (\ref{eqRnphiTildeConjugate}) and (\ref{eqR*m}), for $u\in\CCC(S^2)$,
\begin{align*}
\int \! u \, \mathrm{d}\big(\RR_{\widetilde\phi}^* (\mu_\phi) \big)  & = \int \! \RR_{\widetilde\phi} (u) u_\phi \,\mathrm{d} m_\phi = \int \! \RR_{\overline\phi} (u u_\phi) \,\mathrm{d} m_\phi \\
                                                           &= \int \! u u_\phi  \,\mathrm{d}\big(\RR_{\overline\phi}^* (m_\phi)\big) = \int \, u u_\phi \,\mathrm{d} m_\phi  = \int \! u \,\mathrm{d} \mu_\phi.
\end{align*}
By (\ref{eqRu=u}) and (\ref{eqRnphiTildeConjugate}), $\RR_{\widetilde\phi}(\mathbbm{1}) = \frac{1}{u_\phi} \RR_{\overline\phi}(u_\phi) =\mathbbm{1}$, i.e.,
\begin{equation}  \label{eqRtildePhi1=1}
\sum\limits_{y\in f^{-1}(x)} \deg_f(y)\exp\widetilde\phi(y) = 1  \qquad \text{for } x\in S^2.
\end{equation}

\begin{lemma}   \label{lmRtildeNorm=1}
Let $f$, $d$, $\phi$, $\alpha$ satisfy the Assumptions. Then the operator norm of $\RR_{\widetilde\phi}$ is given by $\|\RR_{\widetilde\phi}\|=1$. In addition, $\RR_{\widetilde\phi}(\mathbbm{1})=\mathbbm{1}$.
\end{lemma}

\begin{proof}
By (\ref{eqRtildePhi1=1}), for each $x\in S^2$ and each $u\in\CCC(S^2)$, we have
\begin{align*}
\Abs{\RR_{\widetilde\phi}(u)(x)} & = \bigg| \sum\limits_{y\in f^{-1}(x)} \deg_f(y)u(y)\exp\widetilde\phi(y) \bigg|  \\
                                & \leq \Norm{u}_{\infty} \sum\limits_{y\in f^{-1}(x)} \deg_f(y)\exp\widetilde\phi(y) \\
                                & = \Norm{u}_{\infty}.
\end{align*}
Thus $\| \RR_{\widetilde\phi}  \| \leq 1$. Since $\RR_{\widetilde\phi}(\mathbbm{1})=\mathbbm{1}$ by (\ref{eqRtildePhi1=1}), $\|\RR_{\widetilde\phi} \| = 1$.
\end{proof}

\begin{lemma}  \label{lmTildePhiIsHolder}
Let $f$, $d$, $\phi$, $\alpha$ satisfy the Assumptions. Then
\begin{equation}  \label{eqTildePhiIsHolder}
\widetilde\phi \in \Holder{\alpha}(S^2,d).
\end{equation}
\end{lemma}

\begin{proof}
By Theorem~\ref{thmMuExist}, $u_\phi \in \Holder{\alpha}(S^2,d)$ and $C_2^{-1} \leq u_\phi (x) \leq C_2$ for each $x\in S^2$, where $C_2\geq 1$ is a constant from Lemma~\ref{lmSigmaExpSnPhiBound}. So $\log u_\phi \in \Holder{\alpha}(S^2,d)$. Note that $\phi \in \Holder{\alpha}(S^2,d)$, so by (\ref{eqDefPhiTilde}) it suffices to prove that $u_\phi \circ f \in \Holder{\alpha}(S^2,d)$. But this follows from the fact that $f$ is Lipschitz with respect to $d$ (Lemma~\ref{lmLipschitz}) and $u_\phi \in \Holder{\alpha}(S^2,d)$.
\end{proof}

\begin{theorem}   \label{thmChbInChb}
Let $f\:S^2 \rightarrow S^2$ be an expanding Thurston map, and $d$ be a visual metric on $S^2$ for $f$ with an expansion factor $\Lambda>1$ and a linear local connectivity constant $L\geq 1$. Then for each $\alpha\in(0,1]$, each $b\geq 0$, and each $\theta \geq 0$, there exist constants $\widehat{b}\geq 0$ and $\widehat{C} \geq 0$ with the following property:

For each abstract modulus of continuity $h$, there exists an abstract modulus of continuity $\widetilde{h}$ such that for each $\phi \in \Holder{\alpha}(S^2,d)$ with $\Hseminorm{\alpha}{\phi} \leq \theta$, we have
\begin{equation}   \label{eqChbInChb1}
\big\{\RR_{\overline{\phi}}^n(u) \,|\, u\in\CCC_h^b(S^2,d),n\in\N_0\big\} \subseteq \CCC_{\widehat{h}}^{\widehat{b}}(S^2,d),
\end{equation}   
\begin{equation}   \label{eqChbInChb2}
\big\{\RR_{\widetilde{\phi}}^n(u) \,|\, u\in\CCC_h^b(S^2,d),n\in\N_0\big\} \subseteq \CCC_{\widetilde{h}}^{b}(S^2,d),
\end{equation}
where $\widehat{h}(t)=\widehat{C}(t^\alpha+h(2C_0 L t))$ is an abstract modulus of continuity, and $C_0 > 1$ is a constant depending only on $f$ and $d$ from Lemma~\ref{lmMetricDistortion}.
\end{theorem}

\begin{proof}
Fix arbitrary $\alpha\in(0,1]$, $b\geq 0$, and $\theta \geq 0$. By Lemma~\ref{lmR1properties}, for $n\in\N_0$, $u\in \CCC_h^b(S^2,d)$, and $\phi\in\Holder{\alpha}(S^2,d)$ with $\Hseminorm{\alpha}{\phi} \leq \theta$, we have
$$
\Norm{\RR_{\overline{\phi}}^n(u)}_{\infty} \leq \Norm{u}_{\infty} \Norm{\RR_{\overline{\phi}}^n(\mathbbm{1})}_{\infty}  \leq C_2 \Norm{u}_{\infty},
$$
where $C_2$ is the constant defined in (\ref{eqC2Bound}) in Lemma~\ref{lmSigmaExpSnPhiBound}. So we can choose $\widehat{b} = C_2 b$. Note that by (\ref{eqC2Bound}) that $C_2$ only depends on $f$, $d$, $\theta$, and $\alpha$.

We pick a Jordan curve $\CC\subseteq S^2$ that satisfies the Assumptions (see Theorem~\ref{thmCexistsBM} for the existence of such $\CC$).

Let $X^0$ be either the white $0$-tile $X^0_w\in\X^0(f,\CC)$ or the black $0$-tile $X^0_b\in\X^0(f,\CC)$. If $X^m\in\X^m(f,\CC)$ is an $m$-tile with $f^m(X^m)=X^0$ for some $m\in\N_0$, then by Proposition~\ref{propCellDecomp}(i), the restriction $f^m|_{X^m}$ of $f^m$ to $X^m$ is a bijection from $X^m$ to $X^0$. So for each $x\in X^0$, there exists a unique point contained in $X^m$ whose image under $f^m$ is $x$. We denote this unique point by $x_m(X^m)$. Note that for each $z=x_m(X^m)$, the number of distinct $m$-tiles $X\in \X^m(f,\CC)$ that satisfy both $f^m(X)=X^0$ and $x_m(X)=z$ is exactly $\deg_{f^m}(z)$.


Then for each $x,y\in X^0$, we have
\begin{align*}
  & \Abs{\RR_{\overline{\phi}}^n(u)(x) - \RR_{\overline{\phi}}^n(u)(y)} \\
= &  \bigg|{ \sum\limits_{\begin{subarray}{l} X^n\in\X^n(f,\CC)\\ f^n(X^n)=X^0 \end{subarray}}  (u \exp(S_n\overline\phi)) ( x_n(X^n)) - (u \exp(S_n\overline\phi))( y_n(X^n) ) }\bigg|   \\
\leq & \bigg|  \sum\limits_{\begin{subarray}{l} X^n\in\X^n(f,\CC)\\ f^n(X^n)=X^0 \end{subarray}}  u(x_n(X^n))\(\exp(S_n\overline\phi(x_n(X^n))) - \exp(S_n\overline\phi(y_n(X^n)) )\)  \bigg|   \\  
     &+ \bigg| \sum\limits_{\begin{subarray}{l} X^n\in\X^n(f,\CC)\\ f^n(X^n)=X^0 \end{subarray}}  \exp(S_n\overline\phi(y_n(X^n))) \(u(x_n(X^n))-u(y_n(X^n))\) \bigg|.
\end{align*}

The second term above is 
$$
\leq C_2 h(C_0 \Lambda^{-n} d(x,y))\leq C_2 h(C_0 d(x,y)),
$$
due to (\ref{eqR1Bound}) and the fact that $d(x_n(X^n),y_n(X^n))\leq C_0 \Lambda^{-n} d(x,y)$ by Lemma~\ref{lmMetricDistortion}, where the constant $C_0$ comes from.

In order to bound the first term, we define 
$$
A_n^+=\{X^n\in\X^n(f,\CC)\,|\, f^n(X^n)=X^0, S_n\overline\phi(x_n(X^n)) \geq S_n\overline\phi(y_n(X^n))\},
$$
and
$$
A_n^-=\{X^n\in\X^n(f,\CC)\,|\, f^n(X^n)=X^0, S_n\overline\phi(x_n(X^n)) < S_n\overline\phi(y_n(X^n))\}.
$$
Then by (\ref{eqR^nExpr}), Lemma~\ref{lmSnPhiBound}, and Lemma~\ref{lmR1properties} the first term is
\begin{align*}
\leq & \sum\limits_{X^n\in A_n^+} \Norm{u}_{\infty}\(\exp(S_n\overline\phi(x_n(X^n))) - \exp(S_n\overline\phi(y_n(X^n)))\) \\
     & +  \sum\limits_{X^n\in A_n^-} \Norm{u}_{\infty}\(\exp(S_n\overline\phi(y_n(X^n))) - \exp(S_n\overline\phi(x_n(X^n)))\) \\
=    & \Norm{u}_{\infty}  \Bigg( \Bigg( \frac{\sum\limits_{X^n\in A_n^+} \exp(S_n\overline\phi(x_n(X^n)))}{\sum\limits_{X^n\in A_n^+}\exp(S_n\overline\phi(y_n(X^n)))}  - 1 \Bigg)\sum\limits_{X^n\in A_n^+}e^{S_n\overline\phi(y_n(X^n))}  \\
     & +  \Bigg( \frac{\sum\limits_{X^n\in A_n^-} \exp(S_n\overline\phi(y_n(X^n)))}{\sum\limits_{X^n\in A_n^-}\exp(S_n\overline\phi(x_n(X^n)))}  - 1 \Bigg)\sum\limits_{X^n\in A_n^-}e^{S_n\overline\phi(x_n(X^n))} \Bigg)   \\
\leq &  2 b C_2(\exp(C_1 d(x,y)^\alpha) -1) \\
\leq &  2 b \widetilde{C}_3 d(x,y)^\alpha,
\end{align*}
for some constant $\widetilde{C}_3>0$ that only depends on $C_1, C_2$, and $\diam_d(S^2)$, where $C_1>0$ is the constant defined in (\ref{eqC1Expression}) in Lemma~\ref{lmSnPhiBound} and $C_2\geq 1$ is the constant defined in (\ref{eqC2Bound}) in Lemma~\ref{lmSigmaExpSnPhiBound}. Note that the justification of the second inequality above is similar to that of (\ref{eqR1Diff}) in Lemma~\ref{lmR1properties}. We observe that by (\ref{eqC1Expression}) and (\ref{eqC2Bound}), both $C_1$ and $C_2$ only depend on $f$, $d$, $\theta$, and $\alpha$, so does $\widetilde{C}_3$.

Hence we get 
\begin{equation*}
       \Abs{\RR_{\overline{\phi}}^n(u)(x) - \RR_{\overline{\phi}}^n(u)(y)} 
\leq 2 b \widetilde{C}_3 d(x,y)^\alpha +  C_2 h(C_0 d(x,y)).
\end{equation*}

Now we consider arbitrary $x\in X^0_w$ and $y\in X^0_b$. Since the metric space $(S^2,d)$ is linearly locally connected with a linear local connectivity constant $L \geq 1$, there exists a continuum $E\subseteq S^2$ with $x,y\in E$ and $E\subseteq B_d(x,Ld(x,y))$. We can then choose $z\in\CC \cap E$. Note that $\max\{d(x,z),d(y,z)\} \leq 2L d(x,y)$.

Hence we get
\begin{align*}
& \Abs{\RR_{\overline{\phi}}^n(u)(x) - \RR_{\overline{\phi}}^n(u)(y)}\\
\leq  & \Abs{\RR_{\overline{\phi}}^n(u)(x) - \RR_{\overline{\phi}}^n(u)(z)}  +\Abs{\RR_{\overline{\phi}}^n(u)(z) - \RR_{\overline{\phi}}^n(u)(y)}\\
\leq  & 2 b \widetilde{C}_3 d(x,z)^\alpha +  C_2 h(C_0 d(x,z)) +2 b \widetilde{C}_3 d(z,y)^\alpha +  C_2 h(C_0 d(z,y))\\
\leq  & 8bL \widetilde{C}_3 d(x,y)^\alpha + 2C_2 h(2C_0 L d(x,y)).
\end{align*}
By choosing $\widehat{C}=\max\big\{8bL \widetilde{C}_3, 2C_2\big\}$, which only depends on $f$, $d$, $\theta$, and $\alpha$, we complete the proof of (\ref{eqChbInChb1}).

\smallskip

We now prove (\ref{eqChbInChb2}). 

We fix an arbitrary $\phi \in \Holder{\alpha}(S^2,d)$ with $\Hseminorm{\alpha}{\phi} \leq \theta$. Then by (\ref{eqU_phiBounds}) in Theorem~\ref{thmMuExist} and (\ref{eqC2Bound}) in Lemma~\ref{lmSigmaExpSnPhiBound}, we have
\begin{equation*}  
\Norm{u_\phi}_\infty \leq b_1,
\end{equation*}
where $b_1 = \exp\(4 \frac{\theta C_0}{1-\Lambda^{-1}}L \big(\diam(S^2)\big)^\alpha \)$. By Theorem~\ref{thmMuExist} and (\ref{eqR1Diff}) in Lemma~\ref{lmR1properties}, for each $x,y\in S^2$, we have
\begin{align*}
\Abs{u_\phi(x) - u_\phi(y)}  &    = \Abs{\lim\limits_{n\to+\infty} \frac{1}{n} \sum\limits_{j=0}^{n-1} \big(\RR_{\overline\phi}^j(\mathbbm{1})(x) - \RR_{\overline\phi}^j(\mathbbm{1})(x)  \big)}  \\
                             & \leq \limsup\limits_{n\to+\infty}  \frac{1}{n}  \sum\limits_{j=0}^{n-1} \Abs{\RR_{\overline\phi}^j(\mathbbm{1})(x) - \RR_{\overline\phi}^j(\mathbbm{1})(x) }   \\
                             & \leq C_2 \( \exp\(4C_1 Ld(x,y)^\alpha\) - 1 \).
\end{align*}
So
\begin{equation}  \label{eqUphiInCbh}
u_\phi \in \CCC^{b_1}_{h_1} (S^2,d),
\end{equation}
where $h_1$ is an abstract modulus of continuity given by 
\begin{equation*}
h_1(t) = C_2 \( \exp\(4C_1 L t^\alpha\) - 1 \), \text{ for } t\in[0,+\infty).
\end{equation*}

Thus by Lemma~\ref{lmChbChbSubsetChB}, there exist a constant $b_2\geq 0$ and an abstract modulus of continuity $h_2$ such that
\begin{equation}   \label{eqUUphiInCbh}
\big\{ u u_\phi \,\big|\, u\in \CCC^b_h(S^2,d), \phi \in \Holder{\alpha}(S^2,d), \Hseminorm{\alpha}{\phi} \leq \theta \big\} \subseteq \CCC^{b_2}_{h_2} (S^2,d).
\end{equation}
Then by (\ref{eqRnphiTildeConjugate}), (\ref{eqUUphiInCbh}), (\ref{eqChbInChb1}), and Lemma~\ref{lmChbChbSubsetChB}, we get that there exist a constant $b_3\geq 0$ and an abstract modulus of continuity $\widetilde{h}$ such that
\begin{equation}
\big\{ \RR_{\widetilde\phi}^n (u) \,\big|\, u\in \CCC^b_h(S^2,d), n\in \N_0 \big\}  \subseteq \CCC^{b_3}_{\widetilde{h}} (S^2,d),
\end{equation}
for each $\phi\in \Holder{\alpha}(S^2,d)$ with $\Hseminorm{\alpha}{\phi} \leq \theta$. 

On the other hand, by Lemma~\ref{lmRtildeNorm=1}, $\Norm{\RR_{\widetilde\phi}^n(u)}_\infty \leq \Norm{u}_\infty \leq b$ for each $u\in  \CCC^b_h(S^2,d)$, each $n\in \N_0$, and each $\phi \in \Holder{\alpha}(S^2,d)$. Therefore, we have proved (\ref{eqChbInChb2}).
\end{proof}

As a consequence, both $\RR_{\overline\phi}$ and $\RR_{\widetilde\phi}$ are almost periodic.

\begin{definition}     \label{defAlmostPeriodic}
A bounded linear operator $L\: B\rightarrow B$ on a Banach space $B$ is \defn{almost periodic} if for each $z\in B$, the closure of the set $\{L^n(z) \,|\, n\in\N_0\}$ is compact in the norm topology.
\end{definition}

\begin{cor}    \label{corRAlmostPeriodic}
Let $f$, $d$, $\phi$, and $\alpha$ satisfy the Assumptions. Let $\CCC(S^2)$ be equipped with the uniform norm. Then both $\RR_{\overline\phi} \: \CCC(S^2)\rightarrow \CCC(S^2)$ and $\RR_{\widetilde\phi} \: \CCC(S^2)\rightarrow \CCC(S^2)$ are almost periodic.
\end{cor}

\begin{proof}
For each $u\in\CCC(S^2)$, we have $u\in\CCC^{\beta}_{h} (S^2,d)$ for $\beta=\Norm{u}_\infty$ and some abstract modulus of continuity $h$ defined in (\ref{eqAbsModContForU}). Then the corollary follows immediately from Theorem~\ref{thmChbInChb} and Arzel\`a-Ascoli theorem.
\end{proof}

\begin{lemma}    \label{lmRtildeSupBound}
Let $f$ and $d$ satisfy the Assumptions. Let $g$ be an abstract modulus of continuity.  Then for $\alpha\in(0,1], K\in (0,+\infty)$, and $\delta_1 \in (0,+\infty)$, there exist constants $\delta_2 \in (0,+\infty)$ and $n\in\N$ with the following property:

For each $u\in \CCC_g^{+\infty}(S^2,d)$, each $\phi\in \Holder{\alpha}(S^2,d)$, and each choice of $m_\phi$ from Theorem~\ref{thmMexists}, if $\Hnorm{\alpha}{\phi}\leq K$, $\int \! u u_\phi\,\mathrm{d}m_\phi = 0$, and $\Norm{u}_{\infty} \geq \delta_1$, then 
\begin{equation*}
\Norm{\RR_{\widetilde\phi}^n(u)}_{\infty} \leq \Norm{u}_{\infty} - \delta_2.
\end{equation*}
\end{lemma}

Note that at this point, we have not proved yet that $m_\phi$ from Theorem~\ref{thmMexists} is unique. We will prove it in Corollary~\ref{corMandMuUnique}. Recall that $u_\phi$ is the continuous function defined in Theorem~\ref{thmMuExist} that only depends on $f$ and $\phi$.

\begin{proof}
Fix arbitrary constants $\alpha\in(0,1], K\in (0,+\infty)$, and $\delta_1 \in (0,+\infty)$. Fix $\epsilon > 0$ small enough such that $g(\epsilon) < \frac{\delta_1}{2}$. Fix a choice of $m_\phi$, an arbitrary $\phi\in\Holder{\alpha}(S^2,d)$, and an arbitrary $u\in \CCC_g^{+\infty}(S^2,d)$ with $\Hnorm{\alpha}{\phi}\leq K$, $\int \! u u_\phi\,\mathrm{d}m_\phi = 0$, and $\Norm{u}_{\infty} \geq \delta_1$.

We pick a Jordan curve $\CC\subseteq S^2$ that satisfies the Assumptions (see Theorem~\ref{thmCexistsBM} for the existence of such $\CC$). 

By Lemma~\ref{lmCellBoundsBM}(iv), there exists $n\in\Z$ depending only on $f$, $\CC$, $d$, $g$, and $\delta_1$ such that for each $z\in S^2$, we have $U^n(z)\subseteq B_d(z,\epsilon)$, where $U^n(z)$ is defined in (\ref{defU^n}). Since $\int \! u u_\phi \,\mathrm{d}m_\phi = 0$, there exist points $y_1,y_2\in S^2$ such that $u(y_1)\leq 0$ and $u(y_2)\geq 0$.

We fix a point $x \in S^2$. Since  $f^n(U^n(y_1)) = S^2$, there exists $y\in f^{-n}(x)$ such that $y \in U^n(y_1) \subseteq B_d(y_1,\epsilon)$. Thus
$$
u(y) \leq u(y_1) + g(\epsilon) < \frac{\delta_1}{2} \leq \Norm{u}_{\infty} -\frac{\delta_1}{2}.
$$
So by Lemma~\ref{lmRtildeNorm=1} and (\ref{eqR^nExpr}) we have
\begin{align*}
     & \RR_{\widetilde\phi}^n(u)(x)  \\
   = & \deg_{f^n}(y)u(y)\exp\big(S_n\widetilde\phi(y)\big) \\
     & +  \sum\limits_{w\in f^{-n}(x)\setminus \{y\}}\deg_{f^n}(w)u(w)\exp\big(S_n\widetilde\phi(w)\big) \\
\leq & \(  \Norm{u}_{\infty} -\frac{\delta_1}{2}  \) \deg_{f^n}(y)\exp\big(S_n\widetilde\phi(y)\big)  \\
     & + \Norm{u}_{\infty}\sum\limits_{w\in f^{-n}(x)\setminus \{y\}}\deg_{f^n}(w)\exp\big(S_n\widetilde\phi(w)\big) \\
\leq & \Norm{u}_{\infty}\sum\limits_{w\in f^{-n}(x)}\deg_{f^n}(w)\exp\big(S_n\widetilde\phi(w)\big)- \frac{\delta_1}{2} \exp\big(S_n\widetilde\phi(y)\big) \\
   = & \Norm{u}_{\infty} -\frac{\delta_1}{2} \exp\big(S_n\widetilde\phi(y)\big).
\end{align*}
Similarly, there exists $z\in f^{-n}(x)$ such that $z\in U^n(y_2) \subseteq B_d(y_2,\epsilon)$ and
$$
\RR_{\widetilde\phi}^n (u)(x) \geq -\Norm{u}_{\infty} + \frac{\delta_1}{2} \exp\big(S_n\widetilde\phi(z)\big).
$$
Hence we get
\begin{equation}   \label{eqPflmRtildeSupBound}
\Norm{\RR_{\widetilde\phi}^n(u)}_{\infty}  \leq  \Norm{u}_{\infty} - \frac{\delta_1}{2}  \inf \big\{\exp\big(S_n\widetilde\phi(w)\big) \,\big|\, w\in S^2   \big\}.
\end{equation}

Now it suffices to bound each term in the definition of $\widetilde\phi$ in (\ref{eqDefPhiTilde}).

First, by the hypothesis, $\Norm{\phi}_{\infty} \leq \Hnorm{\alpha}{\phi} \leq K$ (see (\ref{eqDefHolderNorm})).

Next, for each fixed $x\in S^2$, by Proposition~\ref{propTopPressureDefPreImg}, we have
\begin{align*}
P(f,\phi) & = \lim\limits_{n\to+\infty} \frac{1}{n} \log \sum\limits_{y\in f^{-n}(x)} \deg_{f^n}(y)\exp(S_n\phi(y))  \\
          & \leq  \lim\limits_{n\to+\infty} \frac{1}{n} \log \sum\limits_{y\in f^{-n}(x)} \deg_{f^n}(y)\exp(nK)  \\
          & = K + \lim\limits_{n\to+\infty}\frac{1}{n}  \log \sum\limits_{y\in f^{-n}(x)} \deg_{f^n}(y)\\
          & = K + \log(\deg f).
\end{align*}
Similarly, $P(f,\phi) \geq -K + \log (\deg f)$. So $\Abs{P(f,\phi)} \leq K + \Abs{\log (\deg f)}$.

Finally, by Theorem~\ref{thmMuExist} and (\ref{eqC2Bound}) in Lemma~\ref{lmSigmaExpSnPhiBound}, we have 
$$
\Norm{u_\phi}_{\infty}  \leq C_2 \leq \exp\( C_5 \),
$$
where 
$$
C_5 = 4 \frac{K C_0}{1-\Lambda^{-1}} L\( \diam_d (S^2) \)^\alpha,
$$
and $C_0 > 1$ is a constant from Lemma~\ref{lmMetricDistortion} depending only on $f$ and $d$.

Therefore, by (\ref{eqDefPhiTilde}) and (\ref{eqPflmRtildeSupBound}), $\Norm{\RR_{\widetilde\phi}^n(u)}_{\infty} \leq  \Norm{u}_{\infty} - \delta_2$, where
$$
\delta_2 = \frac{\delta_1}{2} \exp\( -n \( 2K + \Abs{\log(\deg f)} + 2 C_5 \)\),
$$
which only depends on $f$, $d$, $\alpha$, $K$, $\delta_1$, $g$, and $n$.
\end{proof}

\begin{theorem}  \label{thmRR^nConv} 
Let $f\:S^2 \rightarrow S^2$ be an expanding Thurston map. Let $d$ be a visual metric on $S^2$ for $f$ with an expansion factor $\Lambda>1$. Let $b \in (0,+\infty)$ be a constant and $h\:[0,+\infty)\rightarrow [0,+\infty)$ an abstract modulus of continuity. Let $H$ be a bounded subset of $\Holder{\alpha}(S^2,d)$ for some $\alpha \in (0,1]$. Then for each $u\in \CCC_h^b(S^2,d)$, each $\phi\in H$, and each choice of $m_\phi$ from Theorem~\ref{thmMexists}, we have
\begin{equation}   \label{eqRR^nConv}
\lim\limits_{n\to+\infty} \Norm{\RR_{\overline\phi}^n (u) - u_\phi \int\! u\,\mathrm{d}m_\phi}_{\infty} = 0.
\end{equation} 
If, in addition, $\int\! u u_\phi\,\mathrm{d}m_\phi = 0$, then
\begin{equation}   \label{eqRR^nConv0}
\lim\limits_{n\to+\infty} \Norm{\RR_{\widetilde\phi}^n(u)}_{\infty} = 0.
\end{equation}
Moreover, the convergence in both (\ref{eqRR^nConv}) and (\ref{eqRR^nConv0}) is uniform in $u\in\CCC_h^b(S^2,d)$, $\phi\in H$, and the choice of $m_\phi$.
\end{theorem}

The equation (\ref{eqRR^nConv0}) demonstrates the contracting behavior of $\RR_{\widetilde\phi}$ on a codimension $1$ subspace of $\CCC(S^2)$.

\begin{proof}
Let $L$ be a linear local connectivity constant of $d$. Fix a constant $K \in (0,+\infty)$ such that $\Hnorm{\alpha}{\phi} \leq K$ for each $\phi\in H$.

Let $M_\phi$ be the set of possible choices of $m_\phi$ from Theorem~\ref{thmMexists}, i.e.,
\begin{equation}
M_\phi  = \{ m\in\PPP(S^2) \,|\, \RR^*_\phi (m) = c m \text{ for some } c\in\R   \}.
\end{equation}

We recall that $\mu_\phi$ defined in Theorem~\ref{thmMuExist} by $\mu_\phi= u_\phi m_\phi$ depends on the choice of $m_\phi$.

Define for each $n\in \N_0$,
$$
a_n = \sup \Big\{ \Norm{\RR_{\widetilde\phi}^n(u)}_{\infty}  \,\Big|\,  \phi \in H, u\in \CCC_h^b(S^2,d),\int\!u\,\mathrm{d}\mu_\phi = 0, m_\phi \in M_\phi \Big\}.
$$
By Lemma~\ref{lmRtildeNorm=1}, $\|\RR_{\widetilde\phi}\|=1$, so $\Norm{\RR_{\widetilde\phi}^n(u)}_{\infty}$ is non-increasing in $n$ for fixed $\phi\in H$ and $u\in\CCC_h^b(S^2,d)$. Note that $a_0\leq b <+\infty$. Thus $\{a_n\}_{n\in\N_0}$ is a non-increasing sequence of non-negative real numbers.

Suppose now that $\lim\limits_{n\to+\infty} a_n = a >0$. By Theorem~\ref{thmChbInChb}, there exists an abstract modulus of continuity $g$ such that
$$
\big\{ \RR_{\widetilde\phi}^n(u) \,\big|\, n\in\N_0 , \phi \in H, u\in \CCC^b_h(S^2,d)  \big\} \subseteq \CCC_g^{b}(S^2,d).
$$
Note that for each $\phi\in H$, each $n\in\N_0$, and each $u\in\CCC^b_h(S^2,d)$ with $\int \! u u_\phi\, \mathrm{d}m_\phi = 0$, we have $\int \! \RR_{\widetilde{\phi}}^n(u) u_\phi \,\mathrm{d}m_\phi = \int \! \RR_{\widetilde{\phi}}^n(u) \,\mathrm{d}\mu_\phi = 0$ by (\ref{eqRtildePhi*Mu=Mu}). So by applying Lemma~\ref{lmRtildeSupBound} with $g,\alpha, K$, and $\delta_1=\frac{a}{2}$, we find constants $n_0 \in \N$ and $\delta_2 > 0$ such that
\begin{equation}
\Norm{\RR_{\widetilde\phi}^{n_0}\Big(\RR_{\widetilde\phi}^n(u)\Big)}_{\infty}   \leq \Norm{\RR_{\widetilde\phi}^n(u)}_{\infty} - \delta_2,
\end{equation}
for each $n\in\N_0$, each $\phi\in H$, each $m_\phi \in M_\phi$, and each $u\in\CCC_h^b(S^2,d)$ with $\int\! u u_\phi \, \mathrm{d} m_\phi = 0$ and $\Norm{\RR_{\widetilde\phi}^n(u)}_{\infty} \geq \frac{a}{2}$. Since $\lim\limits_{n\to+\infty} a_n = a$, we can fix $m>1$ large enough such that $a_m\leq a +\frac{\delta_2}{2}$. Then for each $\phi\in H$, each $m_\phi \in M_\phi$, and each $u\in \CCC_h^b(S^2,d)$ with $\int\! u \, \mathrm{d} \mu_\phi = 0$ and $\Norm{\RR_{\widetilde\phi}^m(u)}_{\infty} \geq \frac{a}{2}$, we have
\begin{equation}
\Norm{\RR_{\widetilde\phi}^{n_0+m}(u)}_{\infty}  \leq \Norm{\RR_{\widetilde\phi}^m(u)}_{\infty} - \delta_2 \leq a_m-\delta_2 \leq a - \frac{\delta_2}{2}.
\end{equation}
On the other hand, since $\Norm{\RR_{\widetilde\phi}^n(u)}_{\infty}$ is non-increasing in $n$, we have that for each $\phi\in H$, each $m_\phi \in M_\phi$, and each $u\in \CCC_h^b(S^2,d)$ with $\int\! u \, \mathrm{d} \mu_\phi = 0$ and $\Norm{\RR_{\widetilde\phi}^m(u)}_{\infty} < \frac{a}{2}$, the following holds:
\begin{equation}
\Norm{\RR_{\widetilde\phi}^{n_0+m}(u)}_{\infty}  \leq \Norm{\RR_{\widetilde\phi}^m(u)}_{\infty} < \frac{a}{2}.
\end{equation}
Thus $a_{n_0+m} \leq \max\big\{a -\frac{\delta_2}{2},\frac{a}{2} \big\} < a$, contradicting the fact that $\{a_n\}_{n\in\N_0}$ is a non-increasing sequence and the assumption that $\lim\limits_{n\to+\infty}a_n = a $. This proves the uniform convergence in (\ref{eqRR^nConv0}).

\smallskip

Next, we prove the uniform convergence in (\ref{eqRR^nConv}). By Lemma~\ref{lmSnPhiBound}, Lemma~\ref{lmSigmaExpSnPhiBound}, Lemma~\ref{lmRtildeNorm=1}, and (\ref{eqRnphiTildeConjugate}), for each $u\in\CCC_h^b(S^2,d)$, each $\phi\in H$, and each $m_\phi \in M_\phi$, we have
\begin{align} \label{eqRR^nConvBound}
     & \Norm{ \RR_{\overline\phi}^n (u) - u_\phi \int\! u\,\mathrm{d}m_\phi}_{\infty}  \\
\leq & \Norm{u_\phi}_{\infty}  \Norm{\frac{1}{u_\phi} \RR_{\overline\phi}^n(u) -\int\! u\,\mathrm{d}m_\phi   }_{\infty} \notag \\
   = & \Norm{u_\phi}_{\infty}  \Norm{ \RR_{\widetilde\phi}^n\(\frac{u}{u_\phi}\)- \int\! \frac{u}{u_\phi}\,\mathrm{d}\mu_\phi   }_{\infty} \notag \\
   = & \Norm{u_\phi}_{\infty}  \Norm{ \RR_{\widetilde\phi}^n \( \frac{u}{u_\phi}- \mathbbm{1} \int\! \frac{u}{u_\phi}\,\mathrm{d}\mu_\phi  \) }_{\infty}.  \notag
\end{align}
By (\ref{eqU_phiBounds}) and (\ref{eqC2Bound}), we have
\begin{equation}  \label{eqU_phiBoundsByK}
\exp\(-C_5 \) \leq \Norm{u_\phi}_{\infty} \leq \exp\(C_5 \),
\end{equation}
where 
$$
C_5 = 4 \frac{K C_0}{1-\Lambda^{-1}} L\( \diam_d (S^2) \)^\alpha,
$$
and $C_0$ is a constant from Lemma~\ref{lmMetricDistortion} depending only on $f$ and $d$. Let $v= \frac{u}{u_\phi}-\mathbbm{1}\int\! \frac{u}{u_\phi}\,\mathrm{d}\mu_\phi$. Then $v$ satisfies
\begin{equation}
\Norm{v}_{\infty} \leq 2 \Norm{\frac{u}{u_\phi}}_{\infty} \leq 2 b \exp\( C_5 \).
\end{equation}
Due to the first inequality in (\ref{eqU_phiBoundsByK}) and the fact that $u_\phi\in \Holder{\alpha}(S^2,d)$ by Theorem~\ref{thmMuExist}, we can apply Lemma~\ref{lmChbChbSubsetChB} and conclude that there exists an abstract modulus of continuity $g$ of $\frac{u}{u_\phi}$ such that $g$ is independent of the choices of $u\in\CCC_h^b(S^2,d)$, $\phi\in H$, and $m_\phi \in M_\phi$. Thus $v \in \CCC_g^{\widehat{b}}(S^2,d)$, where $\widehat{b} =  2 b \exp(C_5)$. Note that $\int\! v u_\phi\,\mathrm{d}m_\phi =\int\! v\,\mathrm{d}\mu_\phi = 0$. Finally, we can apply the uniform convergence in (\ref{eqRR^nConv0}) with $u=v$ to conclude the uniform convergence in (\ref{eqRR^nConv}) by (\ref{eqRR^nConvBound}) and (\ref{eqU_phiBoundsByK}).
\end{proof}

Theorem~\ref{thmRR^nConv} implies in particular the uniqueness of $m_\phi$ and $\mu_\phi$.
\begin{cor}    \label{corMandMuUnique}
Let $f$, $d$, $\phi$, $\alpha$ satisfy the Assumptions. Then the measure $m_\phi\in\PPP(S^2)$ defined in Theorem~\ref{thmMexists} is unique, i.e., $m_\phi$ is the unique Borel probability measure on $S^2$ that satisfies $\RR^*_\phi (m_\phi) = cm_\phi$ for some constant $c\in \R$. Moreover, $\mu_\phi= u_\phi m_\phi$ is the unique Borel probability measure on $S^2$ that satisfies $\RR_{\widetilde\phi}^* (\mu_\phi) = \mu_\phi$. In particular, we have $m_{\widetilde\phi} = \mu_\phi$.
\end{cor}

\begin{proof}
Let $m_\phi, \widehat{m}_\phi \in \PPP(S^2)$ be two measures, both of which arise from Theorem~\ref{thmMexists}. Recall that for each $u\in\CCC(S^2)$, there exists some abstract modulus of continuity $h$ such that $u\in\CCC_h^{\beta} (S^2,d)$, where $\beta=\Norm{u}_\infty$. Then by (\ref{eqRR^nConv}) and (\ref{eqU_phiBounds}), we see that $\int \! u \,\mathrm{d} m_\phi = \int \! u \,\mathrm{d} \widehat{m}_\phi$ for each $u\in\CCC(S^2)$. Thus $m_\phi = \widehat{m}_\phi$.

By (\ref{eqRtildePhi*Mu=Mu}), $\RR_{\widetilde\phi}^* (\mu_\phi) = \mu_\phi$. Since $\widetilde\phi \in \Holder{\alpha}(S^2,d)$ by Lemma~\ref{lmTildePhiIsHolder}, we get that $\mu_\phi = m_{\widetilde\phi}$ and $\mu_\phi$ is the only Borel probability measure on $S^2$ that satisfies $\RR_{\widetilde\phi}^* (\mu_\phi) = \mu_\phi$.
\end{proof}

\begin{lemma}   \label{lmDerivConv}
Let $f$ and $d$ satisfy the Assumptions. Let $b \geq 0$ be a constant and $h$ an abstract modulus of continuity. Let $H$ be a bounded subset of $\Holder{\alpha}(S^2,d)$ for some $\alpha \in (0,1]$. Then for each $x\in S^2$, each $u\in \CCC_h^b(S^2,d)$, and each $\phi\in H$, we have
\begin{equation}  \label{eqDerivConv}
\lim\limits_{n\to+\infty} \frac{\frac{1}{n}\sum\limits_{y\in f^{-n}(x)} \deg_{f^n}(y) \(S_n u(y)\) \exp(S_n\phi(y))}{\sum\limits_{y\in f^{-n}(x)} \deg_{f^n}(y) \exp(S_n\phi(y))}  =  \int \! u \,\mathrm{d} \mu_\phi.
\end{equation}
Moreover, the convergence is uniform in $x\in S^2$, $u\in \CCC_h^b(S^2,d)$, and $\phi\in H$.
\end{lemma}

\begin{proof}
By (\ref{eqR^nExpr}) and (\ref{eqLocalDegreeProduct}), for $x\in S^2$, $u\in \CCC_h^b(S^2,d)$, $\phi\in H$, and $n\in\N$,
\begin{align*}
  & \frac{\frac{1}{n}\sum\limits_{y\in f^{-n}(x)} \deg_{f^n}(y) \(S_n u(y)\) \exp(S_n\phi(y))}{\sum\limits_{y\in f^{-n}(x)} \deg_{f^n}(y) \exp(S_n\phi(y))} \\
 =& \frac{\frac{1}{n} \sum\limits_{j=0}^{n-1} \sum\limits_{y\in f^{-n}(x)}  \deg_{f^n}(y) u(f^j (y)) \exp(S_n\phi(y))}{\RR_\phi^n(\mathbbm{1})(x)}   \\
 =& \frac{\frac{1}{n} \sum\limits_{j=0}^{n-1} \sum\limits_{z\in f^{-(n-j)}(x)} \sum\limits_{y\in f^{-j}(z)}  \deg_{f^{n-j}}(z)\deg_{f^j}(y) u(z) e^{S_j\phi(y)+S_{n-j}\phi(z)}}{\RR_\phi^n(\mathbbm{1})(x)}   \\
 =&  \frac{\frac{1}{n} \sum\limits_{j=0}^{n-1} \sum\limits_{z\in f^{-(n-j)}(x)} \deg_{f^{n-j}}(z) u(z)\RR_\phi^j(\mathbbm{1})(z)\exp(S_{n-j}\phi(z))}{\RR_\phi^n(\mathbbm{1})(x)}   \\
 =& \frac{\frac{1}{n} \sum\limits_{j=0}^{n-1} \RR_\phi^{n-j} \Big(u  \RR_\phi^j(\mathbbm{1})\Big)(x)}{\RR_\phi^n(\mathbbm{1})(x)}   \\
 =& \frac{\frac{1}{n} \sum\limits_{j=0}^{n-1} \RR_{\overline\phi}^{n-j} \Big(u  \RR_{\overline\phi}^j(\mathbbm{1})\Big)(x)}{\RR_{\overline\phi}^n(\mathbbm{1})(x)}.
\end{align*}
By Theorem~\ref{thmChbInChb}, $\big\{\RR_{\overline\phi}^n(\mathbbm{1}) \,|\, n\in \N_0 \big\}  \subseteq \CCC_{\widehat{h}}^{\widehat{b}}(S^2,d)$, for some constant $\widehat{b}\geq 0$ and some abstract modulus of continuity $\widehat{h}$, which are independent of the choice of $\phi\in H$. Thus by Lemma~\ref{lmChbChbSubsetChB},
\begin{equation}  \label{eqUR^n1}
\big\{u \RR_{\overline\phi}^n(\mathbbm{1}) \,|\, n\in \N_0, u\in \CCC_h^b(S^2,d) \big\}  \subseteq \CCC_{h_1}^{b_1}(S^2,d),
\end{equation}
for some constant $b_1\geq 0$ and some abstract modulus of continuity $h_1$, which are independent of the choice of $\phi\in H$.

Hence, by Theorem~\ref{thmRR^nConv} and Corollary~\ref{corMandMuUnique}, we have
\begin{equation}  \label{eqRRphi^n1ConvToUphi}
\Norm{\RR_{\overline\phi}^l(\mathbbm{1}) - u_\phi}_{\infty} \longrightarrow 0,
\end{equation}
and
\begin{equation}  \label{eqURRphi^nBounds}
\Norm{\RR_{\overline\phi}^l \(u\RR_{\overline\phi}^j(\mathbbm{1})\)- u_\phi \int \! u \RR_{\overline\phi}^j (\mathbbm{1}) \,\mathrm{d}m_\phi}_{\infty} \longrightarrow 0,
\end{equation}
as $l\longrightarrow +\infty$, uniformly in $j\in\N_0,\phi\in H$, and $u\in\CCC_h^b(S^2,d)$.

Fix a constant $K\in(0,+\infty)$ such that for each $\phi\in H$, $\Hnorm{\alpha}{\phi} \leq K$. By (\ref{eqU_phiBounds}) and (\ref{eqC2Bound}), we have that for each $x\in S^2$,
\begin{equation}   \label{eqUphiBoundsC5}
\exp (- C_5  ) \leq u_\phi(x) \leq \exp ( C_5 ),
\end{equation}
where
$$
C_5 = 4 \frac{K C_0}{1-\Lambda^{-1}} L\( \diam_d (S^2) \)^\alpha,
$$
and $C_0\geq 1$ is a constant from Lemma~\ref{lmMetricDistortion} depending only on $f$ and $d$. So by (\ref{eqUR^n1}), we get that for $j\in\N_0$, $u\in\CCC_h^b(S^2,d)$, and $\phi\in H$,
\begin{equation}  \label{eqUphiULBounds}
\Norm{u_\phi \int \! u \RR_{\overline\phi}^j (\mathbbm{1}) \,\mathrm{d}m_\phi}_{\infty} \leq \Norm{u_\phi}_{\infty}\Norm{u\RR_{\overline\phi}^j (\mathbbm{1})}_{\infty}  \leq b_1 \exp ( C_5 ) .
\end{equation}
By (\ref{eqChbInChb1}) in Theorem~\ref{thmChbInChb} and (\ref{eqUR^n1}), we get some constant $b_2 >0$ such that for all $j,l\in\N_0$, each $u\in\CCC_h^b(S^2,d)$, and each $\phi\in H$,
\begin{equation}   \label{eqLuLBounds}
\Norm{\RR_{\overline\phi}^l \(u \RR_{\overline\phi}^j (\mathbbm{1}) \)}_\infty < b_2.
\end{equation}
Hence we can conclude from (\ref{eqUphiULBounds}), (\ref{eqLuLBounds}), and (\ref{eqURRphi^nBounds}) that
\begin{equation*}
 \lim\limits_{n\to+\infty} \frac{1}{n}  \bigg| \sum\limits_{j=0}^{n-1} \RR_{\overline\phi}^{n-j} \(u  \RR_{\overline\phi}^j(\mathbbm{1})\)(x)   -   \sum\limits_{j=0}^{n-1}u_\phi(x) \int \! u \RR_{\overline\phi}^j (\mathbbm{1}) \,\mathrm{d}m_\phi \bigg|= 0,
\end{equation*}
uniformly in $u\in\CCC_h^b(S^2,d),\phi\in H$, and $x\in S^2$. Thus by (\ref{eqRRphi^n1ConvToUphi}) and (\ref{eqUphiBoundsC5}), we have
\begin{equation*}
 \lim\limits_{n\to+\infty} \Bigg| \frac{\frac{1}{n} \sum\limits_{j=0}^{n-1} \RR_{\overline\phi}^{n-j} \(u  \RR_{\overline\phi}^j(\mathbbm{1})\)(x)}{\RR_{\overline\phi}^n(\mathbbm{1})(x)} -  \frac{\frac{1}{n} \sum\limits_{j=0}^{n-1}u_\phi(x) \displaystyle\int \! u \RR_{\overline\phi}^j (\mathbbm{1}) \,\mathrm{d}m_\phi}{u_\phi(x)} \Bigg| = 0,
\end{equation*}
uniformly in $u\in\CCC_h^b(S^2,d),\phi\in H$, and $x\in S^2$. Combining the above with (\ref{eqUR^n1}), (\ref{eqRRphi^n1ConvToUphi}), (\ref{eqUphiBoundsC5}), and the calculation in the beginning part of the proof, we can conclude, therefore, that the left-hand side of (\ref{eqDerivConv}) is equal to
\begin{equation*}
\lim\limits_{n\to+\infty} \frac{1}{n} \sum\limits_{j=0}^{n-1}\int \! u \RR_{\overline\phi}^j (\mathbbm{1}) \,\mathrm{d}m_\phi = \lim\limits_{n\to+\infty} \frac{1}{n} \sum\limits_{j=0}^{n-1}\int \! u u_\phi \,\mathrm{d}m_\phi = \int \! u\,\mathrm{d}\mu_\phi,
\end{equation*}
and the convergence is uniform in $u\in\CCC_h^b(S^2,d)$ and $\phi\in H$.
\end{proof}

We record the following well-known fact for the convenience of the reader.

\begin{lemma}   \label{lmHolderDense}
For each metric $d$ on $S^2$ that generates the standard topology on $S^2$ and each $\alpha\in(0,1]$, $\Holder{\alpha}(S^2,d)$ is a dense subset of $\CCC(S^2)$ with respect to the uniform norm. In particular, $\Holder{\alpha}(S^2,d)$ is a dense subset of $\CCC(S^2)$ in the weak topology.
\end{lemma}

\begin{proof}
The lemma follows from the fact that the set of Lipschitz functions are dense in $\CCC(S^2)$ with respect to the uniform norm (see for example, \cite[Theorem~6.8]{He01}).
\end{proof}

\begin{theorem}  \label{thmPisDiff}
Let $f\:S^2 \rightarrow S^2$ be an expanding Thurston map, and $d$ be a visual metric on $S^2$ for $f$. Let $\phi,\gamma \in \Holder{\alpha}(S^2,d)$ be real-valued H\"{o}lder continuous functions with an exponent $\alpha\in(0,1]$. Then for each $t\in \R$, we have
\begin{equation}
\frac{\mathrm{d}}{\mathrm{d}t} P(f,\phi + t\gamma) = \int \!\gamma \,\mathrm{d}\mu_{\phi + t\gamma}.
\end{equation}

\end{theorem}

\begin{proof}
We will use the well-known fact from real analysis that if a sequence $\{g_n\}_{n\in\N}$ of real-valued differentiable functions defined on a finite interval in $\R$ converges pointwise to some function $g$ and the sequence of the corresponding derivatives $\big\{\frac{\mathrm{d}g_n}{\mathrm{d}t} \big\}_{n\in\N}$ converges uniformly to some function $h$, then $g$ is differentiable and $\frac{\mathrm{d}g}{\mathrm{d}t} = h$.

Fix a point $x\in S^2$ and a constant $l \in (0,+\infty)$. For $n\in\N$ and $t\in\R$, define
\begin{equation}
P_n(t) = \frac{1}{n} \log \sum\limits_{y\in f^{-n}(x)} \deg_{f^n}(y)\exp (S_n(\phi + t\gamma)(y)).
\end{equation}

Observe that there exists a bounded subset $H$ of $\Holder{\alpha}(S^2,d)$ such that $\phi + t\gamma \in H$ for each $t\in (-l,l)$. Then by Lemma~\ref{lmDerivConv},
\begin{equation}
\frac{\mathrm{d}P_n}{\mathrm{d}t} (t) = \frac{\frac{1}{n}  \sum\limits_{y\in f^{-n}(x)}\deg_{f^n}(y)(S_n\gamma(y))\exp(S_n(\phi + t\gamma)(y))}{ \sum\limits_{y\in f^{-n}(x)}\deg_{f^n}(y)\exp(S_n(\phi + t\gamma)(y))}
\end{equation}
converges to $\int \! \gamma \,\mathrm{d}\mu_{\phi + t\gamma}$ as $n\longrightarrow +\infty$, uniformly in $t\in (-l,l)$.

On the other hand, by Proposition~\ref{propTopPressureDefPreImg}, for each $t\in (-l,l)$, we have 
\begin{equation}
\lim\limits_{n\to+\infty}P_n(t) = P(f,\phi + t\gamma).
\end{equation}

Hence $P(f,\phi + t\gamma)$ is differentiable with respect to $t$ on $(-l,l)$ and
$$
\frac{\mathrm{d}}{\mathrm{d}t} P(f,\phi + t\gamma) =\lim\limits_{n\to+\infty} \frac{\mathrm{d}P_n}{\mathrm{d}t} (t) = \int \!\gamma \,\mathrm{d}\mu_{\phi + t\gamma}.
$$

Since $l\in (0,+\infty)$ is arbitrary, the proof is complete.
\end{proof}

\begin{theorem}  \label{thmUniqueES}
Let $f\:S^2 \rightarrow S^2$ be an expanding Thurston map and $d$ be a visual metric on $S^2$ for $f$. Let $\phi\in \Holder{\alpha}(S^2,d)$ be a real-valued H\"{o}lder continuous function with an exponent $\alpha\in(0,1]$. Then there exists a unique equilibrium state $\mu_\phi$ for $f$ and $\phi$. Moreover, the map $f$ with respect to $\mu_\phi$ is forward quasi-invariant (i.e., for each Borel set $A\subseteq S^2$, if $\mu_\phi(A)=0$, then $\mu_\phi(f(A))=0$), and nonsingular (i.e., for each Borel set $A\subseteq S^2$, $\mu_\phi(A)=0$ if and only if $\mu_\phi(f^{-1}(A))=0$).
\end{theorem}
\begin{proof}
The existence is proved in Corollary~\ref{corExistES}.

We now prove the uniqueness.

Since $\phi \in \Holder{\alpha}(S^2,d)$, by Theorem~\ref{thmPisDiff} the function
$$
t\longmapsto P(f,\phi + t\gamma)
$$
is differentiable at $0$ for $\gamma \in \Holder{\alpha}(S^2,d)$. Recall that by Lemma~\ref{lmHolderDense} $\Holder{\alpha}(S^2,d)$ is dense in $\CCC(S^2)$ in the weak topology. We note that the topological pressure function $P(f,\cdot) \: \CCC(S^2)\rightarrow\R$ is convex continuous (see for example, \cite[Theorem~3.6.1 and Theorem~3.6.2]{PU10}). Thus by Theorem~\ref{thmUniqueTangent} with $V=\CCC(S^2)$, $x=\phi$, $U=\Holder{\alpha}(S^2,d)$, and $Q=P(f,\cdot)$, we get $\card \big( V_{\phi,P(f,\cdot)}^* \big) = 1$.

On the other hand, if $\mu$ is an equilibrium state for $f$ and $\phi$, then by (\ref{defMeasTheoPressure}) and (\ref{eqVPPressure}),
\begin{equation*}
 h_\mu (f)  + \int \! \phi \,\mathrm{d}\mu= P(f,\phi),
\end{equation*}
and for each $\gamma \in \CCC(S^2)$,
\begin{equation*}
 h_\mu (f)  + \int \! (\phi + \gamma)  \,\mathrm{d}\mu  \leq  P(f,\phi + \gamma).
\end{equation*}
So $\int \! \gamma \,\mathrm{d}\mu \leq P(f,\phi + \gamma) - P(f,\phi)$. Thus by (\ref{eqDefTangent}), the continuous functional $\gamma\longmapsto \int\! \gamma\,\mathrm{d}\mu$ on $\CCC(S^2)$ is in $V_{\phi,P(f,\cdot)}^*$. Since $\mu_\phi=u_\phi m_\phi$ defined in Theorem~\ref{thmMuExist} is an equilibrium state for $f$ and $\phi$, and $\card \big( V_{\phi,P(f,\cdot)}^* \big) = 1$, we get that each equilibrium state $\mu$ for $f$ and $\phi$ must satisfy $\int\! \gamma\,\mathrm{d}\mu = \int\! \gamma\,\mathrm{d}\mu_\phi$ for $\gamma\in\CC(S^2)$, i.e., $\mu=\mu_\phi$. 

The fact that the map $f$ is forward quasi-invariant and nonsingular with respect to $\mu_\phi$ follows from the corresponding result for $m_\phi$ in Theorem~\ref{thmMexists}, Lemma~\ref{lmTildePhiIsHolder}, and the fact that $m_{\widetilde\phi} = \mu_\phi$ from Corollary~\ref{corMandMuUnique}.
\end{proof}

\begin{remark}
Since the entropy map $\mu \longmapsto h_\mu(f)$ for an expanding Thurston map $f$ is \defn{affine} (see for example, \cite[Theorem~8.1]{Wa82}), i.e., if $\mu,\nu\in\MMM(S^2,f)$ and $p\in[0,1]$, then $h_{p\mu+(1-p)\nu}(f) = p h_\mu(f) + (1-p) h_\nu(f)$, so is the pressure map $\mu\longmapsto P_\mu(f,\phi)$ for $f$ and a H\"older continuous potential $\phi\: S^2\rightarrow \R$. Thus the uniqueness of the equilibrium state $\mu_\phi$ and the Variational Principle (\ref{eqVPPressure}) imply that $\mu_\phi$ is an extreme point of the convex set $\MMM(S^2,f)$. It follows from the fact (see for example, \cite[Theorem~2.2.8]{PU10}) that the extreme points of $\MMM(S^2,f)$ are exactly the ergodic measures in $\MMM(S^2,f)$ that $\mu_\phi$ is ergodic. However, we are going to prove a much stronger ergodic property of $\mu_\phi$ in Section~\ref{sctErgodicity}.
\end{remark}

The following proposition is an immediate consequence of Theorem~\ref{thmRR^nConv}.

\begin{prop}    \label{propWeakConvR*^nProbToMu}
Let $f$, $d$, $\phi$ satisfy the Assumptions. Let $\mu_\phi$ be the unique equilibrium state for $f$ and $\phi$. Then for each Borel probability measure $\mu\in\PPP(S^2)$, we have
\begin{equation}  \label{eqWeakConvR*^nProbToMu}
\big(\RR^*_{\widetilde\phi} \big)^n (\mu)  \stackrel{w^*}{\longrightarrow} \mu_\phi \text{ as } n\longrightarrow +\infty.
\end{equation}
\end{prop}

\begin{proof}
Recall that for each $u\in\CCC(S^2)$, there exists some abstract modulus of continuity $h$ such that $u\in\CCC_h^{\beta} (S^2,d)$, where $\beta=\Norm{u}_\infty$. By Theorem~\ref{thmUniqueES} and Theorem~\ref{thmMuExist}, we have $\mu_\phi= u_\phi m_\phi$ as constructed in Theorem~\ref{thmMuExist}. Then by Lemma~\ref{lmRtildeNorm=1} and (\ref{eqRR^nConv0}) in Theorem~\ref{thmRR^nConv},
\begin{align*}
   & \lim\limits_{n\to+\infty}  \big\langle   \big(\RR^*_{\widetilde\phi} \big)^n (\mu) , u \big\rangle \\
=  &  \lim\limits_{n\to+\infty} \big( \big\langle   \mu, \RR^n_{\widetilde\phi}(u - \langle \mu_\phi, u \rangle \mathbbm{1} ) \big\rangle + \big\langle   \mu, \RR^n_{\widetilde\phi}(\langle \mu_\phi, u \rangle \mathbbm{1} ) \big\rangle  \big) \\
=  &   0 + \langle \mu , \langle \mu_\phi,u \rangle \mathbbm{1} \rangle    \\
=  &  \langle \mu_\phi, u\rangle,
\end{align*}
for each $u\in\CCC(S^2)$. Therefore, (\ref{eqWeakConvR*^nProbToMu}) holds.
\end{proof}

\section{Ergodic properties}   \label{sctErgodicity}
In this section, we first prove that if $f$, $\CC$, $d$, and $\phi$ satisfies the Assumptions, then any edge in the cell decompositions induced by $f$ and $\CC$ is a zero set with respect to the measures $m_\phi$ or $\mu_\phi$. This result is also important for Theorem~\ref{thmCohomologous}. We then show in Theorem~\ref{thmFisExact} that the measure-preserving transformation $f$ of the probability space $(S^2,\mu_\phi)$ is \emph{exact} (Definition~\ref{defExact}), and as an immediate consequence, mixing (Corollary~\ref{corMixing}). Another consequence of Theorem~\ref{thmFisExact} is that $\mu_\phi$ is non-atomic (Corollary~\ref{corNonAtomic}).

\smallskip

\begin{prop}    \label{propEdgeIs0set}
Let $f$, $\CC$, $n_\CC$, $d$, $\phi$, $\alpha$ satisfy the Assumptions. Let $\mu_\phi$ be the unique equilibrium state for $f$ and $\phi$, and $m_\phi$ be as in Corollary~\ref{corMandMuUnique}. Then 
\begin{equation}
m_\phi \(\bigcup_{i=0}^{+\infty} f^{-i}(\CC)\) = \mu_\phi \(\bigcup_{i=0}^{+\infty} f^{-i}(\CC)\)  = 0.
\end{equation}
\end{prop}

\begin{proof}
Since $\mu_\phi \in \MMM(S^2,f)$ is $f$-invariant, and $\CC\subseteq f^{-in_\CC}(\CC)$ for each $i\in\N$, we have $\mu_\phi \( f^{-in_\CC}(\CC) \setminus \CC  \) = 0$ for each $i\in\N$. Since $f$ is expanding, by Lemma~\ref{lmTileInIntTile}, there exists $m\in\N$ and an $(mn_\CC)$-tile $X\in\X^{mn_\CC}$ such that $X \cap \CC =\emptyset$. Then $\partial X \subseteq f^{mn_\CC}(\CC)\setminus \CC$. So $\mu_\phi (\partial X ) = 0$. Since $\mu_\phi = u_\phi m_\phi$, where $u_\phi$ is bounded away from $0$ (see Theorem~\ref{thmMuExist}), we have $m_\phi (\partial X ) = 0$. Note that $f^{mn_\CC}|_{\partial X }$ is a homeomorphism from $\partial X $ to $\CC$ (see Proposition~\ref{propCellDecomp}). Thus by the information on the Jacobian for $f$ with respect to $m_\phi$ in Theorem~\ref{thmMexists}, we get $m_\phi (\CC) = 0$. 

Now suppose there exist $k\in \N$ and a $k$-edge $e\in \E^k$ such that $m_\phi (e) > 0$. Then by using the Jacobian for $f$ with respect to $m_\phi$ again, we get $m_\phi (\CC) >0$, a contradiction. Hence $m_\phi \(\bigcup\limits_{i=0}^{+\infty} f^{-i}(\CC)\) = 0 $. Since $\mu_\phi = u_\phi m_\phi$, we get $\mu_\phi \(\bigcup\limits_{i=0}^{+\infty} f^{-i}(\CC)\)  = 0$.
\end{proof}

For each Borel measure $\mu$ on a compact metric space $(X,d)$, we denote by $\overline\mu$ the \defn{completion} of $\mu$, i.e., $\overline\mu$ is the unique measure defined on the smallest $\sigma$-algebra $\overline\BB$ containing all Borel sets and all subsets of $\mu$-null sets, satisfying $\overline\mu(E)=\mu(E)$ for each Borel set $E\subseteq X$.

\begin{definition}  \label{defExact}
Let $g$ be a measure-preserving transformation of a probability space $(X,\mu)$. Then $g$ is called \defn{exact} if for every measurable set $E$ with $\mu(E)>0$ and measurable images $g(E),g^2(E),\dots$, the following holds:
\begin{equation*}
\lim\limits_{n\to+\infty}\mu\( g^n(E)\) = 1.
\end{equation*}
\end{definition}

Note that in Definition~\ref{defExact}, we do not require $\mu$ to be a Borel measure. In the case when $g$ is a Thurston map on $S^2$ and $\mu$ is a Borel measure, the set $g^n(E)$ is a Borel set for each $n\in\N$ and each Borel set $E\subseteq S^2$. Indeed, a Borel set $E\subseteq S^2$ can be covered by $n$-tiles in the cell decompositions of $S^2$ induced by $g$ and any Jordan curve $\CC\subseteq S^2$ containing $\post g$. For each $n$-tile $X\in\X^n(f,\CC)$, the restriction $g^n|_{X}$ of $g^n$ to $X$ is a homeomorphism from the closed set $X$ onto $g^n(X)$ by Proposition~\ref{propCellDecomp}. It is then clear that the set $g^n(E)$ is also Borel.

We now prove that the measure-preserving transformation $f$ of the probability space $(S^2,\mu_\phi)$ is exact. The argument that we use here is similar to that in the proof of the exactness of an open, topologically exact, distance-expanding self-map of a compact metric space equipped with a certain Gibbs state (\cite[Theorem~5.2.12]{PU10}).

\begin{theorem}   \label{thmFisExact}
Let $f\:S^2 \rightarrow S^2$ be an expanding Thurston map and $d$ be a visual metric on $S^2$ for $f$. Let $\phi\in \Holder{\alpha}(S^2,d)$ be a real-valued H\"{o}lder continuous function with an exponent $\alpha\in(0,1]$. Let $\mu_\phi$ be the unique equilibrium state for $f$ and $\phi$, and $\overline{\mu_\phi}$ its completion. 

Then the measure-preserving transformation $f$ of the probability space $(S^2,\mu_\phi)$ (resp.\ $(S^2,\overline{\mu_\phi})$) is exact.
\end{theorem}

\begin{proof}
We fix a Jordan curve $\CC\subseteq S^2$ that satisfies the Assumptions (see Theorem~\ref{thmCexistsBM} for the existence of such $\CC$).

Since $\mu_\phi = u_\phi m_\phi$, by (\ref{eqU_phiBounds}), it suffices to prove that 
$$
\lim\limits_{n\to+\infty} m_\phi(S^2\setminus f^n(A)) = 0
$$
for each Borel set $A\subseteq S^2$ with $m_\phi(A) > 0$.

Let $A\subseteq S^2$ be an arbitrary Borel  subset of $S^2$ with $m_\phi (A) > 0$. Then there exists a compact set $E\subseteq A$ such that $m_\phi(E)>0$. Fix an arbitrary $\epsilon > 0$.  Since $f$ is expanding, by Lemma~\ref{lmTileInIntTile}, $n$-tiles have uniformly small diameters if $n$ is large. This and the outer regularity of the Borel measures enable us to choose $N\in\N$ such that for each $n\geq N$, the collection 
$$
\PP^n=\{X^n\in\X^n(f,\CC)  \,|\, X^n\cap E \neq \emptyset \}
$$
of $n$-tiles satisfies $m_\phi\(\bigcup\PP^n\) \leq m_\phi(E) + \epsilon$. Thus for each $n \geq N$, we have $m_\phi\Big( \bigcup\limits_{X^n\in\PP^n} X^n \setminus E\Big) \leq\epsilon$. So $\sum\limits_{X^n\in\PP^n} m_\phi\(X^n \setminus E\) \leq \epsilon$ by Proposition~\ref{propEdgeIs0set}. Hence
\begin{equation}
\frac{\sum\limits_{X^n\in\PP^n} m_\phi\(X^n \setminus E\)}{\sum\limits_{X^n\in\PP^n} m_\phi\(X^n\) } \leq \frac{\epsilon}{m_\phi(E)}.
\end{equation}
Thus for each $n\geq N$, there exists some $n$-tile $Y^n\in\PP^n$ such that 
\begin{equation}   \label{eqPfthmFisExact}
\frac{m_\phi(Y^n\setminus E)}{m_\phi(Y^n)}  \leq \frac{\epsilon}{m_\phi(E)}.
\end{equation}
By Proposition~\ref{propCellDecomp}(i), the map $f^n$ is injective on $Y^n$. So by Theorem~\ref{thmMexists}, Lemma~\ref{lmSnPhiBound}, (\ref{eqC2Bound}), and (\ref{eqPfthmFisExact}), we have
\begin{align*}
 & \frac{m_\phi \( f^n(Y^n)\setminus f^n(E)  \)}{ m_\phi \(f^n(Y^n) \)}
\leq  \frac{m_\phi \( f^n(Y^n\setminus E)  \)}{ m_\phi \(f^n(Y^n) \)}  \\
= & \frac{\displaystyle\int_{Y^n\setminus E} \! \exp(-S_n\phi) \,\mathrm{d}m_\phi}{\displaystyle\int_{Y^n} \! \exp(-S_n\phi) \,\mathrm{d}m_\phi} 
\leq C_2^2 \frac{m_\phi(Y^n\setminus E)}{m_\phi(Y^n)}  \leq  \frac{C_2^2 \epsilon}{m_\phi(E)},
\end{align*}
where $C_2 \geq 1$ is the constant defined in (\ref{eqC2Bound}) that depends only on $f$, $d$, $\phi$, and $\alpha$. By Lemma~\ref{lmTileInIntTile}, there exists $k\in\N$ that depends only on $f$ and $\CC$ such that $f^k(X_w^0) = f^k(X_b^0) = S^2$, where $X_w^0$ and $X_b^0$ are the while $0$-tile and the black $0$-tile, respectively. Since $f^n(Y^n)$ is either $X_w^0$ or $X_b^0$, by Proposition~\ref{propMfABounds}, for each $n\geq N$,
\begin{align*}
     &  m_\phi \( S^2\setminus f^{n+k}(E) \)   \leq  m_\phi \(f^k\( f^n(Y^n) \setminus f^n(E) \)\)  \\
\leq &  \int_{f^n(Y^n) \setminus f^n(E)}  \!  \exp(-S_k\phi) \,\mathrm{d}m_\phi  \leq \exp(k\Norm{\phi}_\infty ) \frac{C_2^2 \epsilon}{m_\phi(E) }.
\end{align*}
Since $\epsilon > 0$ was arbitrary, we get
\begin{equation}
\lim\limits_{n\to+\infty} m_\phi \(S^2 \setminus f^{n+k}(E)\) = 0.
\end{equation}
Thus
\begin{equation*}
\lim\limits_{n\to+\infty} m_\phi \(f^n(A)\) \geq \lim\limits_{n\to+\infty} m_\phi \(f^n(E)\) = 1.
\end{equation*}
Hence the measure-preserving transformation $f$ of the probability space $(S^2,\mu_\phi)$ is exact.

Next, we observe that since $f$ is $\mu_\phi$-measurable, and is a measure-preserving transformation of the probability space $(S^2,\mu_\phi)$, it is clear that $f$ is also $\overline{\mu_\phi}$-measurable, and is a measure-preserving transformation of the probability space $(S^2,\overline{\mu_\phi})$.

To prove that the measure-preserving transformation $f$ of the probability space $(S^2,\overline{\mu_\phi})$ is exact, we consider a $\overline{\mu_\phi}$-measurable set $B\subseteq S^2$ with $\overline{\mu_\phi}(B)>0$. Since $\overline{\mu_\phi}$ is the completion of the Borel probability measure $\mu_\phi$, we can choose Borel sets $A$ and $C$ such that $A\subseteq B \subseteq C \subseteq S^2$ and $\overline{\mu_\phi}(B)=\overline{\mu_\phi}(A)=\overline{\mu_\phi}(C)=\mu_\phi(A)=\mu_\phi(C)$. For each $n\in\N$, we have $f^n(A)\subseteq f^n(B) \subseteq f^n(C)$ and both $f^n(A)$ and $f^n(C)$ are Borel sets (see the discussion following Definition~\ref{defExact}). Since $f$ is forward quasi-invariant with respect to $\mu_\phi$ (see Theorem~\ref{thmUniqueES}), it is clear that $\mu_\phi\( f^n(A) \) = \mu_\phi\( f^n(C) \)$. Thus
\begin{equation*}
\mu_\phi\( f^n(A) \) = \overline{\mu_\phi}\( f^n(A) \) = \overline{\mu_\phi}\( f^n(B) \) = \overline{\mu_\phi}\( f^n(C) \) = \mu_\phi\( f^n(C) \).
\end{equation*}
Therefore, $\lim\limits_{n\to+\infty} \overline{\mu_\phi} \(f^n(B)\)=\lim\limits_{n\to+\infty} \mu_\phi \(f^n(A)\)=1$.
\end{proof}

Let $\mu$ be a measure on a topological space $X$. Then $\mu$ is called \defn{non-atomic} if $\mu(\{x\})=0$ for each $x\in X$.

The following corollary strengthens Theorem~\ref{thmMexists}.

\begin{cor}   \label{corNonAtomic}
Let $f$, $d$, $\phi$, $\alpha$ satisfy the Assumptions. Let $\mu_\phi$ be the unique equilibrium state for $f$ and $\phi$, and $m_\phi$ be as in Corollary~\ref{corMandMuUnique}. Then both $\mu_\phi$ and $m_\phi$ as well as their corresponding completions are non-atomic.
\end{cor}

\begin{proof}
Since $\mu_\phi=u_\phi m_\phi$, where $u_\phi$ is bounded away from $0$ (see Theorem~\ref{thmMuExist}), it suffices to prove that $\mu_\phi$ is non-atomic.

Suppose there exists a point $x\in S^2$ with $\mu_\phi(\{x\}) > 0$, then for all $y\in S^2$, we have
\begin{equation*}
\mu_\phi(\{y\}) \leq \max \{\mu_\phi(\{x\}), 1-\mu_\phi(\{x\})  \}.
\end{equation*}
Since the transformation $f$ of $(S^2,\mu_\phi)$ is exact by Theorem~\ref{thmFisExact}, we get that $\mu_\phi(\{x\}) = 1$ and $f(x)=x$.

We fix a Jordan curve $\CC\subseteq S^2$ that satisfies the Assumptions (see Theorem~\ref{thmCexistsBM} for the existence of such $\CC$). It is clear from Lemma~\ref{lmTileInIntTile} that there exist $n\in\N$ and an $n$-tile $X^n\in\X^n(f,\CC)$ with $x\notin X^n$. Then $\mu_\phi(X^n) = 0$, which contradicts with the fact that $\mu_\phi$ is a Gibbs state for $f$, $\CC$, and $\phi$ (see Theorem~\ref{thmMuExist} and Definition~\ref{defGibbsState}).

The fact that the completions are non-atomic now follows immediately.
\end{proof}

Let $f$, $d$, $\phi$, $\alpha$ satisfy the Assumptions. Let $\mu_\phi$ be the unique equilibrium state for $f$ and $\phi$, and $\overline{\mu_\phi}$ its completion. Then by Theorem~2.7 in \cite{Ro49}, the complete separable metric space $(S^2,d)$ equipped the complete non-atomic measure $\overline{\mu_\phi}$ is a Lebesgue space in the sense of V.~Rokhlin. We omit V.~Rokhlin's definition of a \emph{Lebesgue space} here and refer the reader to \cite[Section~2]{Ro49}, since the only results we will use about Lebesgue spaces are V.~Rokhlin's definition of exactness of a measure-preserving transformation on a Lebesgue space and its implication to the mixing properties. More precisely, in \cite{Ro61}, V.~Rokhlin gave a definition of exactness for a measure-preserving transformation on a Lebesgue space equipped with a complete non-atomic measure, and showed \cite[Section~2.2]{Ro61} that in such a context, it is equivalent to our definition of exactness in Definition~\ref{defExact}. Moreover, he proved \cite[Section~2.6]{Ro61} that if a measure-preserving transformation on a Lebesgue space equipped with a complete non-atomic measure is exact, then it is \emph{mixing} (he actually proved that it is \emph{mixing of all degrees}, which we will not discuss here).

Let us recall the definition of mixing for a measure-preserving transformation.

\begin{definition}  \label{defMixing}
Let $g$ be a measure-preserving transformation of a probability space $(X,\mu)$. Then $g$ is called \defn{mixing} if for all measurable sets $A,B\subseteq X$, the following holds:
\begin{equation*}
\lim\limits_{n\to+\infty}\mu\( g^{-n}(A) \cap B \) = \mu(A) \cdot \mu(B).
\end{equation*}
We call $g$ \defn{ergodic} if for each measurable set $E\subseteq X$, $g^{-1}(E)=E$ implies either $\mu(E)=0$ or $\mu(E)=1$.
\end{definition}

It is well-known and easy to see that if $g$ is mixing, then it is ergodic (see for example, \cite{Wa82}).

\begin{cor}  \label{corMixing}
Let $f$, $d$, $\phi$, $\alpha$ satisfy the Assumptions. Let $\mu_\phi$ be the unique equilibrium state for $f$ and $\phi$, and $\overline{\mu_\phi}$ its completion. Then the measure-preserving transformation $f$ of the probability space $(S^2,\mu_\phi)$ (resp.\ $(S^2,\overline{\mu_\phi})$) is mixing and ergodic.
\end{cor}

\begin{proof}
By the discussion preceding Definition~\ref{defMixing}, we know that the measure-preserving transformation $f$ of $(S^2,\overline{\mu_\phi})$ is mixing and thus ergodic. Since any $\mu_\phi$-measurable sets $A,B\subseteq S^2$ are also $\overline{\mu_\phi}$-measurable,  the measure-preserving transformation $f$ of $(S^2,\mu_\phi)$ is also mixing and ergodic.
\end{proof}

\section{Co-homologous potentials}    \label{sctCohomologous}
The goal of this section is to prove in Theorem~\ref{thmCohomologous} that two equilibrium states are identical if and only if there exists a constant $K\in\R$ such that $K\mathbbm{1}_{S^2}$ and the difference of the corresponding potentials are \emph{co-homologous} (see Definition~\ref{defCohomologous}). We use some of the ideas from \cite{PU10} in the process of proving Theorem~\ref{thmCohomologous}. We establish a form of the \emph{closing lemma} for expanding Thurston maps in Lemma~\ref{lmClosingLemma}.

\begin{theorem}    \label{thmCohomologous}
Let $f\: S^2\rightarrow S^2$ be an expanding Thurston map, and $d$ be a visual metric on $S^2$ for $f$. Let $\phi,\psi\in\Holder{\alpha}(S^2,d)$ be real-valued H\"older continuous functions with an exponent $\alpha\in (0,1]$. Let $\mu_\phi$ (resp.\ $\mu_\psi$) be the unique equilibrium state for $f$ and $\phi$ (resp.\ $\psi$). Then $\mu_\phi =\mu_\psi$ if and only if there exists a constant $K\in\R$ such that $\phi-\psi$ and $K\mathbbm{1}_{S^2}$ are co-homologous in the space $\CCC(S^2)$ of real-valued continuous functions.
\end{theorem}

\begin{definition}   \label{defCohomologous}
Let $g\: X\rightarrow X$ be a continuous map on a metric space $(X,d)$. Let $\mathcal{K} \subseteq \CCC(X)$ be a subspace of the space $\CCC(X)$ of real-valued continuous function on $X$. Two functions $\phi,\psi \in \CCC(X)$ are said to be \defn{co-homologous (in $\mathcal{K}$)} if there exists $u\in\mathcal{K}$ such that $\phi-\psi = u\circ g - u$.
\end{definition}

\begin{rem}
As we will see in the proof of Theorem~\ref{thmCohomologous} at the end of this section, if $\mu_\phi=\mu_\psi$ then the corresponding $u$ can be chosen from $\Holder{\alpha}(S^2,d)$.
\end{rem}

\begin{lemma}     \label{lmFixedPtInFlower}
Let $f$ and $\CC$ satisfy the Assumptions. If $f(\CC)\subseteq \CC$, then for $m,n\in\N$ with $m\geq n$ and each $m$-vertex $v^m \in \V^m(f,\CC)$ with $\overline{W^m(v^m)} \subseteq W^{m-n}(f^n(v^m))$, there exists $x\in\overline{W^m(v^m)}$ such that $f^n(x)=x$.
\end{lemma}

Here $\overline{W^m(v^m)}$ denotes the closure of the open set $W^m(v^m)$.

\begin{proof}
Since $v^m\in W^{m-n}(f^n(v^m))$ and $f(\CC)\subseteq \CC$, depending on the location of $v^m$, there are exactly three cases, namely, (i) $v^m=f^n(v^m)$; (ii) $v^m$ is contained in the interior of some $(m-n)$-edge; (iii) $v^m$ is contained in the interior of some $(m-n)$-tile. We will find a fixed point $x\in \overline{W^m(v^m)}$ of $f^n$ in each case.

\smallskip

\emph{Case 1.} When $v^m = f^n(v^m)$, we can just set $x=v^m$.

\smallskip

\emph{Case 2.} When $v^m \in \inte (e^{m-n})$ for some $(m-n)$-edge $e^{m-n} \in\E^{m-n}$ with $\inte (e^{m-n}) \subseteq W^{m-n}\(f^n(v^m)\)$, it is clear that $W^m(v^m) \subseteq X_1 \cup X_2$ when $X_1,X_2\in\X^{m-n}$ form the unique pair of distinct $(m-n)$-tiles contained in $W^{m-n}\(f^n(v^m)\)$ with $X_1\cap X_2 = e^{m-n}$. We can choose a pair of distinct $m$-tiles $Y_1, Y_2 \in \X^m$ with $Y_1\cup Y_2 \subseteq \overline{W^m(v^m)}$, $f^n(Y_1)=X_1$, $f^n(Y_2)=X_2$, and $Y_1\cap Y_2 = e^m$ for some $m$-edge $e^m\in\E^m$. If either $Y_1\subseteq X_1$ or $Y_2\subseteq X_2$, say $Y_2\subseteq X_2$, then since $X_2$ is homeomorphic to the closed unit disk in $\R^2$, and $f^n$ maps $Y_2$ homeomorphically onto $X_2$ (Proposition~\ref{propCellDecomp}(i)), we can conclude by applying Brouwer's Fixed Point Theorem on $((f^n)|_{Y_2})^{-1}$ that there exists a fixed point $x\in Y_2$ of $f^n$. (See for example, Figure~\ref{figFixPt1}.) So we can assume without loss of generality that $Y_1\subseteq X_2$ and $Y_2\subseteq X_1$. Suppose now that $\inte (e^m)\subseteq \inte (X_i)$, then $Y_1\cup Y_2 \subseteq X_i$, for $i\in\{1,2\}$. So $e^m\subseteq e^{m-n}$. Since $f^n$ maps $e^m$ homeomorphically onto $e^{m-n}$ by Proposition~\ref{propCellDecomp}(i), and $e^{m-n}$ is homeomorphic to the closed unit interval in $\R$, it is clear that there exists a fixed point $x\in e^m$ of $f^n$. (See for example, Figure~\ref{figFixPt2}.)

\begin{figure}
    \centering
    \begin{picture}(240,240)
      \put(20,120){\line(-1,-6){20}}   
      \put(20,120){\line(-1,6){20}}
      \put(120,20){\line(-6,-1){120}}
      \put(120,20){\line(6,-1){120}}
      \put(220,120){\line(1,-6){20}}
      \put(220,120){\line(1,6){20}}
      \put(120,220){\line(-6,1){120}}
      \put(120,220){\line(6,1){120}}
      \put(120,20){\line(0,1){200}}
      \put(20,120){\line(1,0){200}}
      \put(140,100){\line(1,1){40}}      
      \put(140,140){\line(1,-1){40}}
      \put(160,100){\line(0,1){40}}
      \put(160,100){\line(-1,-1){10}}
      \put(160,100){\line(1,-1){10}}
      \put(160,140){\line(1,1){10}}  
      \put(160,140){\line(-1,1){10}}      
      \put(140,120){\line(-1,1){10}}  
      \put(140,120){\line(-1,-1){10}}      
      \put(180,120){\line(1,1){10}}  
      \put(180,120){\line(1,-1){10}}       
      \put(130,130){\line(1,1){20}}  
      \put(130,110){\line(1,-1){20}}  
      \put(190,130){\line(-1,1){20}}  
      \put(190,110){\line(-1,-1){20}} 
      \put(120,120){\circle*{5}}
      \put(83,110){$f^n(v^m)$}
      \linethickness{2pt}
      \put(120,120){\line(1,0){100}} 
      \put(194,110){$e^{m-n}$}
      \put(180,50){$X_1$}      
      \put(180,185){$X_2$} 
      \put(146,135){$Y_1$}      
      \put(163,135){$Y_2$}         
    \end{picture}
    \caption{An example for Case 2 when $Y_2\subseteq X_2$.}
    \label{figFixPt1}
\end{figure}

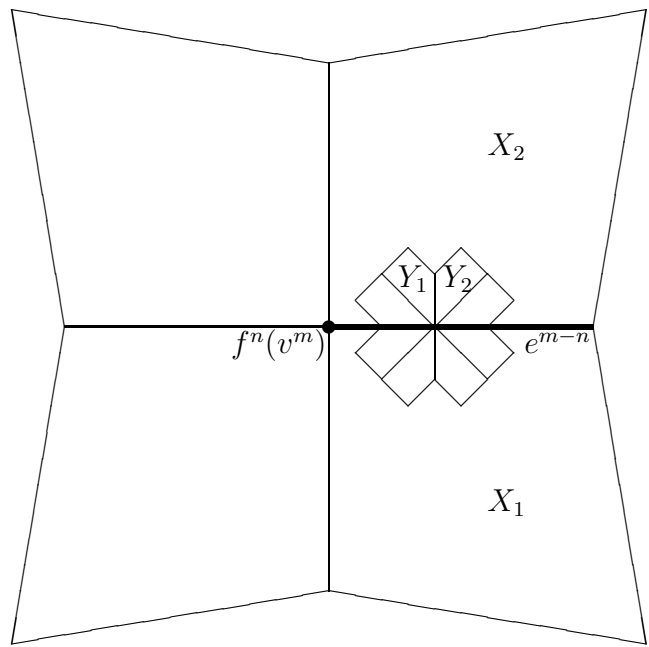
\begin{figure}
    \centering
    \begin{picture}(240,240)
      \put(20,120){\line(-1,-6){20}}   
      \put(20,120){\line(-1,6){20}}
      \put(120,20){\line(-6,-1){120}}
      \put(120,20){\line(6,-1){120}}
      \put(220,120){\line(1,-6){20}}
      \put(220,120){\line(1,6){20}}
      \put(120,220){\line(-6,1){120}}
      \put(120,220){\line(6,1){120}}
      \put(120,20){\line(0,1){200}}
      \put(20,120){\line(1,0){200}}
      \put(140,100){\line(1,1){40}}      
      \put(140,140){\line(1,-1){40}}
      \put(160,100){\line(0,1){40}}
      \put(160,100){\line(-1,-1){10}}
      \put(160,100){\line(1,-1){10}}
      \put(160,140){\line(1,1){10}}  
      \put(160,140){\line(-1,1){10}}      
      \put(140,120){\line(-1,1){10}}  
      \put(140,120){\line(-1,-1){10}}      
      \put(180,120){\line(1,1){10}}  
      \put(180,120){\line(1,-1){10}}       
      \put(130,130){\line(1,1){20}}  
      \put(130,110){\line(1,-1){20}}  
      \put(190,130){\line(-1,1){20}}  
      \put(190,110){\line(-1,-1){20}} 
      \put(120,120){\circle*{5}}
      \put(83,110){$f^n(v^m)$}
      \linethickness{2pt}
      \put(120,120){\line(1,0){100}} 
      \put(194,110){$e^{m-n}$}
      \put(180,50){$X_1$}      
      \put(180,185){$X_2$} 
      \put(138,123){$Y_1$}      
      \put(138,110){$Y_2$}          
    \end{picture}
    \caption{An example for Case 2 when $Y_1\nsubseteq X_1$, $Y_2\nsubseteq X_2$.}
    \label{figFixPt2}
\end{figure}

\smallskip

\emph{Case 3.} When $v^m\in\inte (X^{m-n})$ for some $(m-n)$-tile $X^{m-n} \in \X^{m-n}$ contained in $\overline{W^{m-n}\(f^n(v^m)\)}$, it is clear that $W^m(v^m)\subseteq X^{m-n}$. Let $X^m \in \X^m$ be an $m$-tile contained in $\overline{W^m(v^m)}$ such that $f^n(X^m)=X^{m-n}$. Since $X^{m-n}$ is homeomorphic to the closed unit disk in $\R^2$, and $f^n$ maps $X^m$ homeomorphically onto $X^{m-n}$ (Proposition~\ref{propCellDecomp}(i)), we can conclude by applying Brouwer's Fixed Point Theorem on $((f^n)|_{X^m})^{-1}$ that there exists a fixed point $x\in X^m$ of $f^n$.
\end{proof}

\begin{lemma}    \label{lmUinW}
Let $f$ and $\CC$ satisfy the Assumptions. Then there exists a number $\kappa\in\N_0$ such that the following statement holds:

For each $x\in S^2$, each $n\in\N_0$, and each $n$-tile $X^n\in\X^n(f,\CC)$, if $x\in X^n$, then there exists an $n$-vertex $v^n\in\V^n(f,\CC) \cap X^n$ with 
\begin{equation}   \label{eqUinW}
U^{n+\kappa}(x)\subseteq W^n(v^n).
\end{equation}
\end{lemma}

\begin{proof}
We will first find $\kappa\in\N_0$ such that the statement above holds when $n=0$. We will then show that the same $\kappa$ works for arbitrary $n\in\N_0$.

We fix a visual metric $d$ on $S^2$ for $f$ with an expansion factor $\Lambda>1$.

Note that the collection of $0$-flowers $\{W^0(v^0) \,|\, v^0\in\V^0 \}$ forms a finite open cover of $S^2$. By the Lebesgue Number Lemma (\cite[Lemma~27.5]{Mu00}), there exists a number $\epsilon>0$ such that any set of diameter at most $\epsilon$ is a subset of $W^0(v^0)$ for some $v^0\in\V^0$. Here $\epsilon$ depends only on $f$, $\CC$, and $d$. Then by Proposition~\ref{lmCellBoundsBM}(iii), there exists $\kappa\in\N_0$ depending only on $f$, $\CC$, and $d$ such that $\diam_d(U^\kappa(x))<\epsilon$ for $x\in S^2$. So for each $x\in S^2$, there exists a $0$-vertex $v^0\in \V^0$ such that $U^\kappa (x) \subseteq W^0(v^0)$. Let $X^0\in\X^0$ be a $0$-tile with $x\in X^0$, then clearly $v^0\in X^0$.

In general, we fix $x\in S^2$, $n\in \N_0$, and $X^n\in\X^n$ with $x\in X^n$. Set $A=\V^n\cap X^n$. By Proposition~\ref{propCellDecomp}, we have $f^n(W^n(v^n))= W^0(f^n(v^n))$ and $f^n(\partial W^n(v^n)) = \partial W^0(f^n(v^n))$ for each $v^n\in A$. Suppose $U^{n+\kappa}(x) \nsubseteq W^n(v^n)$ for all $v^n\in A$. Since $x\in X^n$ and $U^{n+\kappa}(x)$ is connected, we have $U^{n+\kappa}(x)\cap \partial W^n(v^n) \neq \emptyset$, and thus by Proposition~\ref{propCellDecomp}(i) 
$$
U^\kappa(f^n(x)) \cap \partial W^0(f^n(v^n)) \supseteq f^n(U^{n+\kappa}(x)) \cap f^n(\partial W^n(v^n)) \neq \emptyset,
$$
for each $v^n\in A$. Since $f^n(A)=\V^0$ by Proposition~\ref{propCellDecomp}, it follows that $U^\kappa(f^n(x)) \nsubseteq  W^0(v^0)$ for all $v^0\in \V^0$, contradicting the discussion above for the case when $n=0$.

Finally, we note that (\ref{eqUinW}) holds or fails independently of the choice of $d$. Therefore, the number $\kappa$ depends only on $f$ and $\CC$.
\end{proof}

The following result can be considered as a form of the \emph{closing lemma} for expanding Thurston maps. It is a key ingredient in the proof of Proposition~\ref{propCohomologousEquiv}, which will be used to prove Theorem~\ref{thmCohomologous}. Note that Lemma~\ref{lmClosingLemma} is more technical and in some sense slightly stronger than the closing lemma for forward-expansive maps (see \cite[Corollary~4.2.5]{PU10}). We need it in this slightly stronger form, since the distortion lemmas (Lemma~\ref{lmSnPhiBound} and Lemma~\ref{lmSigmaExpSnPhiBound}) cannot be applied in the proof of Proposition~\ref{propCohomologousEquiv}.

\begin{lemma}[Closing lemma]     \label{lmClosingLemma}
Let $f$, $\CC$, $d$, $\Lambda$ satisfy the Assumptions. If $f(\CC)\subseteq\CC$, then there exist $M\in\N_0$, $\delta_0\in(0,1)$, and $\beta_0>1$ such that the following statement holds:

For each $\delta\in  (0, \delta_0 ]$, if $x\in S^2$ and $l\in\N$ satisfy $l>M$ and $d(x,f^l(x))\leq \delta$, then there exists $y\in S^2$ such that $f^l(y)=y\in U^{N+l}(x)$ and $d(f^i(x),f^i(y))\leq \beta_0\delta\Lambda^{-(l-i)}$ for each $i\in\{0,1,\dots,l\}$, where $N=\big\lceil - \log_\Lambda\(\delta_0^{-1} \delta\)\big\rceil \in \N_0$.
\end{lemma}

\begin{proof}
Define
\begin{align}  
\delta_0 & = (2K)^{-1} \Lambda^{-(\kappa+1)},  \label{eqPfClosingLmDelta0}  \\
\beta_0  & = 4K^2\Lambda^{\kappa+1} = 2K\delta_0^{-1},    \label{eqPfClosingLmBeta0}   \\
M        & =\big\lceil \log_\Lambda \(10 K^2\) \big\rceil + \kappa \in\N_0, \label{eqPfClosingLmM}  
\end{align}
where $K\geq 1$ and $\kappa \in \N_0$ are constants depending only on $f$, $\CC$, and $d$ from Lemma~\ref{lmCellBoundsBM} and Lemma~\ref{lmUinW}, respectively.

We fix $\delta\in(0, \delta_0 ]$ and set 
\begin{equation}      \label{eqPfClosingLmBeta}
\beta = \beta_0 \delta.  
\end{equation}
Note that
$
N   =\big\lceil - \log_\Lambda \(\delta_0^{-1} \delta \)\big\rceil    =\big\lceil - \log_\Lambda \frac{\beta}{2K} \big\rceil \in \N_0
$
by (\ref{eqPfClosingLmBeta}) and (\ref{eqPfClosingLmBeta0}). So
\begin{equation}   \label{eqPfClosingLmBetaBounds}
2K \Lambda^{-N} \leq  \beta \leq 2K \Lambda^{-N+1},
\end{equation}
and by (\ref{eqPfClosingLmBeta}) and (\ref{eqPfClosingLmBeta0}), we have
\begin{equation}   \label{eqPfClosingLmAlphaBound}
\delta \leq (2K)^{-1} \Lambda^{-(N+\kappa)}.
\end{equation}

Recall that by Lemma~\ref{lmCellBoundsBM}(iii), for $z\in S^2$ and $n\in\N_0$, we have
\begin{equation}   \label{eqPfClosingLmUxRadii}
B_d(z,K^{-1}\Lambda^{-n}) \leq U^n(z) \leq B_d(z,K\Lambda^{-n}).
\end{equation}

Fix $x\in S^2$ and $l\in\N$ as in the lemma. Let $X^N \in \X^N $ be an $N$-tile containing $f^l(x)$. By Lemma~\ref{lmUinW}, there exists an $N$-vertex $v^N\in \V^N \cap X^N$ such that
\begin{equation}    \label{eqPfClosingLmUInW}
U^{N+\kappa}\(f^l(x)\) \subseteq W^N\(v^N\).
\end{equation}
There exist $X^{N+l} \in \X^{N+l}$ and $v^{N+l} \in \V^{N+l}\cap X^{N+l}$ such that $x\in X^{N+l}$, $f^l\(X^{N+l}\) = X^N$, and $f^l\(v^{N+l}\)= v^N$. Since $l>M$ and $W^{N+l}\( v^{N+l}\) \subseteq U^{N+l}(x)$,  we get from (\ref{eqPfClosingLmAlphaBound}), (\ref{eqPfClosingLmUxRadii}), and (\ref{eqPfClosingLmM}) that if $z\in W^{N+l}\(v^{N+l}\)$, then
\begin{align*}
 d\(f^l(x),z\)  \leq & d \( f^l(x), x\) + d(x,z)   \leq \delta + 2K \Lambda^{-(N+l)} \\
                \leq & \frac{\Lambda^{-(N+\kappa)}}{2K}  + \frac{2K\Lambda^{-(N+\kappa)}}{10K^2} \leq K^{-1} \Lambda^{-(N+\kappa)}.
\end{align*}
Thus by (\ref{eqPfClosingLmUxRadii}) and (\ref{eqPfClosingLmUInW}), we get
$$
 \overline{ W^{N+l}\(v^{N+l}\)} \subseteq U^{N+\kappa} \(f^l(x)\) \subseteq W^N\(v^N\).
$$
By Lemma~\ref{lmFixedPtInFlower}, there exists $y \in  \overline{ W^{N+l}\(v^{N+l}\)} \subseteq U^{N+l}(x)$ such that $f^l(y)=y$.

It suffices now to verify that $d\(f^i(x),f^i(y)\) \leq \beta_0 \delta \Lambda^{-(l-i)}$ for $i\in\{0,1,\dots,l\}$. Indeed, since by Proposition~\ref{propCellDecomp}, 
$$
\{f^i(x), f^i(y)\} \subseteq \overline{W^{N+l-i} \(f^i\(v^{N+l}\)\)} \subseteq U^{N+l-i} \(f^i\(v^{N+l}\)\)
$$
for $i\in\{0,1,\dots,l\}$, we get from (\ref{eqPfClosingLmUxRadii}),  (\ref{eqPfClosingLmBetaBounds}), and (\ref{eqPfClosingLmBeta}) that 
$$
d\(f^i(x),f^i(y)\) \leq 2K\Lambda^{-(N+l-i)} \leq \beta \Lambda^{-(l-i)}= \beta_0\delta \Lambda^{-(l-i)}.
$$
\end{proof}

The next lemma follows from the \emph{topological transitivity} (see \cite[Definition~4.3.1]{PU10}) of expanding Thurston maps and Lemma~4.3.4 in \cite{PU10}. We include a direct proof here for completeness.

\begin{lemma}   \label{lmwLimitPt}
Let $f\: S^2 \rightarrow S^2$ be an expanding Thurston map. Then there exists a point $x\in S^2$ such that the set $\{f^n(x) \,|\, n\in\N\}$ is dense in $S^2$.
\end{lemma}

\begin{proof}
By Theorem~1.6 in \cite{BM10}, the topological dynamical system $(S^2,f)$ is a factor of the topological dynamical system $(J^\omega,\Sigma)$ of the left-shift $\Sigma$ on the space $J^\omega$ of all infinite sequences in a finite set $J$ of cardinality $\card J = \deg f$. More precisely, if we equip $J^\omega =\prod\limits_{i=1}^{+\infty} J$ with the product topology, where $J=\{1,2,\dots,\deg f\}$, and let the left-shift operator $\Sigma$ map $(i_1,i_2,\dots)\in J^\omega$ to $(i_2,i_3,\dots)$, then there exists a surjective continuous map $\xi \: J^\omega \rightarrow S^2$ such that $\xi\circ \Sigma = f \circ \xi$.

It suffices now to find $y\in J^\omega$ such that the set $\{ \Sigma^n(y) \,|\, n\in \N \}$ is dense in $J^\omega$. Indeed, if we let $\{w_i\}_{i\in\N}$ be an enumeration of all elements in the set $\bigcup\limits_{i=1}^{+\infty} J^i$ of all finite sequences in $J$, and set $y$ to be the concatenation of $w_1, w_2, \dots$, then it is clear that $\{ \Sigma^n(y) \,|\, n\in \N \}$ is dense in $J^\omega$.
\end{proof}

Following similar argument as in the proof of Proposition~4.4.5 in \cite{PU10}, we get the next proposition. Note that here we do not explicitly use the distortion lemmas (Lemma~\ref{lmSnPhiBound} and Lemma~\ref{lmSigmaExpSnPhiBound}).

\begin{prop}    \label{propCohomologousEquiv}
Let $f$, $\CC$, $d$, $\Lambda$ satisfy the Assumptions. Let $\phi,\psi\in \Holder{\alpha}(S^2,d)$ be real-valued H\"older continuous functions with an exponent $\alpha\in(0,1]$. If $f(\CC)\subseteq \CC$, then the following conditions are equivalent:

\smallskip

\begin{enumerate}
\item[(i)] If $x\in S^2$ satisfies $f^n(x)= x$ for some $n\in\N$, then $S_n\phi(x)= S_n\psi(x)$.

\smallskip

\item[(ii)] There exists a constant $C>0$ such that $\abs{S_n\phi(x) - S_n\psi(x)} \leq C$ for $x\in S^2$ and $n\in\N_0$.

\smallskip

\item[(iii)] There exists $u\in\Holder{\alpha}(S^2,d)$ such that $\phi-\psi = u\circ f - u$.
\end{enumerate}
\end{prop}

\begin{proof}
The implication from (iii) to (ii) holds since $\abs{S_n\phi(x)-S_n\psi(x)}=\abs{(u\circ f^n )(x) - u(x)} \leq 2 \Norm{u}_\infty$ for $x\in S^2$ and $n\in\N$.

\smallskip

To prove that (ii) implies (i), we suppose that $f^n(x)=x$ and $D = S_n\phi(x) - S_n\psi(x)  \neq 0$ for some $x\in S^2$ and some $n\in\N$. Then $\abs{S_{ni}\phi(x) - S_{ni}\psi(x)}=iD > C$ for $i$ large enough, contradicting (ii).

\smallskip

We now prove the implication from (i) to (iii).

Let $x\in S^2$ be a point from Lemma~\ref{lmwLimitPt} so that the set $A=\{f^i(x) \,|\, i\in\N\}$ is dense in $S^2$. Set $x_i=f^i(x)$ for $i\in\N$. Note that $x_i\neq x_j$ for $j>i\geq 0$. Denote $\eta=\phi - \psi$. Then $\eta \in \Holder{\alpha}(S^2,d)$. We define a function $v$ on $A$ by setting $v(x_n)= S_n \eta(x)$. We will prove that $v$ extends to a H\"older continuous function $u\in \Holder{\alpha}(S^2,d)$ defined on $S^2$ by showing that $v$ is H\"older continuous with an exponent $\alpha$ on $A$.

Fix some $n,m\in \N$ with $n<m$ and $d(x_n,x_m)< \frac12 \delta_0$, where $\delta_0\in(0,1)$ is a constant depending only on $f$, $\CC$, and $d$ from Lemma~\ref{lmClosingLemma}. Set $\epsilon= d(x_n,x_m)$. We can choose $k\in\N$ such that $d(x_m,x_k)<\epsilon$ and $k>m+M$, where $M\in\N_0$ is a constant from Lemma~\ref{lmClosingLemma}. Note that $d(x_n,x_k)\leq d(x_n,x_m)+d(x_m,x_k)<2\epsilon < \delta_0$ and $k> n+M$. Thus by applying Lemma~\ref{lmClosingLemma} with $\delta=2\epsilon$, there exist periodic points $p,q\in S^2$ such that $f^{k-n}(p)=p$, $f^{k-m}(q)=q$, $d\(f^i(x_n),f^i(p)\) <\beta_0\delta \Lambda^{-(k-n-i)}$ for $i\in\{0,1,\dots, k-n\}$, and $d\(f^j(x_m),f^j(q)\) <\beta_0 \delta \Lambda^{-(k-m-j)}$ for $j\in\{0,1,\dots, k-m\}$, where $\beta_0 >0$ is a constant depending only on $f$, $\CC$, and $d$ from Lemma~\ref{lmUinW}. Then by (i), we get that  
\begin{align*}
                        &  \abs{v(x_n)-v(x_m)}            = \abs{S_n\eta(x)-S_m\eta(x)} \\
                   \leq &  \abs{S_{k-n}\eta(x_n)}+\abs{S_{k-m}\eta(x_m)} \\
                     =  &  \abs{S_{k-n}\eta(x_n) - S_{k-n}\eta(p) }+\abs{S_{k-m}\eta(x_m)-S_{k-m}\eta(q)} \\
                   \leq &  \sum\limits_{i=0}^{k-n-1}  \Abs{\eta\(f^i(x_n)\)-\eta\(f^i(p)\)} + \sum\limits_{j=0}^{k-m-1}  \Abs{\eta\(f^j(x_m)\)-\eta\(f^j(q)\)}  \\
                   \leq &  \Hseminorm{\alpha}{\eta} \beta_0^\alpha \delta^\alpha \(\sum\limits_{i=0}^{k-n-1}\Lambda^{-\alpha(k-n-i)} + \sum\limits_{j=0}^{k-m-1}\Lambda^{-\alpha(k-m-i)}\) \\
                   \leq & 2^{1+\alpha} \Hseminorm{\alpha}{\eta} \beta_0^\alpha \epsilon^\alpha \sum\limits_{i=0}^{\infty}\Lambda^{-\alpha i} 
                     = C d(x_n,x_m)^\alpha,
\end{align*}
where $C=2^{1+\alpha} (1-\Lambda^{-\alpha})^{-1}\Hseminorm{\alpha}{\eta} \beta_0^\alpha$ is a constant depending only on $f$, $\CC$, $d$, $\eta$, and $\alpha$. It immediately follows that $v$ extends continuously to a H\"older continuous function $u\in\Holder{\alpha}(S^2,d)$ with an exponent $\alpha$ defined on $\overline{A}=S^2$. Since $u|_A=v$ and 
\begin{align*}
(v\circ f)(x_i)-v(x_i) & = v(x_{i+1})-v(x_i)= S_{i+1}\eta(x)-S_i\eta(x)\\
                     & =\eta(f^i(x))=\phi(x_i)-\psi(x_i),
\end{align*}
for $i\in\N$, we get that $(u\circ f)(y)-u(y)=\phi(y)-\psi(y)$ for $y\in S^2$ by continuity.
\end{proof}

We are now ready to prove Theorem~\ref{thmCohomologous}.

\begin{proof}[Proof of Theorem~\ref{thmCohomologous}]
We fix a Jordan curve $\CC\subseteq S^2$ that satisfies the Assumptions (see Theorem~\ref{thmCexistsBM} for the existence of such $\CC$). 

We first prove the backward implication. We assume that
\begin{equation}   \label{eqPfthmCohomologousCohoEq}
\phi - \psi - K\mathbbm{1}_{S^2} = u\circ f - u
\end{equation}
for some $u\in\CCC(S^2)$ and $K\in\R$. It follows immediately from Proposition~\ref{propTopPressureDefPreImg} that
\begin{equation}  \label{eqPfthmCohomologousPressure}
P(f,\phi) =P(f,\psi)+K.
\end{equation}
By Theorem~\ref{thmMuExist}, Proposition~\ref{propTopPressureDefPreImg}, Corollary~\ref{corExistES}, and Theorem~\ref{thmUniqueES}, the measure $\mu_\phi$ (resp.\ $\mu_\psi$) is a Gibbs state with respect to $f$, $\CC$, and $\phi$ (resp.\ $\psi$) with constants $P_{\mu_\phi}=P(f,\phi)$ and $C_{\mu_\phi}$ (resp.\ $P_{\mu_\psi}=P(f,\psi)$ and $C_{\mu_\psi}$). Then by (\ref{eqGibbsState}), (\ref{eqPfthmCohomologousCohoEq}), and (\ref{eqPfthmCohomologousPressure}), for $i\in\N_0$ and $X^i\in\X^i(f,\CC)$,
\begin{align}
\frac{\mu_\phi(X^i)}{\mu_\psi(X^i)}  & \leq C_{\mu_\phi} C_{\mu_\psi} \frac{\exp(S_i\psi(x)- iP(f,\psi))}{\exp(S_i\phi(x)- iP(f,\phi))} \notag \\
                                     & = C_{\mu_\phi} C_{\mu_\psi} \exp(u(x)-(u\circ f)(x)) \label{eqPfthmCohomologousMuComparable}\\
                                     & \leq C_{\mu_\phi} C_{\mu_\psi} \exp(2\Norm{u}_\infty),  \notag
\end{align}
where $x\in X^i$. Let $E\subseteq S^2$ be a Borel set with $\mu_\psi(E)=0$. Fix an arbitrary number $\epsilon>0$. We can find an open set $U\subseteq S^2$ such that $E\subseteq U$ and $\mu_\psi(U)<\epsilon$. Set
$$
V= \bigcup \bigg\{ \inte (X) \,\bigg|\, X\in \bigcup\limits_{i=0}^{+\infty} \V^i(f,\CC), \, X\cap E \neq \emptyset, \, X\subseteq U \bigg\}.
$$
Then $E\subseteq V\cup A$, where $A=\bigcup\limits_{i=0}^{+\infty} f^{-i}(\CC)$. By Proposition~\ref{propEdgeIs0set}, we have $\mu_\phi(A)=\mu_\psi(A)=0$. So by (\ref{eqPfthmCohomologousMuComparable}), we get 
$$
\mu_\phi(E)\leq \mu_\phi(V) \leq D\mu_\psi(V) \leq D\mu_\psi(U) < D\epsilon,
$$
where $D=C_{\mu_\phi} C_{\mu_\psi} \exp(2\Norm{u}_\infty)$. Thus $\mu_\phi$ is absolutely continuous with respect to $\mu_\psi$. Similarly $\mu_\psi$ is absolutely continuous with respect to $\mu_\phi$. On the other hand, by Corollary~\ref{corMixing}, both $\mu_\phi$ and $\mu_\psi$ are ergodic measures. So suppose $\mu_\phi\neq \mu_\psi$, then they must be mutually singular (see for example, \cite[Theorem~6.10(iv)]{Wa82}). Hence $\mu_\phi=\mu_\psi$.

\smallskip

We will now prove the forward implication. We assume $\mu_\phi=\mu_\psi$. 

Denote $F=f^n$, where $n=n_\CC$ is a number from the Assumptions with $f^n(\CC)=F(\CC)\subseteq \CC$. By Remark~\ref{rmExpanding} the map $F$ is also an expanding Thurston map.

For the rest of the proof, we denote $S_m \eta = \sum\limits_{i=0}^{m-1} \eta\circ f^i$ and $\widetilde{S}_m \eta = \sum\limits_{i=0}^{m-1} \eta \circ F^i$ for $\eta\in\CCC(S^2)$ and $m\in\N_0$.

Denote $\phi_n=S_n\phi$ and $\psi_n=S_n\psi$. It follows immediately from Lemma~\ref{lmLipschitz} that $\phi_n, \psi_n \in \Holder{\alpha}(S^2,d)$.

Note that since $\mu_\phi$ is an equilibrium state for $f$ and $\phi$, it follows that $\mu_\phi$ is also an equilibrium state for $F$ and $\phi_n$. Indeed, by (\ref{defMeasTheoPressure}) and the fact that $h_{\mu_\phi}(f^n)= nh_{\mu_\phi}(f)$ (see for example, \cite[Theorem~4.13]{Wa82}), we have
\begin{align*}
P_{\mu_\phi}(F,\phi_n) & = h_{\mu_\phi}(f^n)+ \int  \! S_n\phi \, \mathrm{d} \mu_\phi = nh_{\mu_\phi}(f)+ n\int  \!  \phi \, \mathrm{d} \mu_\phi  \\
                       & = n P(f,\phi) = P(F, \phi_n),
\end{align*}
where the last equality follows immediately from Proposition~\ref{propTopPressureDefPreImg}. Similarly, the measure $\mu_\phi=\mu_\psi$ is an equilibrium state for $F$ and $\psi_n$.

Thus by Theorem~\ref{thmMuExist}, Proposition~\ref{propTopPressureDefPreImg}, Corollary~\ref{corExistES}, and Theorem~\ref{thmUniqueES}, the measure $\mu_\phi = \mu_\psi$ is both a Gibbs state with respect to $F$, $\CC$, and $\phi_n$, and with constants $P(F,\phi_n)$ and $C$, as well as a Gibbs state with respect to $F$, $\CC$, and $\psi_n$, and with constants $P(F,\psi_n)$ and $C'$, for some $C\geq 1$ and $C'\geq 1$. By (\ref{eqGibbsState}), we have
$$
\frac{1}{C  C'}\leq \exp\big(\widetilde{S}_m\phi_n(x)-\widetilde{S}_m\psi_n(x) - mP(F,\phi_n)+mP(F,\psi_n)\big) \leq C  C'
$$
for $x\in S^2$ and $m\in\N_0$. So $\Abs{\widetilde{S}_m\overline{\phi}_n(x)-\widetilde{S}_m\overline{\psi}_n(x)} \leq \log(C C')$ for $x\in S^2$ and $m\in \N_0$, where $\overline{\phi}_n(x) = \phi_n(x)-P(F,\phi_n)\in \Holder{\alpha}(S^2,d)$ and $\overline{\psi}_n(x) = \psi_n(x)-P(F,\psi_n)\in \Holder{\alpha}(S^2,d)$. By Proposition~\ref{propCohomologousEquiv}, there exists $u\in\Holder{\alpha}(S^2,d)$ such that
\begin{equation}   \label{eqPfthmCohomologousStepF}
(u\circ f^n)(x) - u(x) = \overline{\phi}_n(x) - \overline{\psi}_n(x) = S_n\phi(x) - S_n\psi(x) - \delta
\end{equation}
for $x\in S^2$, where $\delta= P(F,\phi_n)-P(F,\psi_n)$.

Fix an arbitrary point $y\in S^2$. By subtracting (\ref{eqPfthmCohomologousStepF}) with $x= y$ from (\ref{eqPfthmCohomologousStepF}) with $x= f(y)$, we get
\begin{align*}
    & (u\circ f^{n+1})(y) - (u\circ f^n)(y) + (u\circ f)(y) - u(y) \\
 =  & (\phi\circ f^n)(y)-\phi(y) - (\psi\circ f^n) (y) + \psi(y),
\end{align*}
or equivalently,
\begin{align}     \label{eqPfthmCohomologousInduction}
   & \phi ( f^n(y)) - \psi ( f^n (y)) - (u\circ f)(f^n(y)) + u(f^n(y)) \\
=  & \phi(y)-\psi(y)-(u\circ f)(y)+ u(y).   \notag
\end{align}

Let $z\in S^2$ be a point from Lemma~\ref{lmwLimitPt} so that the set $A=\{f^{ni}(z) \,|\, i\in\N\}$ is dense in $S^2$. By replacing $y$ in (\ref{eqPfthmCohomologousInduction}) with $f^{ni}(z)$ for $i\in\N_0$ and induction, we get that
$$
\phi(f^{ni}(z)) - \psi ( f^{ni} (z)) - (u\circ f)(f^{ni}(z)) + u(f^{ni}(z)) = K
$$
for $i\in\N$, where $K = \phi(z)-\psi(z) - (u\circ f)(z)+u(z)$. Since $A$ is dense in $S^2$, we get that $\phi(x)-\psi(x)-(u\circ f)(x)+ u(x) = K$ for $x\in S^2$, i.e., the functions $\phi-\psi$ and $K\mathbbm{1}_{S^2}$ are co-homologous in $\Holder{\alpha}(S^2,d)$.
\end{proof}

\section{Equidistribution}     \label{sctEquidistribution}

In this section, we will discuss equidistribution results for preimages. Let $f$, $d$, $\phi$, $\alpha$ satisfy the Assumptions and let $\mu_\phi$ be the unique equilibrium state for $f$ and $\phi$ throughout this section. We prove in Proposition~\ref{propWeakConvPreImgWithWeight} three versions of equidistribution of preimages under $f^n$ as $n\longrightarrow +\infty$ with respect to $\mu_\phi$ and $m_\phi$ as defined in Corollary~\ref{corMandMuUnique}, respectively. Proposition~\ref{propWeakConvPreImgWithWeight} partially generalizes Theorem~1.2 in \cite{Li13}, where we established the equidistribution of preimages with respect to the measure of maximal entropy. In Theorem~\ref{thmASConvToES}, we generalizes Theorem~7.1 in \cite{Li13} following the idea of J.~Hawkins and M.~Taylor \cite{HT03}, to show that for each $p\in S^2$, the equilibrium state $\mu_\phi$ is almost surely the limit of 
$$
\frac{1}{n}\sum\limits_{i=0}^{n-1} \delta_{q_i}
$$
as $n\longrightarrow +\infty$ in the weak$^*$ topology, where $q_0 = p$, and for each $i\in\N_0$, the point $q_{i+1}$ is one of the points $x$ in $f^{-1}(q_i)$, chosen with probability proportional to $\deg_f(x) \exp{\widetilde\phi(x)}$, where $\widetilde\phi$ is defined in (\ref{eqDefPhiTilde}).

\begin{prop}  \label{propWeakConvPreImgWithWeight}
Let $f$, $d$, $\phi$, $\alpha$ satisfy the Assumptions. Let $\mu_\phi$ be the unique equilibrium state for $f$ and $\phi$, and $m_\phi$ be as in Corollary~\ref{corMandMuUnique} and $\widetilde\phi$ as defined in (\ref{eqDefPhiTilde}). For each sequence $\{ x_n \}_{n\in\N}$ of points in $S^2$, we define the Borel probability measures
\begin{align}
         \nu_n & = \frac{1}{Z_n(\phi)} \sum\limits_{y\in f^{-n}(x_n)} \deg_{f^n} (y) \exp\(S_n\phi(y)\) \delta_y,  \\
\widehat\nu_n & = \frac{1}{Z_n(\phi)} \sum\limits_{y\in f^{-n}(x_n)} \deg_{f^n} (y) \exp\(S_n\phi(y)\) \frac{1}{n} \sum\limits_{i=0}^{n-1} \delta_{f^i(y)},\\
         \widetilde\nu_n & = \frac{1}{Z_n\big(\widetilde\phi \hspace{0.5mm} \big)} \sum\limits_{y\in f^{-n}(x_n)} \deg_{f^n} (y) \exp\big(S_n\widetilde\phi(y)\big) \delta_y, 
\end{align}
for each $n\in\N_0$, where $Z_n(\psi) = \sum\limits_{y\in f^{-n}(x_n)} \deg_{f^n} (y) \exp\(S_n\psi(y)\)$, for $\psi\in\CCC(S^2)$. Then
\begin{equation}  \label{eqWeakConvPreImgToMWithWeight}
\nu_n \stackrel{w^*}{\longrightarrow} m_\phi \text{ as } n\longrightarrow +\infty,
\end{equation}
\begin{equation}  \label{eqWeakConvPreImgToMuSumWithWeight}
\widehat\nu_n \stackrel{w^*}{\longrightarrow} \mu_\phi \text{ as } n\longrightarrow +\infty,
\end{equation}
\begin{equation}  \label{eqWeakConvPreImgToMuTildeWithWeight}
\widetilde\nu_n \stackrel{w^*}{\longrightarrow} \mu_\phi \text{ as } n\longrightarrow +\infty.
\end{equation}
\end{prop}

We note that when $\phi \equiv 0$ and $x_n=x_{n+1}$ for each $n\in\N$, the versions (\ref{eqWeakConvPreImgToMWithWeight}) and (\ref{eqWeakConvPreImgToMuTildeWithWeight}) reduce to (1.2) of Theorem~1.2 in \cite{Li13}.

\begin{proof}
We note that (\ref{eqWeakConvPreImgToMuSumWithWeight}) follows directly from Lemma~\ref{lmDerivConv}.

The proof of (\ref{eqWeakConvPreImgToMWithWeight}) is similar to that of Lemma~\ref{lmDerivConv}. For the completeness, we include it here in detail.

For each sequence $\{x_n\}_{n\in\N}$ of points in $S^2$, and each $u\in\CCC(S^2,d)$, by (\ref{eqR^nExpr}) and (\ref{eqDefR-}) we have
\begin{equation*}
  \langle \nu_n, u \rangle 
=  \frac{\RR^n_\phi (u)(x_n)}{\RR^n_\phi (\mathbbm{1})(x_n)}  
= \frac{\RR^n_{\overline\phi} (u)(x_n)}{\RR^n_{\overline\phi} (\mathbbm{1})(x_n)}.
\end{equation*}
By Theorem~\ref{thmRR^nConv},
\begin{equation*}
\Norm{\RR^n_{\overline\phi}(\mathbbm{1}) - u_\phi }_\infty  \longrightarrow  0 \text{ and }\Norm{\RR^n_{\overline\phi}(u) - u_\phi \int \! u \, \mathrm{d} m_\phi }_\infty  \longrightarrow  0
\end{equation*}
as $n\longrightarrow +\infty$. So by (\ref{eqU_phiBounds}),
\begin{equation*}
\lim\limits_{n\to+\infty} \frac{\RR^n_{\overline\phi} (u)(x_n)}{\RR^n_{\overline\phi} (\mathbbm{1})(x_n)} = \int \! u \, \mathrm{d}m_\phi.
\end{equation*}
Hence, (\ref{eqWeakConvPreImgToMWithWeight}) holds.

Finally, (\ref{eqWeakConvPreImgToMuTildeWithWeight}) follows from (\ref{eqWeakConvPreImgToMWithWeight}) and the fact that $\widetilde\phi \in \Holder{\alpha}(S^2,d)$ (Lemma~\ref{lmTildePhiIsHolder}) and $m_{\widetilde\phi} = \mu_\phi$ (Corollary~\ref{corMandMuUnique}).
\end{proof}

For the rest of this section, we prove that almost surely,
$$
\frac{1}{n}\sum\limits_{i=0}^{n-1} \delta_{q_i}  \stackrel{w^*}{\longrightarrow} \mu_\phi \text{ as } n\longrightarrow +\infty,
$$
where $q_i \in S^2$, $i\in \N_0$, is the location of the $i$-th step of a certain random walk on $S^2$ induced by $\RR_{\widetilde\phi}$ starting from an arbitrary fixed starting position $q_0 \in S^2$. This result generalizes Theorem~7.1 of \cite{Li13}, which is Theorem~\ref{thmASConvToES} in the case when $\phi\equiv 0$.

More precisely, let $Q=\RR_{\widetilde\phi}$. Then for each $u\in\CCC(S^2)$,
$$
Qu(x) = \int \! u(y) \,\mathrm{d}\mu_x(y),
$$
where
$$
\mu_x = \sum\limits_{z\in f^{-1}(x)} \deg_f(z) \exp\big( \widetilde\phi (z)\big) \delta_z.
$$
By (\ref{eqRtildePhi1=1}), we get that $\mu_x\in\PPP(S^2)$ for each $x\in S^2$. We showed that the Ruelle operator in (\ref{eqDefRuelleOp}) is well-defined, from which it immediately follows that the map $x\mapsto \mu_x$ from $S^2$ to $\PPP(S^2)$ is continuous with respect to weak$^*$ topology on $\PPP(S^2)$. The operator $Q$ (or equivalently, the measures $\mu_x$, $x\in S^2$) and an arbitrary starting point $q_0\in S^2$ determine a random walk $\{q_i\}_{i\in\N_0}$ on $S^2$ with the probability that $q_{i+1}\in A$ is equal to $\mu_{q_i}(A)$ for each $i\in\N_0$ and each Borel set $A\subseteq S^2$. In the language of \cite[Section~8]{Li13}, this random walk is a \emph{Markov process determined by the operator $Q$}. We refer the reader to \cite[Section~8]{Li13} for a more detailed discussion.

\begin{theorem}    \label{thmASConvToES}
Let $f$, $d$, $\phi$, $\alpha$ satisfy the Assumptions. Let $\mu_\phi$ be the unique equilibrium state for $f$ and $\phi$, and $\widetilde\phi$ be as defined in (\ref{eqDefPhiTilde}). Suppose that $q_0\in S^2$ and $\{q_i\}_{i\in\N_0}$ is the random walk determined by $Q=\RR_{\widetilde\phi}$ described above. Then almost surely,
$$
\frac{1}{n}\sum\limits_{i=0}^{n-1} \delta_{q_i}  \stackrel{w^*}{\longrightarrow} \mu_\phi \text{ as } n\longrightarrow +\infty.
$$
\end{theorem}

\begin{proof}
By Corollary~\ref{corMandMuUnique}, the equilibrium state $\mu_\phi$ is the unique Borel probability measure on $S^2$ that satisfies $Q^*(\mu_\phi)=\mu_\phi$. Then the theorem follows directly from a theorem of H.~Furstenberg and Y.~Kifer in \cite{FK83} formulated as Theorem~7.2 in \cite{Li13}.
\end{proof}

\end{document}